\documentclass[11pt]{article}
    \usepackage{url}
    \usepackage{verbatim}
    \usepackage[titletoc]{appendix}
    \usepackage{graphicx}
	\usepackage{amsmath}
	\usepackage{amsthm}
	\usepackage{amsfonts}
	\usepackage{graphicx}
    \usepackage{nicefrac}
    \usepackage{longtable}
    \usepackage{color}
    \usepackage{graphicx, amssymb, subcaption,graphics}

    \newcommand{\be}{\begin{equation}}
    \newcommand{\ee}{\end{equation}}

    \newcommand{\nrm}[1]{\left\| #1 \right\|}

    \newcommand\dt {{\Delta t}}

    \def\x{\mbox{\boldmath $x$}}

	\newtheorem{thm}{Theorem}[section]
	\newtheorem{prop}[thm]{Proposition}
	\newtheorem{cor}[thm]{Corollary}
	
	\newtheorem{lem}[thm]{Lemma}
	\newtheorem{rem}[thm]{Remark}
	\newtheorem{defi}[thm]{Definition}
	
\begin{document}
	
	\title{A second order accurate scalar auxiliary variable (SAV) numerical method for the square phase field crystal equation}

	\author{
Min Wang \thanks{Faculty of Sciences, Beijing University of Technology, Beijing 100124, P. R. China (w\_min@emails.bjut.edu.cn)}
\and
Qiumei Huang \thanks{Faculty of Sciences, Beijing University of Technology, Beijing 100124, P. R. China (Corresponding Author: qmhuang@bjut.edu.cn)}
\and		
Cheng Wang\thanks{Department of Mathematics, University of Massachusetts, North Dartmouth, MA  02747, USA (cwang1@umassd.edu)}
}

	\maketitle
	\numberwithin{equation}{section}

	\begin{abstract}
In this paper we propose and analyze a second order accurate (in time) numerical scheme  for the square phase field crystal (SPFC) equation, a gradient flow modeling crystal dynamics at the atomic scale in space but on diffusive scales in time. Its primary difference with the standard phase field crystal model is an introduction of the 4-Laplacian term in the free energy potential, which in turn leads to a much higher degree of nonlinearity. To make the numerical scheme linear while preserving the nonlinear energy stability, we make use of the scalar auxiliary variable (SAV) approach, in which a second order backward differentiation formula (BDF) is applied in the temporal stencil. Meanwhile, a direct application of the SAV method faces certain difficulties, due to the involvement of the 4-Laplacian term, combined with a derivation of the lower bound of the nonlinear energy functional. In the proposed numerical method, an appropriate decomposition for the physical energy functional is formulated, so that the nonlinear energy part has a well-established global lower bound, and the rest terms lead to constant-coefficient diffusion terms with positive eigenvalues. In turn, the numerical scheme could be very efficiently implemented by constant-coefficient Poisson-like type solvers (via FFT), and energy stability is established by introducing an auxiliary variable, and an optimal rate convergence analysis is provided for the proposed SAV method. A few numerical experiments are also presented, which confirm the efficiency and accuracy of the proposed scheme.
	\end{abstract}
	
	
\medskip
	


	\section{Introduction}
	
The phase field crystal (PFC) equation, originally proposed in \cite{elder02}, stands for a new model to simulating crystal dynamics at the atomic scale in space but on diffusive scales in time. This model naturally incorporates elastic and plastic deformations, multiple crystal orientations and defects, and it has already been used to simulate a wide variety of microstructures, such as epitaxial thin film growth~\cite{elder04}, grain growth~\cite{stefanovic06}, eutectic solidification~\cite{elder07}, and dislocation formation and motion~\cite{stefanovic06}, etc.  Also see a related review~\cite{provatas07}. In more details, the phase variable describes a coarse-grained temporal average of the number density of atoms, which is related to dynamic density functional theory \cite{backofen07, marconi99}.  A significant advantage of this approach has been observed over other atomistic methods, such as molecular dynamics methods where the time steps are constrained by atomic-vibration time scales. In the PFC approach, the dimensionless energy is given by the following form~\cite{elder02, elder04, swift77}
	\begin{equation}
E_{\rm pfc}(\phi) = \int_\Omega \left\{ \frac{1}{4}\phi^4 +\frac{1-\varepsilon}{2}\phi^2-\left|\nabla\phi\right|^2 +\frac{1}{2} (\Delta\phi)^2 \right\} d{\bf x}  ,
  \quad  \varepsilon > 0 ,
	\label{energy-PFC}
	\end{equation}
where $\Omega\subset \mathbb{R}^D$, $D = 2$ or 3, $\phi:\Omega\rightarrow \mathbb{R}$ is the atom density field, and the parameter $\varepsilon$ represents a deviation from the melting temperature with $0< \varepsilon < 1$. 
For simplicity, a periodic boundary condition is imposed for $\phi$; the analysis for the homogeneous Neumann boundary condition case could be similarly extended. 
In turn, the standard PFC equation becomes the associated $H^{-1}$ gradient flow:
	\[
\partial_t \phi = \Delta\mu, \quad \mu := \delta_\phi E_{\rm pfc} = \phi^3 + a \phi + 2 \Delta \phi + \Delta^2 \phi ,  \quad  a = 1- \varepsilon .
	\]
For $\varepsilon>0$, spatial oscillations could be observed in the solution of the PFC equation; typically in 2D, the peaks and valleys of $\phi$ are arranged in a hexagonal pattern. These solutions represent ``solid phase" solutions in the model. Meanwhile, ``liquid phase" solutions, which are spatially uniform and constant, may also be possible. In fact, these solutions even be in coexistence with the solid phase solutions to describe a crystal in equilibrium with its melt; see the related discussions in~\cite{provatas10}.

On the other hand, alternate lattice structures, such as ``square" symmetry crystal lattices, are possible in 2D solutions. As mentioned in \cite{elder04, golovin03}, a different choice of nonlinear term in the PFC model is needed to obtain a square symmetry crystal lattice rather than the usual hexagonal structure. In particular, such a symmetry can be obtained~\cite{golovin03} by replacing $\phi^4$ in (\ref{energy-PFC}) with $\left| \nabla \phi \right|^4$; also see~\cite{wu10} for a related method. This results in the following energy functional
	\begin{align}
E_{\rm spfc}(\phi) & = \int_\Omega\left\{ \frac{a}{2}\phi^2 +\frac{1}{4} \left| \nabla \phi \right|^4  -\left|\nabla\phi\right|^2 + \frac12 (\Delta\phi)^2 \right\} d{\bf x} .
	\label{energy-SPFC}
	\end{align}
In fact, there are essential similarities between this energy and the Aviles-Giga-type energy~\cite{aviles96}.  The square phase field crystal (SPFC) equation is given by the following dynamics
	\begin{equation}
	\partial_t \phi =  \Delta \mu \ ,   \quad  \mu := \delta_\phi E_{\rm spfc} =  - \nabla \cdot \left(  | \nabla \phi |^2  \nabla \phi \right) + a \phi + 2 \Delta \phi + \Delta^2 \phi .
  	\label{equation-SPFC}
	\end{equation}
We will  assume for simplicity that $a=1-\varepsilon>0$. 
For the standard PFC model and its modified version, there have been extensive numerical works~\cite{baskaran13a, baskaran13b, dong18, hu09, wang10c, wang11a, wise09, zhang13}, etc.  
In terms of the nonlinearity, the only difference between the standard PFC and SPFC equations is the replacement of $\phi^4$ by $| \nabla \phi |^4$ in the free energy functional, while the analysis and numerical approximation of the later one are much more challenging, especially when using pseudo-spectral approximations of spatial derivatives. Very limited numerical results have been available for the SPFC equation in the existing literature. For instance, some simulation results are reported for a closely related equation in~\cite{golovin03}. A modified backward differentiation formula (BDF) scheme was presented in a more recent work~\cite{cheng2019d}, in which the energy stability (in the original phase variable) and the convergence analysis have been theoretically justified.

Meanwhile, most existing works of energy stable schemes for a gradient flow containing $|\nabla \phi|^4$ energy potential are based on an implicit treatment of the 4-Laplacian part; see the related works~\cite{cheng2019d, fengW18a, feng2017preconditioned, fengW17c, shen12, wang10c}, etc. In particular, the preconditioned steepest descent (PSD) nonlinear iteration has been proposed in~\cite{feng2017preconditioned} for the 4-Laplacian solver in both the $L^2$ and $H^{-1}$ gradient flow, due to its convex structure, so that the computational cost is decomposed of certain Poisson-like solvers at each iteration stage. Extensive numerical experiments have implied that, approximately 10 to 15 iteration stages are needed for such a PSD algorithm in most practical numerical simulations of physical examples. As a result, the computational cost of implicit nonlinear 4-Laplacian solvers is approximately 10 to 15 times of a linear scheme for the corresponding physical system .

On the other hand, a theoretical justification of linear schemes for the gradient flows containing 4-Laplacian energy potential turns out to be a challenging issue. The scalar auxiliary variable (SAV) approach for various gradient flows has attracted more and more attentions in recent years~\cite{ChengQ2018, shen18b, shen18a, shen19}. To overcome the difficulty associated with the nonlinearity, the energy functional is split into two parts: a nonlinear energy functional with a uniform lower bound, combined with a quadratic surface diffusion energy with constant-coefficients. In turn, the elevated nonlinear energy part (which contains a global constant to make its value positive) is rewritten as a quadratic term, not in terms of the original physical variable, but in terms of an artificially-introduced auxiliary variable. As a result, linear schemes could be derived for the gradient flow reformulated in the quadratic nonlinear energy and the surface diffusion energy, so that both the unique solvability and modified energy stability could be theoretically justified for the linear schemes. Also notice that such an energy estimate is in terms of the reformulated energy functional, not in terms of the original energy functional.

However, a direct application of the SAV method to the SPFC equation faces certain technical difficulties. It is observed that, the concave diffusion energy $- \| \nabla \phi \|^2$ corresponds to a linear part in the chemical potential, while such a functional does not have a global lower bound. In addition, its combination with two quadratic convex energy parts, namely, $\frac{a}{2} \| \phi \|^2$ and $\frac12 \| \Delta \phi \|^2$, does not have a global lower bound, either. As a result, if the concave diffusion energy is placed into the linear diffusion energy part, the SAV method would not be effectively derived. In this article, we come up with an alternate split, which places the  concave diffusion energy $- \| \nabla \phi \|^2$ into the nonlinear energy functional part. In additional, a combination of the 4-Laplacian energy $\frac14 \| \nabla \phi \|_{L^4}^4$ and the concave energy $- \| \nabla \phi \|^2$ has a well-established global lower bound, $- | \Omega |$, so that the nonlinear energy part is well-defined, and the linear surface diffusion energy only contains two terms with positive eigenvalues. Based on such an energy split, the PDE system is reformulated, and the SAV scheme could be derived via the second order BDF2 temporal discretization. Similar to the epitaxial thin film growth and other related gradient flow models, an explicit extrapolation is applied to obtain a second order approximation to the nonlinear chemical potential and nonlinear energy functional value. The resulting numerical system could be very efficiently solved; only a few Poisson-like solvers, via the FFT-based algorithms, are needed at each time step, since only constant-coefficient equations are involved in the numerical scheme.

An unconditional energy stability could be proved via a careful estimate. Again, such a stability estimate is in terms of the reformulated energy functional, not in terms of the original energy functional. In the spatial discretization, we use Fourier pseudo-spectral approximation for its ability to capture more detailed structures with a reduced computational cost. Summation-by-parts formulas enable us to derive unique solvability and energy stability for the fully discrete numerical scheme. As a result of this discrete energy stability, a uniform-in-time discrete $H^2$ bound for the numerical solution becomes available. In addition to this uniform $H^2$ bound for the numerical solution (of the phase variable), a higher order $H^3 $ estimate could also be derived, with the help of various discrete Sobolev inequality in the Fourier pseudo-spectral space. With such an $H^3$ bound at hand, we are able to control a discrete gradient of the nonlinear chemical potential error function, in the Fourier pseudo-spectral space. 
In addition, one nonlinear error inner product could be cancelled between the error evolutionary equations for the original phase variable and the one for the introduced auxiliary variable. These preliminary estimates enable one to obtain an optimal rate ($O (\dt^2 + h^m)$) convergence analysis for the proposed numerical scheme in the energy norm, i.e., in the $\ell^\infty (0,T; H_N^2) \cap \ell^2 (0,T; H_N^5)$ norm. In particular, the aliasing error control techniques have to be applied in the nonlinear error estimate associated with the 4-Laplacian term.

 The outline of the paper is given as follows. In Section~\ref{sec:numerical scheme} we present the numerical scheme. First we review the Fourier pseudo-spectral approximation in space and recall an aliasing error control technique. Then we formulate the proposed numerical scheme, and prove its unique solvability. Subsequently, the energy stability analyses is provided in Section~\ref{sec:stability}, and an optimal rate convergence analysis is established in Section~\ref{sec:convergence}. Some numerical results are presented in Section~\ref{sec:numerical results}.  Finally, some concluding remarks are made in Section~\ref{sec:conclusion}.

	\section{The numerical scheme}
	\label{sec:numerical scheme}
	
\subsection{Review of Fourier pseudo-spectral approximations}

The Fourier pseudo-spectral method is also referred as the Fourier collocation spectral method. It is closely related to the Fourier spectral method, but complements the basis by an additional pseudo-spectral basis, which allows to represent functions on a quadrature grid. This simplifies the evaluation of certain operators, and can considerably speed up the calculation when using fast algorithms such as the fast Fourier transform (FFT); see the related descriptions in~\cite{Boyd2001, chen20a, chen20b, chen20c, cheng2015fourier, cheng2019c, cheng16b, cheng2020a, cheng16a, gottlieb12a, gottlieb12b, HGG2007, zhangC18a, zhangC17a}, etc.

To simplify the notation in our pseudo-spectral analysis, we assume that the domain is given by $\Omega = (0,1)^3$, $N_x = N_y = N_z =: N\in\mathbb{N}$ and $N \cdot h = 1$. We further assume that $N$ is odd:
	\[
N = 2K+1, \quad \mbox{for some} \ K\in\mathbb{N}.
	\]
The analyses for more general cases are a bit more tedious, but can be carried out without essential difficulty. The spatial variables are evaluated on the standard 3D numerical grid $\Omega_N$, which is defined by grid points $(x_i, y_j, z_k)$, with $x_i = i h$, $y_j=jh$, $z_k = k h$, $0 \le i , j, k \le 2K +1$. This description for three-dimensional mesh ($d=3$) can here and elsewhere be trivially modified for the two-dimensional case ($d=2$).

We define the grid function space
	\begin{equation}
\mathcal{G}_N := \left\{ f:\mathbb{Z}^3 \to \mathbb{R} \ \middle| \ f \ \mbox{is $\Omega_N$-periodic} \right\} .
	\end{equation}
Given any periodic grid functions $f,g\in\mathcal{G}_N$, the $\ell^2$ inner product and norm are defined as
	\begin{eqnarray}
 \left\langle f , g \right\rangle  := h^3 \sum_{i,j,k=0}^{N -1}   f_{i,j,k}\cdot g_{i,j,k} , \quad \left\| f \right\|_2 := \sqrt{ \left\langle f , f \right\rangle } .
	\label{spectral-coll-inner product-1}
	\end{eqnarray}
The zero-mean grid function subspace is denoted $\mathring{\mathcal{G}}_N := \left\{ f\in \mathcal{G}_N \ \middle| \  \langle f, 1\rangle =:  \overline{f}  = 0\right\}$.  For  $f\in \mathcal{G}_N$, we have the discrete Fourier expansion
	\begin{equation}
f_{i,j,k} = \sum_{\ell,m,n=-K}^{K} \hat{f}_{\ell,m,n}^N \exp \left( 2 \pi {\rm i} ( \ell x_i + m y_j + n z_k ) \right) ,
	\label{spectral-coll-1}
	\end{equation}
where the discrete Fourier coefficients are given by
	\begin{equation}
\hat{f}_{\ell,m,n}^N := h^3\sum_{i,j,k = 0}^{N-1} f_{i,j,k}\exp\left(-2\pi {\rm i} \left(\ell x_i + m x_j + n z_k \right)\right) .
	\end{equation}
The collocation Fourier spectral first and second order \emph{derivatives} of $f$ are defined as
	\begin{eqnarray}
{\cal D}_x f_{i,j,k} := \sum_{\ell,m,n = -K}^{K}  \left( 2 \pi {\rm i} \ell \right) \hat{f}_{\ell,m,n}^N  \exp \left( 2 \pi {\rm i} ( \ell x_i + m y_j + n z_k ) \right) ,
	\label{spectral-coll-2-1}
	\\
{\cal D}_x^2 f_{i,j,k} := \sum_{\ell,m,n = -K}^{K}   \left( - 4 \pi^2 \ell^2 \right) \hat{f}_{\ell,m,n}^N \exp \left( 2 \pi {\rm i} ( \ell x_i + m y_j + n z_k) \right) .
	\label{spectral-coll-2-3}
	\end{eqnarray}
The differentiation operators in the $y$ and $z$ directions, ${\cal D}_y$, ${\cal D}_y^2$, ${\cal D}_z$ and ${\cal D}_z^2$ can be defined in the same fashion. In turn, the discrete Laplacian, gradient and divergence operators are given by
	\begin{equation}
\Delta_N f :=  \left( {\cal D}_x^2  + {\cal D}_y^2 + {\cal D}_z^2 \right) f , \quad \nabla_N f := \left(
	\begin{array}{c}
{\cal D}_x f
	\\
{\cal D}_y f
	\\
{\cal D}_z f
	\end{array}
\right),  \quad  \nabla_N \cdot \left(
	\begin{array}{c}
f _1
	\\
f _2
	\\
f_3
	\end{array}
\right)  := {\cal D}_x f_1 + {\cal D}_y f_2 + {\cal D}_z f_3 ,
	\label{spectral-coll-3}
	\end{equation}
at the point-wise level. It is straightforward to verify that
	\begin{equation}
\nabla_N \cdot \nabla_N f = \Delta_N f .
	\label{spectral-coll-4-a}
	\end{equation}
See the derivations in the related references~\cite{Boyd2001, canuto82,
Gottlieb1977}.

	\begin{defi}
Suppose that the grid function $f\in\mathcal{G}_N$ has the discrete Fourier expansion (\ref{spectral-coll-1}).  Its spectral extension into the trigonometric polynomial space $\mathcal{P}_K$ (the space of trigonometric polynomials of degree at most $K$) is defined as
	\begin{equation}
f_S (x,y,z)  = \sum_{\ell,m,n=-K}^{K}   \hat{f}_{\ell,m,n}^N \exp \left( 2 \pi {\rm i} ( \ell x + m y + n z) \right) .
	\label{spectral-coll-projection-2}
	\end{equation}
We write $S_N(f) = f_S$ and denote $S_N:\mathcal{G}_N \to \mathcal{P}_K$ the spectral interpolation operator. Suppose $g\in C_{\rm per}(\Omega,\mathbb{R})$. We define the grid projection $Q_N: C_{\rm per}(\Omega,\mathbb{R})\to\mathcal{G}_N$ via
	\begin{equation}
Q_N(g)_{i,j,k} := g(x_i,y_j,z_k),
	\end{equation}
The resultant grid function may, of course, be expressed as a discrete Fourier expansion:
	\[
Q_N(g)_{i,j,k} = \sum_{\ell,m,n=-K}^{K}   \widehat{Q_N(g)}_{\ell,m,n}^N \exp \left( 2 \pi {\rm i} ( \ell x_i + m y_j + n z_k)\right) .
	\]
We define the de-aliasing operator $R_N :C_{\rm per}(\Omega,\mathbb{R}) \to \mathcal{P}_K$ via $R_N := S_N(Q_N)$. In other words,
	\begin{eqnarray}
R_N(g)(x,y,z)  = \sum_{\ell,m,n=-K}^{K}   \widehat{Q_N(g)}_{\ell,m,n}^N \exp \left( 2 \pi {\rm i} ( \ell x + m y + n z)\right) .
    \label{spectral-coll-projection-3}
	\end{eqnarray}
Finally, for any $g\in L^2(\Omega,\mathbb{R})$, we define the (standard) Fourier projection operator $P_N:L^2(\Omega,\mathbb{R}) \to {\mathcal P}_K$ via
	\[
P_N(g) (x,y,z) = \sum_{\ell,m,n = -K}^K \hat{g}_{\ell,m,n} \exp \left( 2 \pi {\rm i} ( \ell x + m y + n z)\right),
	\]
where
	\[
\hat{g}_{\ell,m,n} = \int_\Omega g(x,y,z) \exp\left(- 2\pi{\rm i} \left(\ell x+m y+n z \right)\right)  d\x,
	\]
are the (standard) Fourier coefficients.
	\end{defi}
	

To overcome a key difficulty associated with the $H^m$ bound of the nonlinear term  obtained by collocation interpolation, the following lemma is introduced. The case of $r=0$ was proven in earlier works~\cite{E92, E93}, and the case of $r \ge 1$ was analyzed in a recent article~\cite{gottlieb12b}.

	\begin{lem}
	\label{lemma:aliasing error}
Suppose that $m$ and $K$ are non-negative integers, and,  as before, assume that $N = 2K+1$. For any $\varphi \in {\cal P}_{mK}$ in $\mathbb{R}^d$, we have the estimate
	\begin{equation}
\left\| R_N (\varphi) \right\|_{H^r}  \le  m^{\frac{d}{2}}  \left\|  \varphi \right\|_{H^r},
	\label{spectral-coll-projection-4}
	\end{equation}
for any non-negative integer $r$.
	\end{lem}

In addition, we introduce the discrete  fractional operator $(-\Delta_N)^\gamma$ (with $\gamma >0$):
	\begin{equation}
(-\Delta_N)^\gamma f_{i,j,k} := \sum_{\ell,m,n=-K}^{K} \lambda_{\ell,m,n}^\gamma \hat{f}_{\ell,m,n}^N  \exp \left( 2 \pi {\rm i} ( \ell x_i + m y_j + n z_k) \right) ,
 \, \, \lambda_{\ell, m, n} = 4 \pi^2 ( \ell^2 + m^2 + n^2 )  .
	\label{spectral-coll-4-b}
	\end{equation}
for a grid function $f$ with the discrete Fourier expansion as (\ref{spectral-coll-1}). Similarly, for a grid function $f\in\mathring{\mathcal{G}}_N$ of  (discrete) mean zero, a discrete version of the operator $(-\Delta)^{-\gamma}$ may be defined as
	\begin{equation}
(-\Delta_N)^{-\gamma} f_{i,j,k} := \sum_{\substack{\ell,m,n =-K \\
(\ell,m,n) \ne {\bf 0}}}^K  \lambda_{\ell,m,n}^{-\gamma} \hat{f}_{\ell,m,n}^N \exp \left( 2 \pi {\rm i} ( \ell x_i + m y_j + n z_k) \right).
	\label{spectral-coll-5}
	\end{equation}
We notice that the right hand side of~\eqref{spectral-coll-5} is a periodic grid function of zero mean, \emph{i.e}, $(-\Delta_N)^{-\gamma} f\in\mathring{\mathcal{G}}_N$. Furthermore, to facilitate the analysis in later sections, we introduce an operator $L_N$ as $L_N f := ( a I + \Delta_N^2 ) f$, for any  $f \in \mathcal{G}_N$. The following fractional operator is similarly defined:
	\begin{equation}
 L_N^\frac12 f_{i,j,k} := \sum_{\ell,m,n =-K}^K  \Big( a + \lambda_{\ell,m,n}^2 \Big)^\frac12 \hat{f}_{\ell,m,n}^N \exp \left( 2 \pi {\rm i} ( \ell x_i + m y_j + n z_k) \right) ,
	\label{spectral-coll-6}
	\end{equation}
based on the fact that, the Fourier eigenvalue of the operator $L_N$ (for the frequency mode $(\ell, m, n)$) is given by $a + \lambda_{\ell,m,n}^2$.

The following summation-by-parts formulas are valid (see the related discussions in~\cite{chen12, chen14,   gottlieb12a, gottlieb12b}): for any periodic grid functions $f,g\in\mathcal{G}_N$,
	\begin{equation}
\left\langle f ,  \Delta_N  g  \right\rangle  = - \left\langle \nabla_N f ,  \nabla_N g   \right\rangle  ,    \quad \left\langle f ,  \Delta_N^2  g  \right\rangle =  \left\langle \Delta_N f ,  \Delta_N g   \right\rangle ,   \quad \left\langle f ,  \Delta_N^3  g  \right\rangle =  - \left\langle \nabla_N \Delta_N f ,  \nabla_N \Delta_N g   \right\rangle .
	\label{spectral-coll-inner product-3}
	\end{equation}
Similarly, the following identity could be derived in the same manner:
\begin{eqnarray}
  \langle f, L_N g \rangle = \langle L_N^\frac12 f , L_N^\frac12 g \rangle , \quad
  \forall  f,g\in\mathcal{G}_N .
  \label{spectral-coll-inner product-4}
\end{eqnarray}

Since the SPFC equation (\ref{equation-SPFC}) is an $H^{-1}$ gradient flow, we need  a discrete version of the norm $\| \cdot \|_{H^{-1}}$ defined on $\mathring{\mathcal G}_N$. For any $f, g \in \mathring{\mathcal G}_N$, we define
\begin{eqnarray}
  \langle f,  g \rangle_{-1,N} := \left\langle f ,  ( - \Delta_N )^{-1} g  \right\rangle
  = \left\langle ( - \Delta_N )^{-\frac12}  f , ( - \Delta_N )^{-\frac12} g   \right\rangle,
  \label{spectral-coll-inner product-6}
\end{eqnarray}
so that the $\| \cdot \|_{-1,N}$ norm could be introduced as
	\begin{equation}
\| f \|_{-1,N} := \sqrt{ \langle f , f \rangle_{-1,N} } = \|  ( - \Delta_N )^{-\frac12}  f  \|_2 .
	\label{spectral-coll-inner product-5}
	\end{equation}


In addition to the standard $\ell^2$ norm, we also introduce the $\ell^p$, $1\le p <\infty$, and $\ell^\infty$ norms for a grid function $f\in\mathcal{G}_N$:
\begin{equation}
 \nrm{f}_{\infty} := \max_{i,j,k} |f_{i,j,k}| ,   \qquad
 \nrm{f}_{p}  := \Bigl( h^3\sum_{i,j,k=0}^{N-1} |f_{i,j,k} |^p \Bigr)^{\frac{1}{p}} , \quad 1\leq p < \infty.  \label{spectral-defi-Lp}
\end{equation}
The discrete $H^1$ and $H^2$ norms are introduced as
\begin{equation}
  \| f \|_{H_N^1}^2 = \| f \|_2^2 + \| \nabla_N f \|_2^2 ,  \quad
  \| f \|_{H_N^2}^2 = \| f \|_{H_N^1}^2 + \| \Delta_N f \|_2^2 .
   \label{spectral-defi-Hm}
\end{equation}
For any periodic grid function $\phi\in\mathcal{G}_N$, the discrete SPFC energy is defined as
	\be
E_N (\phi) := \frac14 \| \nabla_N \phi \|_4^4 + \frac{a}{2} \| \phi \|_2^2 - \| \nabla_N \phi \|_2^2 + \frac12 \nrm{\Delta_N \phi}_2^2 .
	\label{energy-discrete-spectral}
	\ee
	

The following result corresponds to a discrete Sobolev embedding from $H_N^2$ to $W_N^{1,6}$ in the pseudo-spectral space. Similar discrete embedding estimates, in the lower order ones, could be found in Lemma 2.1 of~\cite{cheng16a}; also see the related results \cite{feng2017preconditioned,  fengW17c} in the finite difference version. A direct calculation is not able to derive these inequalities; instead, a discrete Fourier analysis has to be applied in the derivation; the details of the proof has been provided in a recent work~\cite{cheng2019d}. .

	\begin{prop} \cite{cheng2019d}
\label{prop:embedding}
  For any periodic grid function $f$, we have
\begin{eqnarray}
  \| \nabla_N f \|_6 \le C \| \Delta_N f \|_2 ,  \quad
  \mbox{for some constant $C$ only dependent on $\Omega$} .
  \label{embedding-0}
\end{eqnarray}
	\end{prop}
	
The following discrete elliptic regularity estimate will be used in the later stability analysis; its proof will be provided in Appendix~\ref{proof:Prop 2.4}.

	\begin{prop}
\label{prop:elliptic regularity}
For any periodic grid function $f$, we have
\begin{eqnarray}
    \| \nabla_N \Delta_N f \|_2  \le \hat{C}_0 \| \Delta_N^3 f \|_2 ,  \quad
  \mbox{for some $\hat{C}_0$ only dependent on $\Omega$} .
  \label{elliptic regularity-0}
\end{eqnarray}
	\end{prop}

\subsection{The fully discrete numerical scheme}

The SPFC energy~\eqref{energy-SPFC} is decomposed into two parts:
\begin{eqnarray}
  E_{\rm spfc} (\phi) = E_1 (\phi) + \frac12 ( \phi , L \phi ) ,  \quad
  E_1 (\phi) = \int_\Omega\left\{ \frac14 | \nabla \phi |^4  - | \nabla \phi |^2 + 2 \right\} d{\bf x} ,  \, \, \, L \phi = a \phi + \Delta^2 \phi .  \label{energy-SPFC-split-1}
\end{eqnarray}
In particular, due to the point-wise quadratic inequality
\begin{eqnarray}
    \frac14 | \nabla \phi |^4  - | \nabla \phi |^2 + 1 \ge 0 ,   \label{energy-SPFC-split-2-1}
\end{eqnarray}
we conclude that $E_1 (\phi)$ have a well-established lower bound:
\begin{eqnarray}
  E_1 (\phi) \ge | \Omega | . \label{energy-SPFC-split-2-2}
\end{eqnarray}
In turn, the nonlinear chemical potential becomes
\begin{eqnarray}
  N (\phi) := \delta_\phi \delta E_1
  = - \nabla \cdot ( | \nabla \phi |^2 \nabla \phi )  + 2 \Delta \phi .
  \label{nonlinear chem pot-1}
\end{eqnarray}

Therefore, with an introduction of a scalar auxiliary variable
\begin{equation}
  r := \sqrt{E_1 (\phi)} ,  \label{SAV-0}
\end{equation}
the original SPFC equation~\eqref{equation-SPFC} could be rewritten as the following system:
 \begin{eqnarray}\label{SAV system-1}
 \begin{cases}
 \phi_t = \Delta \Big( \frac{r}{\sqrt{E_{1}(\phi)}} N(\phi) + L \phi \Big) , \\
 r_t  = \frac{1}{2\sqrt{E_{1}(\phi)}} \int_\Omega \, N(\phi) \phi_t \, d {\bf x}  .
 \end{cases}
\end{eqnarray}
Based on this reformulation, the fully discrete second order SAV scheme is proposed as follows, with Fourier pseudo-spectral spatial approximation:
\begin{eqnarray} \label{scheme-SAV-1}
\begin{cases}
\frac{\frac32 \phi^{n+1} - 2 \phi^n +\frac12 \phi^{n-1}}{\dt} = \Delta_N \Big( \frac{r^{n+1}}{\sqrt{E_{1,N}(\hat{\phi}^{n+1})}} N_N (\hat{\phi}^{n+1}) + L_N \phi^{n+1} \Big), \ \ \ \ \ \ ({\theequation a}) \\
\frac{\frac32 r^{n+1} - 2 r^n + \frac12 r^{n-1}}{\dt} = \frac{1}{2\sqrt{E_{1,N}(\hat{\phi}^{n+1})}} \langle N_N ( \hat{\phi}^{n+1}) , \frac{\frac32 \phi^{n+1} - 2 \phi^n + \frac12 \phi^{n-1}}{\dt} \rangle,  \ \ \ \ \ \ ({\theequation b})
\end{cases}
\end{eqnarray}
in which $N_N (\phi) := -\nabla_N \cdot ( | \nabla_N \phi |^2 \nabla_N \phi) + 2 \Delta_N \phi$, $L_N \phi = a \phi + \Delta_N^2 \phi$, and a second order explicit extrapolation is applied to obtain $\hat{\phi}^{n+1} = 2\phi^{n}-\phi^{n-1}$. The discrete nonlinear energy functional is introduced as $E_{1,N} (\phi) := \frac14 \| \nabla_N \phi \|_4^4  - \| \nabla_N \phi \|_2^2 + 2 | \Omega |$, similar to the notation in~\eqref{energy-discrete-spectral}.

Since~\eqref{scheme-SAV-1} is a two-step numerical method, a ``ghost" point extrapolation for $\phi^{-1}$ is useful. To preserve the second order accuracy in time, we apply the following approximation:
	\begin{equation}
  \phi^{-1} = \phi^0 - \dt \Delta_N \mu^0, \quad  \mu^0 :=
  - \nabla_N \cdot ( | \nabla_N \phi^0 |^2 \nabla_N \phi^0 )
   + a \phi^0 + 2 \Delta_N \phi^0 + \Delta_N^2 \phi^0 .  \label{scheme-SAV-initial-1}
	\end{equation}
A careful Taylor expansion indicates an $O (\dt^2 + h^m)$ accuracy for such an approximation:
\begin{eqnarray}
  \| \phi^{-1} - \Phi^{-1} \|_2 \le C (\dt^2 + h^m) ,  \quad
  \mbox{in which $\Phi$ is the exact solution for~\eqref{equation-SPFC} }.
   \label{scheme-SAV-initial-2}
\end{eqnarray}
In turn, we take $r^0 := \sqrt{E_{1,N} (\phi^0)}$, $r^{-1} := \sqrt{E_{1,N} (\phi^{-1}) }$

\subsection{Unique solvability and efficient numerical solver for the proposed scheme}

In this section we analyze the unique solvability of the proposed SAV scheme~\eqref{scheme-SAV-1}. From (\ref{scheme-SAV-1}a), one can get
\begin{eqnarray}  \label{solver-1}
\Big( \frac32 I - \dt \Delta_N L_N \Big) \phi^{n+1} = \dt \Delta_N \Big( \frac{r^{n+1}}{\sqrt{E_{1,N}(\hat{\phi}^{n+1})}} N_N (\hat{\phi}^{n+1}) \Big) + 2 \phi^n - \frac12 \phi^{n-1} .
\end{eqnarray}
Define $A_N =\frac32 I - \dt \Delta_N L_N$, so that the following identity is valid:
$$
  \phi^{n+1} = \dt \frac{r^{n+1}}{\sqrt{E_{1,N} (\hat{\phi}^{n+1}) }}
  A_N^{-1} \Delta_N N_N (\hat{\phi}^{n+1}) + A_N^{-1} ( 2 \phi^n - \frac12 \phi^{n-1} ) .
$$
From (\ref{scheme-SAV-1}b), we see that
\begin{eqnarray} \label{solver-2}
r^{n+1} = \frac43 r^n - \frac13 r^{n-1} + \frac{1}{3 \sqrt{E_{1,N} (\hat{\phi}^{n+1})} } \langle N_N (\hat{\phi}^{n+1}) , \frac32 \phi^{n+1} -2 \phi^n + \frac12 \phi^{n-1} \rangle .
\end{eqnarray}
A substitution of (\ref{solver-2}) into (\ref{solver-1}) gives
\begin{eqnarray*}
\nonumber&& \Big( \frac32 I - \dt \Delta_N L_N \Big) \phi^{n+1}
- \frac{\Delta_N N_N (\hat{\phi}^{n+1})}{2 E_{1,N}(\hat{\phi}^{n+1})} \dt \langle N_N (\hat{\phi}^{n+1}) , \phi^{n+1} \rangle \\
&=&\frac{\dt \Delta_N N_N (\hat{\phi}^{n+1})}{\sqrt{E_{1,N} (\hat{\phi}^{n+1})}}
 \Big(\frac43 r^n - \frac13 r^{n-1} + \frac{1}{3 \sqrt{E_{1,N}(\hat{\phi}^{n+1})}} \langle N_N (\hat{\phi}^{n+1}) , -2 \phi^n+ \frac12 \phi^{n-1} \rangle \Big)
 + 2 \phi^n - \frac12 \phi^{n-1} .
\end{eqnarray*}
Let $g_N^n$ denotes the right-hand of the above equation, then it becomes
\begin{eqnarray*}
A_N \phi^{n+1} - \frac{\Delta_N N_N (\hat{\phi}^{n+1})}{2 E_{1,N}(\hat{\phi}^{n+1}) }
 \dt  \langle N_N (\hat{\phi}^{n+1}) , \phi^{n+1} \rangle = g_N^n .
\end{eqnarray*}
Multiplying both sides by $A_N^{-1}$ implies that
 \begin{eqnarray}\label{solver-3}
 \phi^{n+1} - \frac{1}{2 E_{1,N} (\hat{\phi}^{n+1})} \dt
 \langle N_N (\hat{\phi}^{n+1}) , \phi^{n+1} \rangle
   \cdot A_N^{-1} \Delta_N  N_N (\hat{\phi}^{n+1}) = A_N^{-1} g_N^n .
 \end{eqnarray}
 Denote $LHS =  \langle N_N (\hat{\phi}^{n+1}) , \phi^{n+1} \rangle$, a scalar value.  Taking a discrete inner product with~\eqref{solver-3} by $N_N (\hat{\phi}^{n+1})$ leads to
$$
   \langle N_N (\hat{\phi}^{n+1}) , \phi^{n+1} \rangle
 - \frac{\dt}{2 E_{1,N} (\hat{\phi}^{n+1})} \cdot LHS \cdot
  \langle N_N (\hat{\phi}^{n+1}) , A_N^{-1} \Delta _N N_N (\hat{\phi}^{n+1}) \rangle
  = \langle N_N (\hat{\phi}^{n+1}) , A_N^{-1} g_N^n \rangle .
$$
 Then we arrive at
 \begin{eqnarray} \label{solver-4}
 \Big( 1 - \frac{\dt}{2 E_{1,N} (\hat{\phi}^{n+1})} \cdot
  \langle N_N (\hat{\phi}^{n+1}) , A_N^{-1} \Delta_N N_N (\hat{\phi}^{n+1}) \rangle \Big)
  \cdot LHS = \langle N_N (\hat{\phi}^{n+1}) , A_N^{-1} g_N^n \rangle .
\end{eqnarray}
In addition, we notice that
\begin{equation}  \label{solver-5}
   \langle N_N (\hat{\phi}^{n+1}) , A_N^{-1} \Delta_N N_N (\hat{\phi}^{n+1}))
   \le 0 ,
\end{equation}
since all the eigenvalues of the symmetric operator $A_N^{-1} \Delta_N$ are non-positive.
As a direct consequence, the coefficient on the left hand side of~\eqref{solver-4} is positive, so that the value of $LHS$ is uniquely solvable. Going back~\eqref{solver-3}, the numerical solution $\phi^{n+1}$ is uniquely determined:
\begin{equation}
  \phi^{n+1} = \frac{\dt}{2 E_{1,N} (\hat{\phi}^{n+1})} \cdot LHS
  \cdot A_N^{-1} \Delta_N N_N (\hat{\phi}^{n+1}) + A_N^{-1} g_N^n .
  \label{solver-6}
\end{equation}
Furthermore, a substitution of $\phi^{n+1}$ into (\ref{solver-2}) gives the numerical value of $r^{n+1}$.

	\begin{thm}
	\label{SPFC solvability}
 Given $\phi^n, \phi^{n-1} \in \mathcal{G}_N$, two scalar values $r^n$, $r^{n-1}$, with $\overline{\phi^n} = \overline{\phi^{n-1}}$, there exists a unique solution $\phi^{n+1} \in \mathcal{G}_N$ for the numerical schemes~\eqref{scheme-SAV-1}. The  scheme is mass conservative, i.e., $\overline{\phi^k} \equiv \overline{\phi^0} := \beta_0$, for any $k \ge 0$,
 provided that $\overline{\phi^{-1}} = \overline{\phi^0} = \beta_0$.
	\end{thm}

\begin{proof}
  The unique solvability comes from the derived identities~\eqref{solver-2}, \eqref{solver-4} and~\eqref{solver-6}. In addition, the mass conservation property is a direct consequence of a summation of (\ref{scheme-SAV-1}a) over $\Omega$, which is turn leads to
\begin{eqnarray}
   \overline{\phi^{n+1}} =  \frac43 \overline{\phi^n} - \frac13 \overline{\phi^{n-1}}
   + \frac23 \overline{\Delta_N \Big( \frac{r^{n+1}}{\sqrt{E_{1,N}(\hat{\phi}^{n+1})}} N_N (\hat{\phi}^{n+1}) + L_N \phi^{n+1} \Big)}
   =  \frac43 \overline{\phi^n} - \frac13 \overline{\phi^{n-1}} ,
\end{eqnarray}
with the fact that $\overline{\Delta_N f}=0$, $\forall f \in  \mathcal{G}_N$, has been applied. An application of induction implies that $\overline{\phi^k} = \beta_0$, for any $k \ge 0$,
 provided that $\overline{\phi^{-1}} = \overline{\phi^0} = \beta_0$. This completes the proof of Theorem~\ref{SPFC solvability}.
\end{proof}

	\section{Unconditional energy stability and the uniform $H^3$ estimate}
	\label{sec:stability}

\subsection{Modified energy stability for the proposed numerical scheme}

	\begin{thm} \label{SPFC-energy stability}
For $k \ge 1$, define the discrete modified energy
	\begin{equation}
\mathcal{E}_{N} (\phi^{k+1}, \phi^k, r^{k+1}, r^k) :=
 \frac14 ( \| L_N^\frac12 \phi^{k+1} \|_2^2 + \| L_N^\frac12 ( 2  \phi^{k+1} - \phi^k ) \|_2^2 )
 + \frac12 ( | r^{k+1} |^2 + | 2 r^{k+1} - r^k |^2 ) .
	\label{discrete energy}
	\end{equation}
Solution of the numerical scheme~\eqref{scheme-SAV-1} satisfies the following dissipation properties
	\begin{equation}
\mathcal{E}_{N} ( \phi^{k+1}, \phi^k, r^{k+1}, r^k)
\le \mathcal{E}_{N} ( \phi^k, \phi^{k-1}, r^k, r^{k-1} ) .
	\label{SPFC-eng stab-est}
	\end{equation}
	\end{thm}
	

\begin{proof}
We begin with a rewritten form of the numerical scheme~\eqref{scheme-SAV-1}:
 \begin{eqnarray} \label{eng stab-1}
\begin{cases}
\frac{\frac32 \phi^{n+1} - 2 \phi^n + \frac12 \phi^{n-1}}{\dt} = \Delta_N \mu_N^{n+1} , \ \ \ \ \ \ ({\theequation a})\\
\mu_N^{n+1} = L_N \phi^{n+1} + \frac{r^{n+1}}{\sqrt{E_{1,N} (\hat{\phi}^{n+1})}}
 N_N (\hat{\phi}^{n+1})  ,  \ \ \ \ \ \ ({\theequation b})\\
\frac{\frac32 r^{n+1} - 2 r^n + \frac12 r^{n-1}}{\dt} = \frac{1}{2 \sqrt{E_{1,N} (\hat{\phi}^{n+1}) } }   \langle N_N (\hat{\phi}^{n+1}) , \frac{\frac32 \phi^{n+1} - 2 \phi^n + \frac12 \phi^{n-1} }{\dt}   \rangle .  \ \ \ \ \ \ ({\theequation c})
\end{cases}
\end{eqnarray}

Subsequently, taking discrete inner product with (\ref{eng stab-1}a) by $\mu_N^{n+1}$, with (\ref{eng stab-1}b) by $-(\frac32 \phi^{n+1} - 2 \phi^n + \frac12 \phi^{n-1})$, with (\ref{eng stab-1}c) by $2r^{n+1}$, we have
\begin{eqnarray}
 \langle \frac32 \phi^{n+1} - 2 \phi^n + \frac12 \phi^{n-1} , \mu_N^{n+1} \rangle
 &=& \dt \langle \Delta_N \mu_N^{n+1} , \mu_N^{n+1} \rangle
 = - \dt \| \nabla_N \mu_N^{n+1} \|_2^2 ,  \label{eng stab-2-1}
\\
 - \langle \frac32 \phi^{n+1} - 2 \phi^n + \frac12 \phi^{n-1} , \mu_N^{n+1} \rangle
  &=& - \langle L_N \phi^{n+1} , \frac32 \phi^{n+1} - 2 \phi^n + \frac12 \phi^{n-1} \rangle
  \nonumber
\\
  &&
+ \frac{r^{n+1} }{\sqrt{E_{1,N} (\hat{\phi}^{n+1})}} \langle - N_N (\hat{\phi}^{n+1}) ,
  \frac32 \phi^{n+1} - 2 \phi^n + \frac12 \phi^{n-1} \rangle ,  \label{eng stab-2-2}
\\
 2 r^{n+1} ( \frac32 r^{n+1} - 2 r^n + \frac12 r^{n-1} )
 &=& \frac{r^{n+1} }{\sqrt{E_{1,N} (\hat{\phi}^{n+1})}}
 \langle N_N (\hat{\phi}^{n+1}) , \frac32 \phi^{n+1} - 2 \phi^n
 + \frac12 \phi^{n-1} \rangle . \label{eng stab-2-3}
\end{eqnarray}
In turn, by adding~\eqref{eng stab-2-1}, \eqref{eng stab-2-2} and \eqref{eng stab-2-3}, we obtain
\begin{equation}
  \langle L_N \phi^{n+1}, \frac32 \phi^{n+1} - 2 \phi^n + \frac12 \phi^{n-1} \rangle
  + 2 r^{n+1} ( \frac32 r^{n+1} - 2 r^n + \frac12 r^{n-1} )
  = - \dt \| \nabla_N \mu_N^{n+1} \|_2^2 .  \label{eng stab-3}
\end{equation}
Meanwhile, the derivation of the following two identities are straightforward:
\begin{eqnarray}
  &&
   \langle L_N \phi^{n+1}, \frac32 \phi^{n+1} - 2 \phi^n + \frac12 \phi^{n-1} \rangle
   = \langle L_N^\frac12 \phi^{n+1},
    L_N^\frac12 ( \frac32 \phi^{n+1} - 2 \phi^n + \frac12 \phi^{n-1} ) \rangle
   \nonumber
\\
  &=&
     \frac14 ( \| L_N^\frac12 \phi^{n+1} \|_2^2 - \| L_N^\frac12 \phi^n \|_2^2
   + \| L_N^\frac12 ( 2 \phi^{n+1} - \phi^n ) \|_2^2
   - \| L_N^\frac12 ( 2 \phi^n - \phi^{n-1} ) \|_2^2  \nonumber
\\
  &&
    + \| L_N^\frac12 ( \phi^{n+1} - 2 \phi^n + \phi^{n-1} ) \|_2^2  ) ,  \label{eng stab-4-1}
\\
  &&
   2 r^{n+1} ( \frac32 r^{n+1} - 2 r^n + \frac12 r^{n-1} )   \nonumber
\\
  &=&
     \frac12 ( | r^{n+1} |^2 - | r^n |^2 + | 2 r^{n+1} - r^n |^2
   - | 2 r^n - r^{n-1} |^2 + | r^{n+1} - 2 r^n + r^{n-1}  |^2  )  ,   \label{eng stab-4-2}
\end{eqnarray}
in which identity~\eqref{spectral-coll-inner product-4} has been applied in the first step of~\eqref{eng stab-4-1}.  Going back~\eqref{eng stab-3}, we arrive at
\begin{eqnarray}
  &&
  \mathcal{E}_{N} ( \phi^{n+1}, \phi^n, r^{n+1}, r^n)
 -  \mathcal{E}_{N} ( \phi^n, \phi^{n-1}, r^n, r^{n-1} )   \nonumber
\\
  &=&
   - \frac14 \| L_N^\frac12 ( \phi^{n+1} - 2 \phi^n + \phi^{n-1} ) \|_2^2
   - \frac12 | r^{n+1} - 2 r^n + r^{n-1}  |^2
   - \dt \| \nabla_N \mu_N^{n+1} \|_2^2  \le 0 . \label{eng stab-5}
\end{eqnarray}
This completes the proof of Theorem~\ref{SPFC-energy stability}.
\end{proof}

As a direct consequence of the energy stability, a uniform-in-time $H_N^2$ bound for the numerical solution is derived as follows.

	\begin{cor} \label{SPFC: H^2 bound}
Suppose that the initial data are sufficiently regular so that
\begin{equation}
\frac14 ( \| L_N^\frac12 \phi^0 \|_2^2 + \| L_N^\frac12 ( 2  \phi^0 - \phi^{-1} ) \|_2^2 )
 + \frac12 ( | r^0 |^2 + | 2 r^0 - r^{-1} |^2 ) \le \tilde{C}_0,  \label{initial energy-0}
\end{equation}
for some $\tilde{C}_0$ that is independent of $h$. Then we have the following uniform-in-time  $H_N^2$ bound for the numerical solution:
	\begin{equation}
\| \phi^m \|_{H_N^2} \le \tilde{C}_1 ,  \quad \forall \, m \ge 1 , \label{SPFC-H2 stab-0}
	\end{equation}
where $\tilde{C}_1>0$  depends on $\Omega$ and $\tilde{C}_0$, but is independent of $h$, $\dt$ and the time step  $t^m$.
	\end{cor}
	
	\begin{proof}
As a result of \eqref{SPFC-eng stab-est}, the following energy bound is available:
	\begin{align}
\frac14 \| L_N^\frac12 \phi^m \|_2^2 & \le \mathcal{E}_{N} (\phi^m, \phi^{m-1}, r^m, r^{m-1})
 \le   \mathcal{E}_{N} (\phi^0, \phi^{-1}, r^0, r^{-1} )
	\nonumber
	\\
& = \frac14 ( \| L_N^\frac12 \phi^0 \|_2^2 + \| L_N^\frac12 ( 2  \phi^0 - \phi^{-1} ) \|_2^2 )
 + \frac12 ( | r^0 |^2 + | 2 r^0 - r^{-1} |^2 ) \le \tilde{C}_0 ,
	\label{SPFC-H2 bound-1}
	\end{align}
for any $m \ge 1$. On the other hand, the eigenvalue expansion~\eqref{spectral-coll-6} implies the following fact
\begin{equation}
    \| L_N^\frac12 f \|_2^2  = a \| f \|_2^2 + \| \Delta_N f \|_2^2 ,  \quad
    \forall  f \in\mathcal{G}_N .
  \label{SPFC-H2 bound-2}
\end{equation}
Then we arrive at
\begin{eqnarray}
   \| \phi^m \|_2^2 + \| \Delta_N \phi^m \|_2^2 \le \frac{4 \tilde{C}_0}{a} ,   \quad
   \forall m \ge 1 .  \label{SPFC-H2 bound-3}
\end{eqnarray}
And also, the following estimate is available:
\begin{eqnarray}
    \| \nabla_N \phi^m \|_2^2 = - \langle \phi^m , \Delta_N \phi^m \rangle
    \le \| \phi^m \|_2 \cdot \| \Delta_N \phi^m \|_2
    \le \frac12 ( \| \phi^m \|_2^2 + \| \Delta_N \phi^m \|_2^2 )
    \le \frac{2 \tilde{C}_0}{a}  .   \label{SPFC-H2 bound-4}
\end{eqnarray}
Therefore, the following bound is obvious
 	\begin{equation}
\| \phi^m \|_{H_N^2} = \Big( \| \phi^m \|_2^2 + \| \nabla_N \phi^m \|_2^2
  +  \| \Delta_N \phi^m \|_2^2 \Big)^\frac12 \le \Big( \frac{6 \tilde{C}_0}{a} \Big)^\frac12 := \tilde{C}_1  ,
    \quad \forall m \ge 1 .  \label{SPFC-H2 bound-5}
	\end{equation}
This completes the proof of Corollary~\ref{SPFC: H^2 bound}.
	\end{proof}
	
\begin{rem}
It is obvious that the modified energy functional~\eqref{discrete energy} is the second order approximation to the original discrete energy~\eqref{energy-discrete-spectral}, under certain regularity assumption for the numerical solution. Meanwhile, such a modified discrete energy is in terms of a scalar auxiliary variable $r$, combined with the linear surface diffusion energy part, not fully in terms of the original phase variable $\phi$, as formulated in~\eqref{energy-discrete-spectral}. Although a direct bound of the original energy functional is not available in terms of the initial data, a uniform-in-time $H_N^2$ bound for the numerical solution could be derived, up to a constant multiple, as demonstrated in Corollary~\ref{SPFC: H^2 bound}.
\end{rem}

\begin{rem} 	
For various gradient flow equations, the second order numerical scheme using the BDF temporal stencil has attracted many attentions in recent years. For these BDF-type method applied to the original phase variables, an artificial Douglas-Dupont regularization term has to be added to ensure the energy stability; see the related works~\cite{fengW17c, Hao2020, LiW18, Meng2020, yan17} for the epitaxial thin film growth and Cahn-Hilliard equations, respectively. 
On the other hand, for an SAV-based numerical algorithm, such an artificial regularization is not needed, since the concave diffusion term has already been included in the scalar quadrant part.
    \end{rem}

	\begin{rem}
	\label{rem:W16 est}
As a combination of the uniform in time $H_N^2$ bound~\eqref{SPFC-H2 stab-0} and the discrete Sobolev embedding inequality~\eqref{embedding-0}, we arrive at a uniform in time $W_N^{1,6}$ estimate for the numerical solution:
	\begin{equation}
\| \nabla_N \phi^m \|_6 \le C \tilde{C}_1 ,  \quad \forall \ m \ge 1 .
	\label{SPFC-W16 est-0}
	\end{equation}
And also, the modified energy inequality~\eqref{SPFC-H2 bound-1} indicates that
\begin{eqnarray}
   \frac12 | r^m |^2 \le \tilde{C}_0 ,  \quad \mbox{so that} \, \, \,
   r^m \le ( 2 \tilde{C}_0 )^\frac12 ,  \quad \forall m \ge 1.
   \label{SPFC-r est-0}
	\end{eqnarray}
These estimates will be useful in the higher order stability analysis presented below.
\end{rem}

Meanwhile, the established energy stability estimate~\eqref{SPFC-eng stab-est} is in terms of the modified energy functional~\eqref{discrete energy}. On the other hand, for the original discrete energy~\eqref{energy-discrete-spectral}, the following estimate is available, with the help of the uniform-in-time $H_N^2$ bound~\eqref{SPFC-H2 stab-0}, established in Corollary~\ref{SPFC: H^2 bound}.

\begin{prop}   \label{prop: energy bound}
Suppose that the initial data are sufficiently regular~\eqref{initial energy-0} is satisfied, for some $\tilde{C}_0$ that is independent of $h$. Then we have the following uniform-in-time bound for the original energy functional:
	\begin{equation}
 E_N (\phi^m) \le \tilde{C}_1^* ,  \quad \forall \, m \ge 1 , \label{SPFC-energy-0}
	\end{equation}
where $\tilde{C}_1^* >0$  depends on $\Omega$ and $\tilde{C}_0$, but is independent of $h$, $\dt$ and the time step $t^m$.
	\end{prop}

\begin{proof}
By the definition of the $\| \cdot \|_{H_N^2}$ norm~\eqref{spectral-defi-Hm}, we see that
\begin{eqnarray}
  &&
   \frac{a}{2} \| \phi^m \|_2^2 + \frac12 \nrm{\Delta_N \phi^m }_2^2
   \le \frac12 \| \phi^m \|_{H_N^2}^2 \le \frac12 \tilde{C}_1^2 ,  \quad
   \mbox{since $0 \le a \le 1$} ,   \label{SPFC-energy-1-1}
\\
  &&
  \| \nabla_N \phi^m \|_4  \le C \| \nabla_N \phi^m \|_6 \le \breve{C}_1 \| \Delta _N  \phi^m \|_2 ,
  \label{SPFC-energy-1-2}
\\
  &&
  \mbox{so that} \, \, \,
  \frac14 \| \nabla_N \phi^m \|_4^4  \le \frac14 \breve{C}_1^4 \| \Delta _N \phi^m \|_2^4
  \le \frac14 \breve{C}_1^4 \tilde{C}_1^4 ,  \label{SPFC-energy-1-3}
\end{eqnarray}
for any $m \ge 1$, in which the uniform-in-time $H_N^2$ bound~\eqref{SPFC-H2 stab-0} has been extensively applied. Also notice that the discrete H\"older inequality, as well as the Sobolev embedding~\eqref{embedding-0}, have been applied in the derivation of~\eqref{SPFC-energy-1-2}. Then we arrive at
\begin{eqnarray}
   E_N (\phi^m) &=& \frac14 \| \nabla_N \phi^m \|_4^4 + \frac{a}{2} \| \phi^m \|_2^2
    - \| \nabla_N \phi^m \|_2^2 + \frac12 \| \Delta_N \phi^m \|_2^2  \nonumber
\\
  &\le&
  \frac14 \| \nabla_N \phi^m \|_4^4 + \frac{a}{2} \| \phi^m \|_2^2
     + \frac12 \| \Delta_N \phi^m \|_2^2
     \le \frac14 \breve{C}_1^4 \tilde{C}_1^4 +  \frac12 \tilde{C}_1^2 := \tilde{C}_1^* ,
     \label{SPFC-energy-2}
\end{eqnarray}
for any $m \ge 1$. Notice that $\tilde{C}_1^*$ only depends on $\Omega$ and the initial data, henceforth on $\Omega$ and $\tilde{C}_0$, and independent on $h$, $\dt$ and final time. This completes the proof of Proposition~\ref{prop: energy bound}.
\end{proof}

\begin{rem}
For the proposed SAV scheme~\eqref{scheme-SAV-1}, the uniform energy bound $\tilde{C}_1^*$ in~\eqref{SPFC-energy-0} depends on the uniform-in-time $H_N^2$ bound $\tilde{C}_1$ established in~\eqref{SPFC-H2 stab-0}. Since $\tilde{C}_1$ could be represented as a constant multiple of $\tilde{C}_0^\frac12$ (as given by~\eqref{SPFC-H2 bound-5}), while $\tilde{C}_0$ is bounded by the initial energy plus a fixed constant, we conclude that the original energy bound $\tilde{C}_1^*$ turns out to be dependent on the original energy in a quadratic way, as revealed by~\eqref{SPFC-energy-2}. In contrast, the following uniform-in-time bound has been derived in a recent work~\cite{cheng2019d} for the SPFC equation:
\begin{equation}
  E_N (\phi^m ) \le E_N (\phi^0) .   \label{SPFC-alt scheme-energy-0}
\end{equation}
Of course, it is a much sharper estimate for the original energy functional than the one established for the SAV approach, namely~\eqref{SPFC-energy-2}. This difference is based on the fact that, an auxiliary variable~\eqref{SAV-0} has been introduced in the SAV algorithm, so that only the dissipation for the reformulated energy functional~\eqref{discrete energy} is preserved, as established in~\eqref{SPFC-eng stab-est}. In comparison, the primitive variable formulation of the SPFC equation was discussed in \cite{cheng2019d}, which in turn leads to a direct bound of the original energy functional~\eqref{SPFC-alt scheme-energy-0}.

In fact, there have been a great deal of efforts to enforce the stability estimate for the original energy functional in the SAV numerical approach. For example, in two recent works~\cite{ChengQ2020a, ChengQ2020b}, a Lagrange multiplier approach has been introduced, so that the dissipation law for the original energy functional becomes available, if the proposed numerical system is solvable. Meanwhile, due to the nonlinear nature of the Lagrange multiplier approach of the SAV method presented in \cite{ChengQ2020a, ChengQ2020b}, more detailed investigations of the unique solvability analysis have to be undertaken. An application of such an approach to the SPFC equation will also be considered in the future works.
\end{rem}


\subsection{The $\ell^\infty (0,T; H^3)$ bound estimate for the numerical solution} 	

\begin{thm}  \label{thm: H3 est}
For the numerical solution~\eqref{scheme-SAV-1}, the following estimate is available:
\begin{equation}
    \| \phi_S^m \|_{H^3} \le Q^{(3)} ,  
    \quad \forall m \ge 1  ,
    \label{H3 est-0}
\end{equation}
in which $\phi_S^m$ stands for the spectral interpolation of the numerical solution $\phi^m$, as given by formula~\eqref{spectral-coll-projection-2}. The constant $Q^{(3)}$ only depends on the initial $H^3$ data and the domain, and it is independent on $\dt$, $h$ and $T$.
\end{thm}

\begin{proof}
 Taking a discrete inner product with (\ref{scheme-SAV-1}a) by $-2 \Delta_N^3 \phi^{n+1}$, we obtain
\begin{eqnarray}
  &&
  \frac{1}{\dt} \langle \frac32 \phi^{n+1} - 2 \phi^n +\frac12 \phi^{n-1} ,
  -2 \Delta_N^3 \phi^{n+1} \rangle
  +  2 \langle \Delta_N L_N \phi^{n+1} , \Delta_N^3 \phi^{n+1}  \rangle  \nonumber
\\
  &=&
  - 2 \frac{r^{n+1}}{\sqrt{E_{1,N}(\hat{\phi}^{n+1})}}
   \langle \Delta_N N_N (\hat{\phi}^{n+1}) , \Delta_N^3 \phi^{n+1}  \rangle  .
    \label{H3 est-1}
\end{eqnarray}
The temporal stencil term could be analyzed in the same way as in~\eqref{eng stab-4-1}:
\begin{eqnarray}
  &&
   \langle \frac32 \phi^{n+1} - 2 \phi^n +\frac12 \phi^{n-1} ,
  -2 \Delta_N^3 \phi^{n+1} \rangle
  =  \langle \nabla_N \Delta_N ( \frac32 \phi^{n+1} - 2 \phi^n +\frac12 \phi^{n-1} ) ,
   2 \nabla_N \Delta_N \phi^{n+1} \rangle  \nonumber
\\
  &=&
     \frac12 ( \| \nabla_N \Delta_N \phi^{n+1} \|_2^2 - \| \nabla_N \Delta_N \phi^n \|_2^2
   + \| \nabla_N \Delta_N ( 2 \phi^{n+1} - \phi^n ) \|_2^2
   - \| \nabla_N \Delta_N ( 2 \phi^n - \phi^{n-1} ) \|_2^2  \nonumber
\\
  &&
    + \| \nabla_N \Delta_N ( \phi^{n+1} - 2 \phi^n + \phi^{n-1} ) \|_2^2  ) .
    	\label{H3 est-2}
\end{eqnarray}
The surface diffusion part could be handled in a more straightforward way:
\begin{eqnarray}
   \langle \Delta_N L_N \phi^{n+1} , \Delta_N^3 \phi^{n+1}  \rangle
   &=& a \langle \Delta_N \phi^{n+1} , \Delta_N^3 \phi^{n+1}  \rangle
   + \langle \Delta_N^3 \phi^{n+1} , \Delta_N^3 \phi^{n+1}  \rangle   \nonumber
\\
  &=&
    a \| \Delta_N^2 \phi^{n+1} \|_2^2 + \| \Delta_N^3 \phi^{n+1} \|_2^2 .
   \label{H3 est-3}
\end{eqnarray}

For the right hand side nonlinear inner product, we begin with the following observations:
\begin{eqnarray}
    E_{1,N} ( \hat{\phi}^{n+1} ) \ge | \Omega | , \quad
  | r^m | \le ( 2 \tilde{C}_0 )^\frac12 ,  \, \, \, (by~\eqref{SPFC-r est-0}) .
     \label{H3 est-4-2}
\end{eqnarray}
These two bounds imply that
\begin{equation}
   \frac{r^{n+1}}{\sqrt{E_{1,N}(\hat{\phi}^{n+1})}}
   \le \Big( \frac{2 \tilde{C}_0}{ | \Omega | } \Big)^\frac12 .
   \label{H3 est-4-3}
\end{equation}
For the nonlinear inner product, the following expansion is recalled
\begin{eqnarray}
   \Delta_N N_N (\hat{\phi}^{n+1})
   = - \Delta_N \nabla_N \cdot ( | \nabla_N \hat{\phi}^{n+1} |^2
   \nabla_N \hat{\phi}^{n+1}  ) + 2 \Delta_N^2 \hat{\phi}^{n+1} .
   \label{H3 est-5}
\end{eqnarray}
The linear part could be controlled in a standard fashion:
\begin{eqnarray}
   - 2 \frac{r^{n+1}}{\sqrt{E_{1,N}(\hat{\phi}^{n+1})}}
   \langle 2 \Delta_N^2 \hat{\phi}^{n+1} , \Delta_N^3 \phi^{n+1}  \rangle
   &\le& 4  \Big( \frac{2 \tilde{C}_0}{ | \Omega | } \Big)^\frac12
   \| \Delta_N^2 \hat{\phi}^{n+1} \|_2 \cdot  \| \Delta_N^3 \phi^{n+1}  \|_2
   \nonumber
\\
   &\le& \frac{16 \tilde{C}_0}{ | \Omega | }
     \| \Delta_N^2 \hat{\phi}^{n+1} \|_2^2
   + \frac12 \| \Delta_N^3 \phi^{n+1}  \|_2^2  .   \label{H3 est-6}
\end{eqnarray}
For the nonlinear 4-Laplacian part, the following grid function is introduced:
\begin{eqnarray}
  \hat{q}^{n+1} : =  | \nabla_N \hat{\phi}^{n+1} |^2
   \nabla_N \hat{\phi}^{n+1} .   \label{H3 est-7-1}
\end{eqnarray}
This in turn implies that
\begin{eqnarray}
   \| \Delta_N \nabla_N \cdot ( | \nabla_N \hat{\phi}^{n+1} |^2
   \nabla_N \hat{\phi}^{n+1}  )   \|_2
   =    \| \Delta ( \nabla \cdot \hat{q}_S^{n+1}  ) \|_{L^2} ,  \label{H3 est-7-2}
\end{eqnarray}
in which $\hat{q}_S^{n+1}$ is the spectral interpolation of $\hat{q}^{n+1}$, given by formula~\eqref{spectral-coll-projection-2}. Moreover, since $\hat{q}^{n+1}$ is the point-wise interpolation of the continuous function
\begin{eqnarray}
  \varphi_{\hat{q}^{n+1}} :=  | \nabla \hat{\phi}_S^{n+1} |^2
   \nabla \hat{\phi}_S^{n+1}  ,  \quad \mbox{with} \, \, \,
   \hat{\phi}_S^{n+1} = 2 \phi_S^n - \phi_S^{n-1} ,
    \label{H3 est-7-3}
\end{eqnarray}
we see that $\hat{q}_S^{n+1} = R_N ( \varphi_{\hat{q}^{n+1}} )$. In turn, by making use of the aliasing error control inequality stated in Lemma~\ref{lemma:aliasing error}, we conclude that
\begin{eqnarray}
    \| \Delta ( \nabla \cdot \hat{q}_S^{n+1}  ) \|_{L^2}
    \le  \| \hat{q}_S^{n+1} \|_{H^3}  = \| R_N ( \varphi_{\hat{q}^{n+1}} ) \|_{H^3}
    \le 3^\frac32 \| \varphi_{\hat{q}^{n+1}} \|_{H^3}  ,  \quad
    \mbox{since $\varphi_{\hat{q}^{n+1}}  \in {\cal P}_{3 K}$ } .
     \label{H3 est-7-4}
\end{eqnarray}
Meanwhile, for $\varphi_{\hat{q}^{n+1}}$ given by~\eqref{H3 est-7-3}, a detailed expansion and repeated applications of H\"older inequality indicate that
\begin{eqnarray}
   \| \varphi_{\hat{q}^{n+1}} \|_{H^3}
   &\le& C (  \| \varphi_{\hat{q}^{n+1}} \| + \| \nabla \Delta \varphi_{\hat{q}^{n+1}} \| )
   = C (  \| | \nabla \hat{\phi}_S^{n+1} |^2 \nabla \hat{\phi}_S^{n+1}  \|
   + \| \nabla \Delta ( | \nabla \hat{\phi}_S^{n+1} |^2 \nabla \hat{\phi}_S^{n+1} ) \| )
   \nonumber
\\
  &\le&
  C \Big(  \| \nabla \hat{\phi}_S^{n+1} \|_{L^\infty}^2
   \cdot \| \nabla \hat{\phi}_S^{n+1}  \|_{H^3}
   + \| \nabla \nabla \hat{\phi}_S^{n+1} \|_{L^6}^3   \nonumber
\\
  &&  \quad
   +   \| \nabla \hat{\phi}_S^{n+1} \|_{L^\infty}  \cdot
   \| \nabla \nabla \hat{\phi}_S^{n+1} \|_{L^\infty}
   \cdot \| \nabla \hat{\phi}_S^{n+1}  \|_{H^2}   \Big) ,  \label{H3 est-7-5}
\end{eqnarray}
in which the following estimates have been applied
\begin{equation*}
\begin{aligned}
  &
  \| | \nabla \hat{\phi}_S^{n+1} |^2 \nabla \hat{\phi}_S^{n+1}  \|
  \le \| \nabla \hat{\phi}_S^{n+1} \|_{L^\infty}^2 \cdot \| \nabla \hat{\phi}_S^{n+1}  \|
  \le \| \nabla \hat{\phi}_S^{n+1} \|_{L^\infty}^2 \cdot \| \nabla \hat{\phi}_S^{n+1}  \|_{H^3} ,
\\
  &
   \Delta ( | \nabla \hat{\phi}_S^{n+1} |^2 \nabla \hat{\phi}_S^{n+1} )
   =  3 | \nabla \hat{\phi}_S^{n+1} |^2   \nabla \Delta \hat{\phi}_S^{n+1}
     + 6 ( \nabla  \hat{\phi}_S^{n+1} ) ( \nabla \nabla \hat{\phi}_S^{n+1} )
    ( \nabla \nabla \hat{\phi}_S^{n+1} )  ,
\\
  &
   \nabla \Delta ( | \nabla \hat{\phi}_S^{n+1} |^2 \nabla \hat{\phi}_S^{n+1} )
   =  3 | \nabla \hat{\phi}_S^{n+1} |^2   ( \nabla \nabla \Delta \hat{\phi}_S^{n+1} )
     +  6 ( ( \nabla  \hat{\phi}_S^{n+1} ) ( \nabla \nabla \hat{\phi}_S^{n+1} ) )
      \otimes (  \nabla \Delta \hat{\phi}_S^{n+1}  )
\\
   &  \qquad \qquad \qquad
     + 6 ( \nabla  \nabla \hat{\phi}_S^{n+1} ) \otimes ( \nabla \nabla \hat{\phi}_S^{n+1} )
    ( \nabla \nabla \hat{\phi}_S^{n+1} )
    + 12 ( \nabla  \hat{\phi}_S^{n+1} ) ( \nabla \nabla \nabla \hat{\phi}_S^{n+1} )
    ( \nabla \nabla \hat{\phi}_S^{n+1} ) ,
\\
   &
    \| | \nabla \hat{\phi}_S^{n+1} |^2   ( \nabla \nabla \Delta \hat{\phi}_S^{n+1} )  \|
    \le   \| \nabla \hat{\phi}_S^{n+1} \|_{L^\infty}^2
    \cdot \| \nabla \nabla \Delta \hat{\phi}_S^{n+1} \|
    \le    \| \nabla \hat{\phi}_S^{n+1} \|_{L^\infty}^2
    \cdot \| \nabla \hat{\phi}_S^{n+1} \|_{H^3} ,
\\
  &
   \| ( ( \nabla  \hat{\phi}_S^{n+1} ) ( \nabla \nabla \hat{\phi}_S^{n+1} ) )
      \otimes (  \nabla \Delta \hat{\phi}_S^{n+1}  )  \|
   \le C  \|  \nabla  \hat{\phi}_S^{n+1} \|_{L^\infty} \cdot \| \nabla \nabla \hat{\phi}_S^{n+1} \|_{L^\infty}
      \|  \nabla \Delta \hat{\phi}_S^{n+1}   \|
\\
  &  \qquad \qquad \qquad \qquad \qquad \qquad \qquad  \qquad
   \le  C \|  \nabla  \hat{\phi}_S^{n+1} \|_{L^\infty} \cdot \| \nabla \nabla \hat{\phi}_S^{n+1} \|_{L^\infty}
      \|  \nabla \hat{\phi}_S^{n+1}   \|_{H^2}  ,
\\
  &
   \| ( \nabla  \nabla \hat{\phi}_S^{n+1} ) \otimes ( \nabla \nabla \hat{\phi}_S^{n+1} )
    ( \nabla \nabla \hat{\phi}_S^{n+1} )   \|
    \le  C \| \nabla  \nabla \hat{\phi}_S^{n+1} \|_{L^6}^3 ,
\\
  &
  \| ( \nabla  \hat{\phi}_S^{n+1} ) ( \nabla \nabla \nabla \hat{\phi}_S^{n+1} )
    ( \nabla \nabla \hat{\phi}_S^{n+1} )  \|
    \le  \| \nabla  \hat{\phi}_S^{n+1} \|_{L^\infty}
    \cdot  \| \nabla \nabla \nabla \hat{\phi}_S^{n+1} \|
    \cdot \| \nabla \nabla \hat{\phi}_S^{n+1} \|_{L^\infty}
\\
  &  \qquad \qquad \qquad \qquad \qquad \qquad  \qquad \quad
   \le C \| \nabla \hat{\phi}_S^{n+1} \|_{L^\infty}  \cdot
   \| \nabla \nabla \hat{\phi}_S^{n+1} \|_{L^\infty}
   \cdot \| \nabla \hat{\phi}_S^{n+1}  \|_{H^2}  .
\end{aligned}
\end{equation*}
Furthermore, the following 3-D Sobolev embedding and interpolation inequalities could be derived:
\begin{eqnarray}
   \| \nabla \hat{\phi}_S^{n+1} \|_{L^\infty}
   &\le&  C  (  \| \Delta \hat{\phi}_S^{n+1} \|
   +  \| \Delta \hat{\phi}_S^{n+1} \|^\frac78
   \cdot \| \Delta^3  \hat{\phi}_S^{n+1} \|^\frac18  )
   \le C  (  \tilde{C}_1
   +  \tilde{C}_1^\frac78 \| \Delta^3  \hat{\phi}_S^{n+1} \|^\frac18  )  ,
   \label{H3 est-7-6}
\\
  \| \nabla \hat{\phi}_S^{n+1} \|_{H^3}
   &\le&  C  \| \Delta \hat{\phi}_S^{n+1} \|^\frac12
   \cdot \| \Delta^3  \hat{\phi}_S^{n+1} \|^\frac12
   \le C  \tilde{C}_1^\frac12 \| \Delta^3  \hat{\phi}_S^{n+1} \|^\frac12  ,
   \label{H3 est-7-7}
\\
  \| \nabla \nabla \hat{\phi}_S^{n+1} \|_{L^6}
   &\le&  C \| \nabla \nabla \hat{\phi}_S^{n+1} \|_{H^1}
   \le C  \| \Delta \hat{\phi}_S^{n+1} \|^\frac34
   \cdot \| \Delta^3  \hat{\phi}_S^{n+1} \|^\frac14
   \le C  \tilde{C}_1^\frac34 \| \Delta^3  \hat{\phi}_S^{n+1} \|^\frac14  ,
   \label{H3 est-7-8}
\\
     \| \nabla \nabla \hat{\phi}_S^{n+1} \|_{L^\infty}
   &\le&  C  (  \| \nabla \Delta \hat{\phi}_S^{n+1} \|
   +  \| \nabla \Delta \hat{\phi}_S^{n+1} \|^\frac56
   \cdot \| \Delta^3  \hat{\phi}_S^{n+1} \|^\frac16  )  \nonumber
\\
  &\le&
   C  (  \| \Delta \hat{\phi}_S^{n+1} \|^\frac34 \cdot \| \Delta^3 \hat{\phi}_S^{n+1} \|^\frac14
   +  ( \| \Delta \hat{\phi}_S^{n+1} \|^\frac34
   \cdot \| \Delta^3 \hat{\phi}_S^{n+1} \|^\frac14  )^\frac56
   \cdot \| \Delta^3  \hat{\phi}_S^{n+1} \|^\frac16  )    \nonumber
\\
  &\le&
    C  (  \tilde{C}_1^\frac34 \cdot \| \Delta^3 \hat{\phi}_S^{n+1} \|^\frac14
   +  \tilde{C}_1^\frac58 \| \Delta^3  \hat{\phi}_S^{n+1} \|^\frac38  )  ,
   \label{H3 est-7-9}
\\
  \| \nabla \hat{\phi}_S^{n+1} \|_{H^2}
   &\le&  C  \| \Delta \hat{\phi}_S^{n+1} \|^\frac34
   \cdot \| \Delta^3  \hat{\phi}_S^{n+1} \|^\frac14
   \le C  \tilde{C}_1^\frac34 \| \Delta^3  \hat{\phi}_S^{n+1} \|^\frac14  ,
   \label{H3 est-7-10}
\end{eqnarray}
in which the uniform in time $H^2$ bound~\eqref{SPFC-H2 stab-0} of the numerical solution has been extensively used. In turn, a substitution of the above estimates into~\eqref{H3 est-7-5} yields
\begin{eqnarray}
   \| \varphi_{\hat{q}^{n+1}} \|_{H^3}
   \le C  ( \tilde{C}_1^3
   + \tilde{C}_1^\frac94 \| \Delta^3  \hat{\phi}_S^{n+1} \|^\frac34 )  .
    \label{H3 est-7-11}
\end{eqnarray}
Subsequently, its combination with~\eqref{H3 est-7-2} and \eqref{H3 est-7-4} reveals that
\begin{eqnarray}
   \| \Delta_N \nabla_N \cdot ( | \nabla_N \hat{\phi}^{n+1} |^2
   \nabla_N \hat{\phi}^{n+1}  )   \|_2
    &\le& C  ( \tilde{C}_1^3
   + \tilde{C}_1^\frac94 \| \Delta^3  \hat{\phi}_S^{n+1} \|^\frac34 )  \nonumber
\\
  &\le& C  ( \tilde{C}_1^3
   + \tilde{C}_1^\frac94 \| \Delta_N^3  \hat{\phi}^{n+1} \|^\frac34 )  ,
    \label{H3 est-7-11-2}
\end{eqnarray}
in which the fact that $\hat{\phi}_S^{n+1} \in {\cal P}_K$ has been applied in the last step. As a consequence, we arrive at
\begin{eqnarray}
  &&
   2 \frac{r^{n+1}}{\sqrt{E_{1,N}(\hat{\phi}^{n+1})}}
   \langle \Delta_N \nabla_N \cdot ( | \nabla_N \hat{\phi}^{n+1} |^2
   \nabla_N \hat{\phi}^{n+1}  )  , \Delta_N^3 \phi^{n+1}  \rangle   \nonumber
\\
   &\le& 2  \Big( \frac{2 \tilde{C}_0}{ | \Omega | } \Big)^\frac12
   \| \Delta_N \nabla_N \cdot ( | \nabla_N \hat{\phi}^{n+1} |^2
   \nabla_N \hat{\phi}^{n+1}  )   \|_2  \cdot  \| \Delta_N^3 \phi^{n+1}  \|_2
   \nonumber
\\
   &\le&
   C  ( \tilde{C}_1^3
   + \tilde{C}_1^\frac94 \| \Delta_N^3  \hat{\phi}^{n+1} \|^\frac34 )
    \cdot  \| \Delta_N^3 \phi^{n+1}  \|_2
   \le C  ( \tilde{C}_1^6
   + \tilde{C}_1^\frac92 \| \Delta_N^3  \hat{\phi}^{n+1} \|^\frac32 )
   + \frac12 \| \Delta_N^3 \phi^{n+1}  \|_2^2  .   \label{H3 est-7-12}
\end{eqnarray}
A combination of~\eqref{H3 est-6} and \eqref{H3 est-7-12} leads to
\begin{eqnarray}
  &&
   - 2 \frac{r^{n+1}}{\sqrt{E_{1,N}(\hat{\phi}^{n+1})}}
   \langle \Delta_N N_N ( \hat{\phi}^{n+1} ) , \Delta_N^3 \phi^{n+1}  \rangle   \nonumber
\\
  &\le&
     \frac{16 \tilde{C}_0}{ | \Omega | }
     \| \Delta_N^2 \hat{\phi}^{n+1} \|_2^2
   + C  ( \tilde{C}_1^6
   + \tilde{C}_1^\frac92 \| \Delta_N^3  \hat{\phi}^{n+1} \|^\frac32 )
   +  \| \Delta_N^3 \phi^{n+1}  \|_2^2  .    \label{H3 est-8}
\end{eqnarray}

  Finally, a substitution of~\eqref{H3 est-2}, \eqref{H3 est-3} and \eqref{H3 est-8} into \eqref{H3 est-1} results in
\begin{eqnarray}
  &&
     \frac{1}{2 \dt} ( \| \nabla_N \Delta_N \phi^{n+1} \|_2^2
     - \| \nabla_N \Delta_N \phi^n \|_2^2
   + \| \nabla_N \Delta_N ( 2 \phi^{n+1} - \phi^n ) \|_2^2
   - \| \nabla_N \Delta_N ( 2 \phi^n - \phi^{n-1} ) \|_2^2  )  \nonumber
\\
  &&
    + 2  a \| \Delta_N^2 \phi^{n+1} \|_2^2 + \| \Delta_N^3 \phi^{n+1} \|_2^2
    \le  \frac{16 \tilde{C}_0}{ | \Omega | }
     \| \Delta_N^2 \hat{\phi}^{n+1} \|_2^2
   + C  ( \tilde{C}_1^6
   + \tilde{C}_1^\frac92 \| \Delta_N^3  \hat{\phi}^{n+1} \|^\frac32 )  .
   \label{H3 est-9-1}
\end{eqnarray}
Meanwhile, the following interpolation inequality and Cauchy inequality are available:
\begin{eqnarray}
    \| \Delta_N^2 \hat{\phi}^{n+1} \|_2
   &\le&  \| \Delta_N \hat{\phi}^{n+1} \|_2^\frac12
   \cdot \| \Delta_N^3  \hat{\phi}^{n+1} \|_2^\frac12
   \le  ( 3 \tilde{C}_1 )^\frac12 \| \Delta_N^3  \hat{\phi}^{n+1} \|_2^\frac12  ,
   \label{H3 est-9-2}
\\
   \| \Delta_N^3 \hat{\phi}^{n+1} \|_2^2
   &=&  \| \Delta_N^3 ( 2 \phi^n  - \phi^{n-1} ) \|_2^2
  = 4 \| \Delta_N^3  \phi^n \|_2^2
  + \| \Delta_N^3  \phi^{n-1} \|_2^2
  - 4 \langle \Delta_N^3 \phi^n , \Delta_N^2 \phi^{n-1}  \rangle \nonumber
\\
   &\le&
   4 \| \Delta_N^3  \phi^n \|_2^2
  + \| \Delta_N^3  \phi^{n-1} \|_2^2
  + 2 ( \| \Delta_N^3 \phi^n \|_2^2 + \| \Delta_N^2 \phi^{n-1} \|_2^2 ) \nonumber
\\
  &\le&  6 \| \Delta_N^3  \phi^n \|_2^2  + 3 \| \Delta_N^3  \phi^{n-1} \|_2^2 .
   \label{H3 est-9-3}
\end{eqnarray}
Then we obtain the following estimates:
\begin{eqnarray}
  \frac{16 \tilde{C}_0}{ | \Omega | }
     \| \Delta_N^2 \hat{\phi}^{n+1} \|_2^2
     &\le&   \frac{48 \tilde{C}_0 \tilde{C}_1}{ | \Omega | }
     \| \Delta_N^3  \hat{\phi}^{n+1} \|_2
     \le  \frac{48^2 \cdot 9 \tilde{C}_0^2 \tilde{C}_1^2}{ | \Omega |^2 }
     + \frac{1}{36} \| \Delta_N^3  \hat{\phi}^{n+1} \|_2^2   \nonumber
\\
  &\le&
       \frac{C \tilde{C}_0^2 \tilde{C}_1^2}{ | \Omega |^2 }
     + \frac16 \| \Delta_N^3  \phi^n \|_2^2
     + \frac{1}{12} \| \Delta_N^3  \phi^{n-1} \|_2^2  ,   \label{H3 est-9-4}
\\
   C \tilde{C}_1^\frac92 \| \Delta_N^3  \hat{\phi}^{n+1} \|^\frac32
   &\le&  C \tilde{C}_1^{18} +
    \frac{1}{36} \| \Delta_N^3  \hat{\phi}^{n+1} \|^2    \nonumber
\\
  &\le&
      C \tilde{C}_1^{18}
     + \frac16 \| \Delta_N^3  \phi^n \|_2^2
     + \frac{1}{12} \| \Delta_N^3  \phi^{n-1} \|_2^2  ,   \label{H3 est-9-5}
 \end{eqnarray}
 in which the Young's inequality has been applied in the first step of~\eqref{H3 est-9-5}. Going back~\eqref{H3 est-9-1}, we arrive at
\begin{eqnarray}
  &&
     \frac{1}{2 \dt} ( \| \nabla_N \Delta_N \phi^{n+1} \|_2^2
     - \| \nabla_N \Delta_N \phi^n \|_2^2
   + \| \nabla_N \Delta_N ( 2 \phi^{n+1} - \phi^n ) \|_2^2
   - \| \nabla_N \Delta_N ( 2 \phi^n - \phi^{n-1} ) \|_2^2  )  \nonumber
\\
  &&
     + 2  a \| \Delta_N^2 \phi^{n+1} \|_2^2 + \| \Delta_N^3 \phi^{n+1} \|_2^2
    \le  \frac13 \| \Delta_N^3  \phi^n \|_2^2
     + \frac16 \| \Delta_N^3  \phi^{n-1} \|_2^2
    + \frac{C \tilde{C}_0^2 \tilde{C}_1^2}{ | \Omega |^2 }
    + C_1 ( \tilde{C}_1^{18} + 1 )  .  \label{H3 est-9-6}
\end{eqnarray}
Moreover, the following quantity is introduced:
\begin{eqnarray}
  G^{n+1} := \frac12 ( \| \nabla_N \Delta_N \phi^{n+1} \|_2^2
       + \| \nabla_N \Delta_N ( 2 \phi^{n+1} - \phi^n ) \|_2^2 )
       + \frac23 \dt \| \Delta_N^3  \phi^{n+1} \|_2^2
     + \frac16 \dt \| \Delta_N^3  \phi^n \|_2^2 .   \label{H3 est-9-7}
\end{eqnarray}
By adding $\frac13 \| \Delta_N^3  \phi^n \|_2^2$ on both sides of~\eqref{H3 est-9-6}, we obtain the following inequality:
\begin{eqnarray}
  G^{n+1} - G^n + \frac13 \dt  \| \Delta_N^3 \phi^{n+1} \|_2^2
  + \frac16 \dt \| \Delta_N^3  \phi^n \|_2^2
  \le M^{(0)} \dt , \quad
    M^{(0)} =   \frac{C \tilde{C}_0^2 \tilde{C}_1^2}{ | \Omega |^2 }
    + C_1 ( \tilde{C}_1^{18} +1 ) .  \label{H3 est-9-8}
\end{eqnarray}
In addition, the following elliptic regularity estimates are valid, with an application of \eqref{elliptic regularity-0} in Proposition~\ref{prop:elliptic regularity} (by taking $C_2 = \hat{C}_0^{-2}$):
\begin{eqnarray}
    C_2 \| \nabla_N \Delta_N \phi^{n+1} \|_2^2  \le \| \Delta_N^3 \phi^{n+1} \|_2^2 ,  \quad
    C_2 \| \nabla_N \Delta_N \phi^n \|_2^2  \le \| \Delta_N^3 \phi^n \|_2^2  ,
     \label{H3 est-9-9}
\end{eqnarray}
so that we arrive at
\begin{eqnarray}
    \frac{1}{24}  C_2 G^{n+1}  \le \frac13  \| \Delta_N^3 \phi^{n+1} \|_2^2
  + \frac16  \| \Delta_N^3  \phi^n \|_2^2  .   \label{H3 est-9-10}
\end{eqnarray}
Going back~\eqref{H3 est-9-8}, we get
\begin{eqnarray}
  G^{n+1} - G^n + \frac{C_2}{24}  \dt G^{n+1}
    \le M^{(0)} \dt  .  \label{H3 est-9-11}
\end{eqnarray}
An application of induction argument implies that
\begin{eqnarray}
   G^{n+1} \le  ( 1 + \frac{C_2}{24}  \dt )^{-(n+1)} G^0
   +  \frac{24 M^{(0)} }{C_2} .  \label{H3 est-9-12}
\end{eqnarray}
Of course, we could introduce a uniform in time quantity $B_3^* := G^0
   +  \frac{24 M^{(0)} }{C_2}$, so that $\| \nabla_N \Delta_N \phi^m \|^2 \le 2 G^m \le 2 B_3^*$ for any $m \ge 0$. In turn, an application of elliptic regularity shows that
\begin{eqnarray}
     \| \phi_S^m \|_{H^3} \le C \Big( | \overline{\phi^m} | + \| \nabla \Delta \phi^m \| \Big)
     \le C ( | \beta_0 | + ( 2 B_3^* )^{1/2} ) := Q^{(3)} ,
     \quad \forall m \ge 0 . \label{H3 est-9-13}
\end{eqnarray}
in which the uniform in time constant $Q^{(3)}$ depends on $\Omega$ and the initial $H^3$ data. This finishes the proof of Theorem~\ref{thm: H3 est}.
\end{proof}

\begin{rem}
  Higher order $H^m$ estimate (beyond the norm given by the physical energy) is available for many gradient flows, due to the analytic property of the surface diffusion parabolic operator; see the related discussions in~\cite{ChenN16}. There have also been quite a few works of uniform in time $H^2$ estimate for certain energy stable numerical schemes for the Cahn-Hilliard equation~\cite{cheng16a, guo16, shen18b}, beyond the $H^1$ bound given by the energy estimate. Similar numerical estimates for also expected for epitaxial thin film growth and SPFC flows, in which the $H^2$ bound is given by the energy estimate, while an $H^3$ estimate could be derived with the help of higher order analysis, combined with Sobolev inequalities. In fact, similar estimates have also been reported for 2-D incompressible Navier-Stokes equations, in terms of the first, second and higher order temporal numerical approximations; see the delated works~\cite{cheng16b, gottlieb12a, WangX2012}, etc.
\end{rem}

	\section{The optimal rate convergence analysis}  \label{sec:convergence}

Now we proceed into the convergence analysis for the proposed numerical scheme~\eqref{scheme-SAV-1}. Due to the SAV structure of the algorithm, the error estimate has to be performed in the energy norm, i.e., in the $\ell^\infty (0,T; H_N^2) \cap \ell^2 (0, T; H_N^5)$ for the phase variable. Similar techniques have also been applied to the convergence estimate~\cite{LiXL19} for the SAV scheme applied to Cahn-Hilliard equation. These ideas have also been reported for the corresponding analysis for the phase field flow coupled with fluid motion~\cite{chen19a, chen16, diegel15a, diegel17, liuY17}. With an initial data with sufficient regularity, we could assume that the exact solution has regularity of class $\mathcal{R}$:
	\begin{equation}
\Phi \in \mathcal{R} := H^3 (0,T; C^0) \cap H^2 (0,T; H^4) \cap L^\infty (0,T; H^{m+6}).
	\label{assumption:regularity.1}
	\end{equation}
In particular, the following bound is available for the exact solution:
\begin{eqnarray}
  \| \partial_t^m \Phi \|_{L^\infty (0,T; L^\infty)} \le C^* ,  \, \, \, (1 \le m \le 3) , \quad
  \|  \Phi^k \|_{H^{m+6}} \le C^* ,  \, \, \, \forall k \ge 0 .
  \label{regularity-exact-1}
\end{eqnarray}
	
		
	\begin{thm}
	\label{thm:convergence}
Given initial data $\Phi_0 \in H_{\rm per}^{m+6} (\Omega)$, suppose the exact solution for SPFC equation~\eqref{equation-SPFC} is of regularity class $\mathcal{R}$. For $\dt$ and $h$ are sufficiently small, we have
	\begin{equation}
\max_{0\le n\le M}\| \Delta_N ( \Phi^n - \phi^n ) \|_2 +  ( \dt   \sum_{k=1}^{M} \| \nabla_N \Delta_N^2 ( \Phi^k - \phi^k ) \|_2^2 )^{1/2}  \le C ( \dt^2 + h^m ),
	\label{SPFC-convergence-0}
	\end{equation}
where $C>0$ is independent of $\dt$ and $h$, and $\dt = T/M$.
	\end{thm}

\subsection{The consistency analysis} 	
	
For $\Phi \in \mathcal{R}$, we construct an approximate scalar value of $R$ as follows
\begin{eqnarray}
  R^{n+1} := \sqrt{E_{1,N} (\Phi^{n+1})} \, ,  \quad
  E_{1,N}^{n+1} (\Phi^{n+1}) = \frac14 \| \nabla_N \Phi^{n+1} \|_4^4
  - \| \nabla_N \Phi^{n+1} \|_2^2 + 2 | \Omega | .     \label{consistency-0}
\end{eqnarray}
A similar extrapolation $\hat{\Phi}^{n+1} := 2 \Phi^n - \Phi^{n-1}$ is taken.  In turn, a careful consistency analysis indicates the following truncation error estimate:
\begin{eqnarray} \label{consistency-1}
\begin{cases}
\frac{\frac32 \Phi^{n+1} - 2 \Phi^n +\frac12 \Phi^{n-1}}{\dt} = \Delta_N \Big( \frac{R^{n+1}}{\sqrt{ E_{1,N} (\hat{\Phi}^{n+1} ) }} N_N (\hat{\Phi}^{n+1}) + L_N \Phi^{n+1} \Big)
 + \tau_\phi^{n+1} , \ \ \ \ \ \ ({\theequation a}) \\
\frac{\frac32 R^{n+1} - 2 R^n + \frac12 R^{n-1}}{\dt} = \frac{1}{2\sqrt{ E_{1,N} (\hat{\Phi}^{n+1})}} \langle N_N ( \hat{\Phi}^{n+1}) , \frac{\frac32 \Phi^{n+1} - 2 \Phi^n + \frac12 \Phi^{n-1}}{\dt} \rangle + \tau_r^{n+1} .   \ \ \ \ \ \ ({\theequation b})
\end{cases}
\end{eqnarray}
with $\| \tau_\phi^{n+1} \|_2 , | \tau_r^{n+1} | \le C (\dt^2 + h^m)$. The derivation of~\eqref{consistency-1} is accomplished with the help of the spectral approximation estimate and other related estimates; the details are left to interested readers.

The numerical error function is defined at a point-wise level:
	\begin{eqnarray}
e^k := \Phi^k - \phi^k ,  \, \, \,  \tilde{N}^k := N_N (\hat{\Phi}^k) - N_N (\hat{\phi}^k) ,
\quad \forall k \ge 0 .
	\label{SPFC-error function-1}
	\end{eqnarray}
And also, the following scalar numerical errors are introduced
	\begin{eqnarray}
\tilde{r}^k := R^k - r^k ,  \, \, \,
\tilde{E}_1^k := E_{1,N} (\hat{\Phi}^k) - E_{1,N} (\hat{\phi}^k) ,
\quad \forall k \ge 0 .
	\label{SPFC-error function-2}
	\end{eqnarray}
In turn, subtracting the numerical scheme \eqref{scheme-SAV-1} from \eqref{consistency-1} gives
\begin{eqnarray} \label{consistency-2}
\begin{cases}
\frac{\frac32 e^{n+1} - 2 e^n +\frac12 e^{n-1}}{\dt} = \Delta_N \Big( ( \frac{ \tilde{r}^{n+1}}{\sqrt{ E_{1,N} (\hat{\phi}^{n+1} ) }}  - B^{n+1} R^{n+1} \tilde{E}_1^{n+1} )
N_N (\hat{\phi}^{n+1})
+ \frac{R^{n+1}}{\sqrt{ E_{1,N} (\hat{\Phi}^{n+1} ) }}  \tilde{N}^{n+1}  \\
  \qquad \qquad \qquad \qquad  \quad
   + L_N e^{n+1} \Big) + \tau_\phi^{n+1} , \ \ \ \ \ \ ({\theequation a}) \\
\frac{\frac32 \tilde{r}^{n+1} - 2 \tilde{r}^n + \frac12 \tilde{r}^{n-1}}{\dt} = \frac{1}{2\sqrt{ E_{1,N} (\hat{\phi}^{n+1})}} \langle N_N ( \hat{\phi}^{n+1}) , \frac{\frac32 e^{n+1} - 2 e^n + \frac12 e^{n-1}}{\dt} \rangle \\
 \qquad \qquad \qquad \qquad \quad
  + \frac{1}{2\sqrt{ E_{1,N} (\hat{\phi}^{n+1})}} \langle \tilde{N}^{n+1} , \frac{\frac32 \Phi^{n+1} - 2 \Phi^n + \frac12 \Phi^{n-1}}{\dt} \rangle  \\
  \qquad \qquad \qquad \qquad \quad
  - \frac12 B^{n+1} \tilde{E}_1^{n+1}  \langle N_N ( \hat{\Phi}^{n+1})  , \frac{\frac32 \Phi^{n+1} - 2 \Phi^n + \frac12 \Phi^{n-1}}{\dt} \rangle
+ \tau_r^{n+1}  ,  \ \ \ \ \ \ ({\theequation b}) \\
\mbox{with} \, \, \, B^{n+1} = \frac{1}{\sqrt{ E_{1,N} (\hat{\Phi}^{n+1} ) } \sqrt{ E_{1,N}  (\hat{\phi}^{n+1}) } ( \sqrt{ E_{1,N} (\hat{\Phi}^{n+1} ) }
+ \sqrt{ E_{1,N}  (\hat{\phi}^{n+1}) } ) }  . \quad ({\theequation c})
\end{cases}
\end{eqnarray}

\subsection{A few preliminary estimates}

The following estimates are needed in the later analysis.

\begin{lem} \label{lem: pre est}
We have
\begin{eqnarray}
  &&
   E_{1,N} (\hat{\phi}^{n+1} ) \ge | \Omega | ,  \, \, \,
   E_{1,N} (\hat{\Phi}^{n+1} ) \ge | \Omega | , \, \, \,
   0 \le B^{n+1} \le \frac12 | \Omega |^{-\frac32} ,  \label{pre est-0-1}
\\
   &&
   | \tilde{E}_1^{n+1} | \le \tilde{C}_2 \| \nabla_N \hat{e}^{n+1} \|_2 ,
   \label{pre est-0-2}
\\
   &&
   \| \nabla_N  N_N (\hat{\phi}^{n+1}) \| \le \tilde{C}_3 ,
   \label{pre est-0-3}
\\
   &&
   \| \nabla_N \tilde{N}^{n+1} \| \le \tilde{C}_4 \| \nabla_N \Delta_N \hat{e}^{n+1} \|_2 ,
   \label{pre est-0-4}
\\
  &&
   \| \nabla_N \Delta_N f \|_2 \le \| \Delta_N f \|_2^\frac23
   \cdot \| \nabla_N \Delta_N^2 f \|_2^\frac13 ,   \, \, \,
   \| \nabla_N \Delta_N^2 f \|_2 \le \| \nabla_N L_N f \|_2 ,  \quad
   \forall f \in  \mathcal{G}_N  ,  \label{pre est-0-5}
\\
   &&
   \|  \frac{\frac32 \Phi^{n+1} - 2 \Phi^n + \frac12 \Phi^{n-1}}{\dt}  \|_{-1,N} , \, \,
   \| \frac{\frac32 \Phi^{n+1} - 2 \Phi^n + \frac12 \Phi^{n-1}}{\dt}  \|_2 \le C C^* ,
    \label{pre est-0-6}
\end{eqnarray}
in which $\hat{e}^{n+1} := \hat{\Phi}^{n+1} - \hat{\phi}^{n+1} = 2 e^n - e^{n-1}$, and $\tilde{C}_j$ are independent of $\dt$ and $h$, $j=2, 3, 4$.
\end{lem}

\begin{proof}
   The lower bound for $E_{1,N} (\hat{\phi}^{n+1} )$ and $E_{1,N} (\hat{\Phi}^{n+1} )$ comes from their definition, and the estimate $0 \le B^{n+1} \le 2 | \Omega |^{-\frac32}$ is a direct result of its representation given by~(\ref{consistency-2}c).

     Moreover, a detailed expansion for $E_{1,N} (\hat{\phi}^{n+1} )$ and $E_{1,N} (\hat{\Phi}^{n+1} )$ implies that
\begin{eqnarray}
   \tilde{E}_1^{n+1} &=&  E_{1,N} (\hat{\Phi}^{n+1} ) - E_{1,N} (\hat{\phi}^{n+1} )
   \nonumber
\\
  &=&
     \frac14 ( \| \nabla_N \hat{\Phi}^{n+1} \|_4^4 - \| \nabla_N \hat{\phi}^{n+1} \|_4^4 )
   -  ( \| \nabla_N \hat{\Phi}^{n+1} \|_2^2 - \| \nabla_N \hat{\phi}^{n+1} \|_2^2 )
   \nonumber
\\
  &=&
   \frac14 \langle|  \nabla_N \hat{\Phi}^{n+1} |^2 + | \nabla_N \hat{\phi}^{n+1} |^2 ,
   \nabla_N ( \hat{\Phi}^{n+1} + \hat{\phi}^{n+1} ) \cdot \nabla_N \hat{e}^{n+1} \rangle
   \nonumber
\\
  &&
   -  \langle \nabla_N ( \hat{\Phi}^{n+1} + \hat{\phi}^{n+1} ) ,
      \nabla_N \hat{e}^{n+1} \rangle .  \label{pre est-1-1}
\end{eqnarray}
For the first error expansion, an application of discrete H\"older inequality shows that
\begin{eqnarray}
   &&
   \frac14 \Big| \langle|  \nabla_N \hat{\Phi}^{n+1} |^2 + | \nabla_N \hat{\phi}^{n+1} |^2 ,
   \nabla_N ( \hat{\Phi}^{n+1} + \hat{\phi}^{n+1} )
   \cdot \nabla_N \hat{e}^{n+1} \rangle  \Big|  \nonumber
\\
  &\le&
  \frac14 ( \| \nabla_N \hat{\Phi}^{n+1} \|_6^2 + \| \nabla_N \hat{\phi}^{n+1} \|_6^2 )
   \cdot ( \| \nabla_N \hat{\Phi}^{n+1} \|_6  + \| \nabla_N \hat{\phi}^{n+1} \|_6 )
    \cdot \| \nabla_N \hat{e}^{n+1} \|_2   \nonumber
\\
  &\le&
  \frac14 ( (C^*)^2 + C \tilde{C}_1^2 )
   \cdot ( C^* + C \tilde{C}_1 )  \cdot \| \nabla_N \hat{e}^{n+1} \|_2
   \le C ( (C^*)^3 + \tilde{C}_1^3 ) \| \nabla_N \hat{e}^{n+1} \|_2  ,
   \label{pre est-1-2}
\end{eqnarray}
in which the regularity assumption~\eqref{regularity-exact-1} for the exact solution and the discrete $W^{1,6}$ bound \eqref{SPFC-W16 est-0} for the numerical solution have been applied. The second error expansion term in~\eqref{pre est-1-1} could be controled in an even simpler way:
\begin{eqnarray}
   \Big| \langle| \nabla_N ( \hat{\Phi}^{n+1} + \hat{\phi}^{n+1} )  ,
   \nabla_N \hat{e}^{n+1} \rangle  \Big|
  &\le&
   ( \| \nabla_N \hat{\Phi}^{n+1} \|_2  + \| \nabla_N \hat{\phi}^{n+1} \|_2 )
    \cdot \| \nabla_N \hat{e}^{n+1} \|_2   \nonumber
\\
  &\le&
   ( C^* + C \tilde{C}_1 )  \| \nabla_N \hat{e}^{n+1} \|_2  ,
   \label{pre est-1-3}
\end{eqnarray}
with~\eqref{regularity-exact-1}, \eqref{SPFC-W16 est-0}, applied again. This comletes the proof of inequality~\eqref{pre est-0-2}, by setting $\tilde{C}_2 := C  ( (C^*)^3 + \tilde{C}_1^3 + C^* + \tilde{C}_1 )$.

  To obtain a discrete $\ell^2$ estimate for $\nabla_N N_N (\hat{\phi}^{n+1})$, we recall the grid function $\hat{q}^{n+1}$ introduced in~\eqref{H3 est-7-1}, so that the following identity is valid:
\begin{eqnarray}
   \| \nabla_N \nabla_N \cdot ( | \nabla_N \hat{\phi}^{n+1} |^2
   \nabla_N \hat{\phi}^{n+1}  )   \|_2
   =    \| \nabla ( \nabla \cdot \hat{q}_S^{n+1}  ) \|_{L^2} ,  \label{pre est-3-1}
\end{eqnarray}
in which $\hat{q}_S^{n+1}$ is the spectral interpolation of $\hat{q}^{n+1}$. Because of the  the fact $\hat{q}_S^{n+1} = R_N ( \varphi_{\hat{q}^{n+1}} )$, as indicated by the point-wise interpolation given by~\eqref{H3 est-7-3}, we make use of the aliasing error control inequality in Lemma~\ref{lemma:aliasing error} and get
\begin{eqnarray}
    \| \nabla ( \nabla \cdot \hat{q}_S^{n+1}  ) \|_{L^2}
    \le  \| \hat{q}_S^{n+1} \|_{H^2}  = \| R_N ( \varphi_{\hat{q}^{n+1}} ) \|_{H^2}
    \le 3^\frac32 \| \varphi_{\hat{q}^{n+1}} \|_{H^2}  ,
     \label{pre est-3-2}
\end{eqnarray}
an inequality similar to~\eqref{H3 est-7-4}. Moreover, a detailed expansion and repeated applications of H\"older inequality lead to
\begin{eqnarray}
   \| \varphi_{\hat{q}^{n+1}} \|_{H^2}
   &\le& C (  \| \varphi_{\hat{q}^{n+1}} \| + \| \Delta \varphi_{\hat{q}^{n+1}} \| )
   = C (  \| | \nabla \hat{\phi}_S^{n+1} |^2 \nabla \hat{\phi}_S^{n+1}  \|
   + \| \Delta ( | \nabla \hat{\phi}_S^{n+1} |^2 \nabla \hat{\phi}_S^{n+1} ) \| )
   \nonumber
\\
  &\le&
  C \Big(  \| \nabla \hat{\phi}_S^{n+1} \|_{L^\infty}^2
   \cdot \| \nabla \hat{\phi}_S^{n+1}  \|_{H^2}
   +   \| \nabla \hat{\phi}_S^{n+1} \|_{L^\infty}  \cdot
   \| \nabla \nabla \hat{\phi}_S^{n+1} \|_{L^4}^2  \Big)  \nonumber
\\
  &\le&
  C  \| \nabla \hat{\phi}_S^{n+1} \|_{H^2}^3
   \le   C  \| \hat{\phi}_S^{n+1} \|_{H^3}^3
   \le C ( Q^{(3)} )^3 ,  \label{pre est-3-3}
\end{eqnarray}
in which the uniform in time $H^3$ estimate~\eqref{H3 est-0} (for the numerical solution) has been applied in the last step. Going back~\eqref{pre est-3-2} and \eqref{pre est-3-1}, we arrive at
\begin{eqnarray}
   \| \nabla_N \nabla_N \cdot ( | \nabla_N \hat{\phi}^{n+1} |^2
   \nabla_N \hat{\phi}^{n+1}  )   \|_2
   \le C ( Q^{(3)} )^3 .   \label{pre est-3-4}
\end{eqnarray}
The other expansion term in $\nabla_N N_N (\hat{\phi}^{n+1})$ could be bounded in a more standard way:
\begin{eqnarray}
   \| 2 \nabla_N \Delta_N  \hat{\phi}^{n+1} \|_2
   \le 2 \| \hat{\phi}_S^{n+1} \|_{H^3}  \le 6 Q^{(3)}  .   \label{pre est-3-5}
\end{eqnarray}
Therefore, a combination of~\eqref{pre est-3-4} and \eqref{pre est-3-5} gives the inequality~\eqref{pre est-0-3}, by taking $\tilde{C}_3 = C ( Q^{(3)} )^3 + 6 Q^{(3)}$.

Inequality~\eqref{pre est-0-4} could be derived in a similar manner. Making a comparison between $N_N (\hat{\Phi}^{n+1})$ and $N_N (\hat{\phi}^{n+1})$, we observe that $\tilde{N}^{n+1}$ turns out to be the point-wise interpolation of the following continuous function
\begin{eqnarray}
  \tilde{N}_S^{n+1} = -\nabla \cdot ( R_N ( \varphi_{\tilde{N}^{n+1}} )
   + 2 \Delta \hat{e}_S^{n+1}  ,   \quad
  \varphi_{\tilde{N}^{n+1}} :=  | \nabla \hat{\Phi}_S^{n+1} |^2
   \nabla \hat{\Phi}_S^{n+1} - | \nabla \hat{\phi}_S^{n+1} |^2
   \nabla \hat{\phi}_S^{n+1}  ,   \label{pre est-4-1}
\end{eqnarray}
with $\hat{\Phi}_S^{n+1} = 2 \Phi_S^n - \Phi_S^{n-1}$, $\hat{e}_S^{n+1} = 2 e_S^n -e_S^{n-1}$. A similar expansion is available for $\varphi_{\tilde{N}^{n+1}}$:
\begin{eqnarray}
  \varphi_{\tilde{N}^{n+1}} =  | \nabla \hat{\Phi}_S^{n+1} |^2
   \nabla \hat{e}_S^{n+1} +  ( \nabla ( \hat{\Phi}_S^{n+1} + \hat{\phi}_S^{n+1} )
   \cdot \nabla \hat{e}_S^{n+1} )  \nabla \hat{\phi}_S^{n+1}  .   \label{pre est-4-2}
\end{eqnarray}
Again, repeated applications of H\"older inequality gives the following estimates
\begin{eqnarray}
\begin{aligned}
  &
  \| \varphi_{\tilde{N}^{n+1}} \|
  \le
     \| | \nabla \hat{\Phi}_S^{n+1} |^2 \nabla \hat{e}_S^{n+1}  \|
    + \|  ( \nabla ( \hat{\Phi}_S^{n+1} + \hat{\phi}_S^{n+1} )
   \cdot \nabla \hat{e}_S^{n+1} )  \nabla \hat{\phi}_S^{n+1} \|
\\
  &  \qquad \qquad
   \le
     \| \nabla \hat{\Phi}_S^{n+1} \|_{L^\infty}^2 \cdot \| \nabla \hat{e}_S^{n+1}  \|
    + (  \| \nabla \hat{\Phi}_S^{n+1} \|_{L^\infty} + \| \nabla \hat{\phi}_S^{n+1} \|_{L^\infty} )
    \| \nabla \hat{e}_S^{n+1} \| \cdot \|  \nabla \hat{\phi}_S^{n+1} \|_{L^\infty}
\\
   &  \qquad \qquad
     \le C ( \| \nabla \hat{\Phi}_S^{n+1} \|_{H^2}^2
     + (  \| \nabla \hat{\Phi}_S^{n+1} \|_{H^2} + \| \nabla \hat{\phi}_S^{n+1} \|_{H^2} )^2 )
     \| \nabla \hat{e}_S^{n+1}  \| ,
\\
  &
  \Delta ( | \nabla \hat{\Phi}_S^{n+1} |^2 \nabla \hat{e}_S^{n+1}  )
  =
  | \nabla \hat{\Phi}_S^{n+1} |^2 ( \nabla \Delta \hat{e}_S^{n+1} )
  + 2 ( \nabla \hat{\Phi}_S^{n+1} ) ( \nabla \nabla \hat{\Phi}_S^{n+1} )
  ( \nabla \nabla \hat{e}_S^{n+1} )
\\
 &  \qquad \qquad \qquad
 + 2 \Delta \hat{\Phi}_S^{n+1} ( \nabla \nabla \hat{\Phi}_S^{n+1} )
  ( \nabla \hat{e}_S^{n+1} )
  + 2 ( \nabla \hat{\Phi}_S^{n+1} \cdot \nabla \Delta \hat{\Phi}_S^{n+1} )
  ( \nabla \hat{e}_S^{n+1} )  ,
\\
  &
  \| \Delta ( | \nabla \hat{\Phi}_S^{n+1} |^2 \nabla \hat{e}_S^{n+1}  )   \|
  \le
  \| \nabla \hat{\Phi}_S^{n+1} \|_{L^\infty}^2 \cdot \| \nabla \Delta \hat{e}_S^{n+1} \|
  + 2 \| \nabla \hat{\Phi}_S^{n+1} \|_{L^\infty} \cdot \| \nabla \nabla \hat{\Phi}_S^{n+1} \|_{L^4}
  \cdot \| \nabla \nabla \hat{e}_S^{n+1} \|_{L^4}
\\
 &  \qquad
 + C \| \Delta \hat{\Phi}_S^{n+1} \|_{L^4} \cdot \| \nabla \nabla \hat{\Phi}_S^{n+1} \|_{L^4}
  \cdot \| \nabla \hat{e}_S^{n+1} \|_{L^\infty}
  + C \| \nabla \hat{\Phi}_S^{n+1} \|_{L^\infty} \cdot \| \nabla \Delta \hat{\Phi}_S^{n+1} \|
  \cdot \| \nabla \hat{e}_S^{n+1} \|_{L^\infty}
\\
  &   \qquad  \qquad \qquad \qquad  \qquad
  \le  C \| \nabla \hat{\Phi}_S^{n+1} \|_{H^2}^2  \cdot \| \nabla \hat{e}_S^{n+1} \|_{H^2}  ,
\\
  &
  \| \Delta ( ( \nabla ( \hat{\Phi}_S^{n+1} + \hat{\phi}_S^{n+1} )
   \cdot \nabla \hat{e}_S^{n+1} )  \nabla \hat{\phi}_S^{n+1} ) \|
\\
  &   \qquad \qquad
   \le  C ( \| \nabla \hat{\Phi}_S^{n+1} \|_{H^2}^2
     + \| \nabla \hat{\phi}_S^{n+1} \|_{H^2}^2   )
     \| \nabla \hat{e}_S^{n+1} \|_{H^2} ,  \quad \mbox{(by a similar analysis)} ,
 \\
  &
  \| \Delta \varphi_{\tilde{N}^{n+1}} \|
  \le
    C \Big(  \| \Delta ( | \nabla \hat{\Phi}_S^{n+1} |^2 \nabla \hat{e}_S^{n+1} ) \|
   + \| \Delta ( ( \nabla ( \hat{\Phi}_S^{n+1} + \hat{\phi}_S^{n+1} )
   \cdot \nabla \hat{e}_S^{n+1} )  \nabla \hat{\phi}_S^{n+1} ) \|   \Big)
\\
  &  \qquad \qquad \, \, \,
  \le
  C   ( \| \nabla \hat{\Phi}_S^{n+1} \|_{H^2}
     + \| \nabla \hat{\phi}_S^{n+1} \|_{H^2}  )^2
   \cdot \| \nabla \hat{e}_S^{n+1} \|_{H^2}  ,
\\
  &
   \| \varphi_{\tilde{N}^{n+1}} \|_{H^2}
   \le C (  \| \varphi_{\tilde{N}^{n+1}} \| + \| \Delta \varphi_{\tilde{N}^{n+1}} \| )
\\
  &  \qquad \qquad \quad
  \le  C  ( \| \nabla \hat{\Phi}_S^{n+1} \|_{H^2}
     + \| \nabla \hat{\phi}_S^{n+1} \|_{H^2}  )^2
   \cdot \| \nabla \hat{e}_S^{n+1} \|_{H^2}
\\
  &  \qquad \qquad \quad
   \le
  C   ( \| \nabla \hat{\Phi}_S^{n+1} \|_{H^2}
     + \| \nabla \hat{\phi}_S^{n+1} \|_{H^2}  )^2
   \cdot \| \nabla \hat{e}_S^{n+1} \|_{H^2}
\\
  &  \qquad \qquad \quad
  \le
    C ( (C^*)^2 + ( Q^{(3)} )^2 )  \| \nabla \hat{e}_S^{n+1} \|_{H^2} ,
\end{aligned}
   \label{pre est-4-3}
\end{eqnarray}
with the uniform in time $H^3$ estimate~\eqref{H3 est-0} and the regularity assumption~\eqref{regularity-exact-1} recalled. Also notice that the 3-D Sobolev embedding, from $H^2$ to $L^\infty$ and $W^{1,4}$, has also been repeatedly applied in the derivation of~\eqref{pre est-4-3}. Since $\varphi_{\tilde{N}^{n+1}} \in {\cal P}_{3K}$, we go back~\eqref{pre est-4-1} and arrive at
\begin{eqnarray}
  \| \nabla_N \tilde{N}^{n+1} \|_2
  &=& \| \nabla \tilde{N}_S^{n+1} \|
  = \| \nabla ( - \nabla \cdot ( R_N ( \varphi_{\tilde{N}^{n+1}} ) ) )
   + 2 \Delta \hat{e}_S^{n+1} \|     \nonumber
\\
  &\le&
    3^\frac32 \| \varphi_{\tilde{N}^{n+1}} \|_{H^2}  +  2 \| \Delta \hat{e}_S^{n+1} \|
    \le C ( (C^*)^2 + ( Q^{(3)} )^2 )  \| \nabla \hat{e}_S^{n+1} \|_{H^2}
    + 2   \| \nabla \hat{e}_S^{n+1} \|_{H^2}    \nonumber
\\
  &\le&
       C ( (C^*)^2 + ( Q^{(3)} )^2 +1 )  \| \nabla \hat{e}_S^{n+1} \|_{H^2}  \nonumber
\\
  &\le&
       C ( (C^*)^2 + ( Q^{(3)} )^2 +1 )  \| \nabla \Delta \hat{e}_S^{n+1} \|  \nonumber
\\
  &\le&
       C ( (C^*)^2 + ( Q^{(3)} )^2 +1 )  \| \nabla_N \Delta_N \hat{e}^{n+1} \|_2 ,
   \label{pre est-4-4}
\end{eqnarray}
in which the elliptic regularity, $\| \nabla \hat{e}_S^{n+1} \|_{H^2}
\le C \| \nabla \Delta \hat{e}_S^{n+1} \|$, has been applied in the fourth step, due to the fact that $\int_\Omega \,   \nabla \hat{e}_S^{n+1} \, d {\bf x} =0$, and the last step comes from the fact that $\hat{e}_S^{n+1}$ is the spectral interpolation function of $\hat{e}^{n+1}$. This completes the proof of inequality~\eqref{pre est-0-4}, by setting $\tilde{C}_4 = C ( (C^*)^2 + ( Q^{(3)} )^2 +1 )$.

For the first inequality in~\eqref{pre est-0-5}, we see that an application of the summation by parts formula~\eqref{spectral-coll-inner product-3} gives
\begin{equation}
  \| \nabla_N \Delta_N f \|_2^2 = - \langle \Delta_N f , \Delta_N^2 f \rangle
  \le \| \Delta_N f \|_2 \cdot \| \Delta_N^2 f \|_2 .
  \label{pre est-5-1}
 \end{equation}
 Meanwhile, another summation by parts formula reveals that
 \begin{equation}
  \| \Delta_N^2 f \|_2^2 = - \langle \nabla_N \Delta_N f ,
  \nabla_N \Delta_N^2 f \rangle
  \le \| \nabla_N \Delta_N f \|_2 \cdot \| \nabla_N \Delta_N^2 f \|_2 .
  \label{pre est-5-2}
 \end{equation}
 Therefore, a combination of~\eqref{pre est-5-1} and \eqref{pre est-5-2} leads to
 \begin{align}
   \| \nabla_N \Delta_N f \|  \le  \| \Delta_N f \|_2^\frac12 \cdot \| \Delta_N^2 f \|_2^\frac12
     & \le  \| \Delta_N f \|_2^\frac12 \cdot ( \| \nabla_N \Delta_N f \|_2^\frac12
     \cdot \| \nabla_N \Delta_N^2 f \|_2^\frac12  )^\frac12
	\nonumber
	\\
&= \| \Delta_N f \|_2^\frac12 \cdot \| \nabla_N \Delta_N f \|_2^\frac14
 \cdot \| \nabla_N \Delta_N^2 f \|_2^\frac14 ,
	\label{pre est-5-3}
	\end{align}
which in turn results in
	\begin{equation}
\| \nabla_N \Delta_N f \|^\frac34 \le  \| \Delta_N f \|_2^\frac12
  \cdot \| \nabla_N \Delta_N^2 f \|_2^\frac14 ,
   \quad \mbox{i.e.,}  \quad  \| \nabla_N \Delta_N f \|
   \le  \| \Delta_N f \|_2^\frac23 \cdot \| \nabla_N \Delta_N^2 f \|_2^\frac13 .
	 \label{pre est-5-4}
	\end{equation}
This finishes the proof of the first inequality in~\eqref{pre est-0-5}.

  For the second inequality, we see that $\nabla_N \Delta_N^2 f$ and $\nabla_N L_N$ have the following discrete Fourier expansions
\begin{eqnarray}
     &&
     \nabla_N \Delta_N^2 f_{i,j,k} := \sum_{\ell,m,n =-K}^K  \Big( 2 \ell \pi {\rm i}, 2 m \pi {\rm i} , 2 n \pi {\rm i} \Big)^T  \lambda_{\ell,m,n}^2
   \hat{f}_{\ell,m,n}^N \exp \left( 2 \pi {\rm i} ( \ell x_i + m y_j + n z_k) \right) ,
	\label{pre est-6-1}
\\
	&&
  \nabla_N L_N f_{i,j,k} := \sum_{\ell,m,n =-K}^K  \Big( 2 \ell \pi {\rm i}, 2 m \pi {\rm i} , 2 n \pi {\rm i} \Big)^T \Big( a + \lambda_{\ell,m,n}^2 \Big)
   \hat{f}_{\ell,m,n}^N \exp \left( 2 \pi {\rm i} ( \ell x_i + m y_j + n z_k) \right) ,
	\label{pre est-6-2}
\end{eqnarray}
for $f$ given by~(\ref{spectral-coll-1}). In turn, an application of the Parseval inequality implies that
\begin{eqnarray}
     &&
     \| \nabla_N \Delta_N^2 f \|_2^2 = \sum_{\ell,m,n =-K}^K   | \lambda_{\ell,m,n} |^5
   | \hat{f}_{\ell,m,n}^N |^2 ,
	\label{pre est-6-3}
\\
	&&
  \| \nabla_N L_N f \|_2^2 = \sum_{\ell,m,n =-K}^K  \lambda_{\ell,m,n}
  \Big( a + \lambda_{\ell,m,n}^2 \Big)^2    | \hat{f}_{\ell,m,n}^N |^2 .
	\label{pre est-6-4}
\end{eqnarray}
As a result, the second inequality in~\eqref{pre est-0-5} comes from the fact that $| \lambda_{\ell,m,n}^2 | \le | a + \lambda_{\ell,m,n}^2 |$.

The last inequality~\eqref{pre est-0-6} is a direct consequence of the following estimates
\begin{eqnarray}
   \|  \frac{\Phi^{n+1} - \Phi^n}{\dt}  \|_\infty \le C^* ,  \, \, \,
   \| \frac{\Phi^n - \Phi^{n-1}}{\dt}  \|_\infty \le C^* ,  \quad
   \mbox{by~\eqref{regularity-exact-1} } ,
\end{eqnarray}
combined with the fact that $\| \cdot \|_\infty$ is a norm stronger than $\| \cdot \|_2$ and $\| \cdot \|_{-1,N}$.
\end{proof}

\subsection{Proof of the convergence theorem}

Now we proceed into the proof of Theorem~\ref{thm:convergence}.

\begin{proof} 	
Taking a discrete inner product of (\ref{consistency-2}a) with $(-\Delta_N)^{-1} ( \frac{\frac32 e^{n+1} - 2 e^n + \frac12 e^{n-1}}{\dt} ) $, with a repeated application of summation by parts, we get
	\begin{align}
  & \frac{1}{\dt} \langle \frac32 e^{n+1} - 2 e^n + \frac12 e^{n-1} , L_N e^{n+1} \rangle
  + \| \frac{\frac32 e^{n+1} - 2 e^n + \frac12 e^{n-1}}{\dt}  \|_{-1,N}^2
	\nonumber
\\
  =&
 - \langle  \nabla_N ( {\cal NLE}_1 + {\cal NLE}_2 + {\cal NLE}_3 - (\Delta_N)^{-1} \tau_\phi^{n+1} ) ,
   \nabla_N (-\Delta_N)^{-1} ( \frac{ \frac32 e^{n+1} - 2 e^n
   + \frac12 e^{n-1} }{\dt} )  \rangle  ,  \label{convergence-1}
\\
  &
  {\cal NLE}_1 = \frac{ \tilde{r}^{n+1}}{\sqrt{ E_{1,N} (\hat{\phi}^{n+1} ) }}
  N_N (\hat{\phi}^{n+1})   ,  \nonumber
\\
  &
{\cal NLE}_2 =  - B^{n+1} R^{n+1} \tilde{E}_1^{n+1}  N_N (\hat{\phi}^{n+1}) , \, \, \,
{\cal NLE}_3 =  \frac{R^{n+1}}{\sqrt{ E_{1,N} (\hat{\Phi}^{n+1} ) }}  \tilde{N}^{n+1} .
     \nonumber
	\end{align}
The temporal stencil term could be analyzed in the same manner as~\eqref{eng stab-4-1}:
\begin{eqnarray}
  &&
   \langle \frac32 e^{n+1} - 2 e^n + \frac12 e^{n-1} ,  L_N e^{n+1} \rangle
   \nonumber
\\
  &=&
     \frac14 ( \| L_N^\frac12 e^{n+1} \|_2^2 - \| L_N^\frac12 e^n \|_2^2
   + \| L_N^\frac12 ( 2 e^{n+1} - e^n ) \|_2^2
   - \| L_N^\frac12 ( 2 e^n - e^{n-1} ) \|_2^2  \nonumber
\\
  &&
    + \| L_N^\frac12 ( e^{n+1} - 2 e^n + e^{n-1} ) \|_2^2  ) .  \label{convergence-2}
\end{eqnarray}
A bound for the truncation error inner product term is standard:
\begin{eqnarray}
  &&
    \langle  \nabla_N  (\Delta_N)^{-1} \tau_\phi^{n+1}  ,
   \nabla_N (-\Delta_N)^{-1} ( \frac{ \frac32 e^{n+1} - 2 e^n
   + \frac12 e^{n-1} }{\dt} )  \rangle    \nonumber
\\
   &\le&
    \| \tau_\phi^{n+1}  \|_{-1,N}
    \cdot  \|  \frac{ \frac32 e^{n+1} - 2 e^n + \frac12 e^{n-1} }{\dt} \|_{-1,N}  \nonumber
\\
    &\le&
    2 \| \tau_\phi^{n+1}  \|_{-1,N}^2
    + \frac18  \|  \frac{ \frac32 e^{n+1} - 2 e^n + \frac12 e^{n-1} }{\dt} \|_{-1,N}^2 .
     \label{convergence-3}
\end{eqnarray}
The first nonlinear inner product term could be rewritten as follows:
\begin{eqnarray}
  &&
   - \langle  \nabla_N {\cal NLE}_1 ,
   \nabla_N (-\Delta_N)^{-1} ( \frac{ \frac32 e^{n+1} - 2 e^n
   + \frac12 e^{n-1} }{\dt} )  \rangle   \nonumber
\\
  &=&
   - \langle \frac{ \tilde{r}^{n+1}}{\sqrt{ E_{1,N} (\hat{\phi}^{n+1} ) }}
   N_N (\hat{\phi}^{n+1}) ,
      \frac{ \frac32 e^{n+1} - 2 e^n + \frac12 e^{n-1} }{\dt}  \rangle  .
     \label{convergence-4}
\end{eqnarray}
For the second and third nonlinear inner product terms, we begin with the following estimates:
\begin{eqnarray}
   \| \nabla_N {\cal NLE}_2 \|_2
   &=& \| B^{n+1} R^{n+1} \tilde{E}_1^{n+1}  \nabla_N N_N (\hat{\phi}^{n+1})  \|_2
   \le | B^{n+1} | \cdot | R^{n+1} | \cdot | \tilde{E}_1^{n+1} |
   \cdot  \| \nabla_N N_N (\hat{\phi}^{n+1})  \|_2   \nonumber
\\
  &\le&
   \frac12 | \Omega |^{-\frac32} \cdot ( \tilde{C}_0 + 1)^\frac12
   \cdot \tilde{C}_2 \| \nabla_N \hat{e}^{n+1} \|_2 \cdot \tilde{C}_3  \nonumber
\\
  &=&
  \tilde{C}_5  \| \nabla_N \hat{e}^{n+1} \|_2  , \, \, \, \mbox{with} \, \,
  \tilde{C}_5 =  \frac12 \tilde{C}_2  \tilde{C}_3 ( \tilde{C}_0 + 1)^\frac12
    | \Omega |^{-\frac32} ,   \label{convergence-5-1}
\\
     \| \nabla_N {\cal NLE}_3 \|_2
   &=& \| R^{n+1} ( E_{1,N} ( \hat{\Phi}^{n+1} ) )^{-\frac12}
   \nabla_N \tilde{N}^{n+1} \|_2
   \le | R^{n+1} | \cdot | \Omega |^{-\frac12}
   \cdot  \| \nabla_N \tilde{N}^{n+1}  \|_2   \nonumber
\\
  &\le&
   ( \tilde{C}_0 + 1)^\frac12 | \Omega |^{-\frac12} \
   \cdot \tilde{C}_4 \| \nabla_N \Delta_N \hat{e}^{n+1} \|_2 \ \nonumber
\\
  &=&
  \tilde{C}_6  \| \nabla_N \Delta_N \hat{e}^{n+1} \|_2  , \, \, \, \mbox{with} \, \,
  \tilde{C}_6 =  \tilde{C}_4 ( \tilde{C}_0 + 1)^\frac12
    | \Omega |^{-\frac12} ,   \label{convergence-5-2}
\end{eqnarray}
in which the preliminary estimates~\eqref{pre est-0-1}-\eqref{pre est-0-4} in Lemma~\ref{lem: pre est} have been extensively applied in the derivation. We also notice that the inequality $| R^{n+1} | \le
 ( \tilde{C}_0 + 1 )^\frac12$ comes from the fact that $E (\Phi(t)) \le E (\Phi^0) = \tilde{C}_0 + h^m$, the pseudo-spectral approximation order, combined with the inequality$E_{1,N} (\Phi^k) \le E_N (\Phi^k)$. And also, the following estimate for $\| \nabla_N {\cal NLE}_1 \|_2$ is derived below, which will be needed in the later analysis:
\begin{eqnarray}
     \| \nabla_N {\cal NLE}_1 \|_2
   &=& \| \tilde{r}^{n+1} ( E_{1,N} ( \hat{\phi}^{n+1} ) )^{-\frac12}
   \nabla_N N_N ( \hat{\phi}^{n+1} ) \|_2
   \le | \tilde{r}^{n+1} | \cdot | \Omega |^{-\frac12}
   \cdot  \| \nabla_N N_N (\hat{\phi}^{n+1})  \|_2   \nonumber
\\
  &\le&
   | \Omega |^{-\frac12} \
   \cdot \tilde{C}_3 \cdot \tilde{r}^{n+1}
   =
  \tilde{C}_7  \tilde{r}^{n+1}  , \, \, \, \mbox{with} \, \,
  \tilde{C}_7 =  \tilde{C}_3 | \Omega |^{-\frac12} .   \label{convergence-5-3}
\end{eqnarray}
As a consequence of~\eqref{convergence-5-1}, \eqref{convergence-5-2}, the following inequalities are available:
\begin{eqnarray}
  &&
   - \langle  \nabla_N ( {\cal NLE}_2 + {\cal NLE}_3  ) ,
   \nabla_N (-\Delta_N)^{-1} ( \frac{ \frac32 e^{n+1} - 2 e^n
   + \frac12 e^{n-1} }{\dt} )  \rangle  \nonumber
\\
  &\le&
  (  \| \nabla_N {\cal NLE}_2 \|_2 + \| \nabla_N {\cal NLE}_3  \|_2 )
   \cdot \| \frac{ \frac32 e^{n+1} - 2 e^n + \frac12 e^{n-1} }{\dt}  \|_{-1,N}
   \nonumber
\\
  &\le&
   2 (  \| \nabla_N {\cal NLE}_2 \|_2^2 + \| \nabla_N {\cal NLE}_3  \|_2^2  )
   + \frac14 \| \frac{ \frac32 e^{n+1} - 2 e^n + \frac12 e^{n-1} }{\dt}  \|_{-1,N}^2  \nonumber
\\
  &\le&
   2 (  \tilde{C}_5^2 \| \nabla_N \hat{e}^{n+1} \|_2^2
   + \tilde{C}_6^2 \| \nabla_N \Delta_N \hat{e}^{n+1}  \|_2^2  )
   + \frac14 \| \frac{ \frac32 e^{n+1} - 2 e^n + \frac12 e^{n-1} }{\dt}  \|_{-1,N}^2
   \nonumber
\\
  &\le&
  \tilde{C}_8 \| \nabla_N \Delta_N \hat{e}^{n+1}  \|_2^2
   + \frac14 \| \frac{ \frac32 e^{n+1} - 2 e^n + \frac12 e^{n-1} }{\dt}  \|_{-1,N}^2  ,
   \, \, \, \tilde{C}_8 = 2 (  \tilde{C}_5^2 C_3^2 + \tilde{C}_6^2  ) ,
   \label{convergence-5-4}
\end{eqnarray}
in which $C_3$ corresponds to the elliptic regularity, $\| \nabla_N f \|_2  \le C_3 \| \nabla_N \Delta_N f \|_2$, an inequality similar to~\eqref{H3 est-9-9}. Therefore, a substitution of~\eqref{convergence-2}-\eqref{convergence-4} and \eqref{convergence-5-4} into~\eqref{convergence-1} yields
\begin{eqnarray}
   &&
     \frac{1}{4 \dt} ( \| L_N^\frac12 e^{n+1} \|_2^2 - \| L_N^\frac12 e^n \|_2^2
   + \| L_N^\frac12 ( 2 e^{n+1} - e^n ) \|_2^2
   - \| L_N^\frac12 ( 2 e^n - e^{n-1} ) \|_2^2  )  \nonumber
\\
  &&
    + \frac58 \| \frac{ \frac32 e^{n+1} - 2 e^n + \frac12 e^{n-1} }{\dt}  \|_{-1,N}^2
    \le   - \langle \frac{ \tilde{r}^{n+1}}{\sqrt{ E_{1,N} (\hat{\phi}^{n+1} ) }}
   N_N (\hat{\phi}^{n+1}) ,
      \frac{ \frac32 e^{n+1} - 2 e^n + \frac12 e^{n-1} }{\dt}  \rangle   \nonumber
\\
  && \qquad \qquad \qquad \qquad \qquad \qquad \qquad
    + \tilde{C}_8 \| \nabla_N \Delta_N \hat{e}^{n+1}  \|_2^2
    +   2 \| \tau_\phi^{n+1}  \|_{-1,N}^2 .   \label{convergence-6-1}
\end{eqnarray}
On the other hand, the original error evolutionary equation (\ref{consistency-2}a) gives
\begin{equation}
   \nabla_N (-\Delta_N)^{-1} ( \frac{ \frac32 e^{n+1} - 2 e^n + \frac12 e^{n-1} }{\dt} )
  =   - \nabla_N (  L_N e^{n+1} + {\cal NLE}_1 + {\cal NLE}_2 + {\cal NLE}_3  - (\Delta_N)^{-1} \tau_\phi^{n+1} )  .  \label{convergence-6-2}
\end{equation}
In turn, an application of quadratic inequality implies that
\begin{eqnarray}
  &&
   \| \frac{ \frac32 e^{n+1} - 2 e^n + \frac12 e^{n-1} }{\dt} \|_{-1,N}^2  \nonumber
\\
  &\ge&
     \frac12   \| \nabla_N L_N e^{n+1} \|_2^2
  - 2 \| \nabla_N ( {\cal NLE}_1 + {\cal NLE}_2 + {\cal NLE}_3 - (\Delta_N)^{-1} \tau_\phi^{n+1} )  \|_2^2  \nonumber
\\
  &\ge&
     \frac12   \| \nabla_N L_N e^{n+1} \|_2^2
  - 4 ( \| \nabla_N ( {\cal NLE}_1 + {\cal NLE}_2 + {\cal NLE}_3 ) \|_2^2
   + \| \tau_\phi^{n+1}  \|_{-1,N}^2  )   \nonumber
\\
  &\ge&
     \frac12   \| \nabla_N L_N e^{n+1} \|_2^2
  - 12 ( \| \nabla_N {\cal NLE}_1 \|_2^2 + \| \nabla_N {\cal NLE}_2 \|_2^2
  + \| \nabla_N {\cal NLE}_3 \|_2^2 )
   - 4 \| \tau_\phi^{n+1}  \|_{-1,N}^2   \nonumber
\\
  &\ge&
     \frac12   \| \nabla_N L_N e^{n+1} \|_2^2
  - 12 ( \tilde{C}_7^2  ( \tilde{r}^{n+1} )^2
  + (  \tilde{C}_5^2 C_3^2 + \tilde{C}_6^2  )
  \| \nabla_N \Delta_N \hat{e}^{n+1}  \|_2^2 )
   - 4 \| \tau_\phi^{n+1}  \|_{-1,N}^2   ,
   \label{convergence-6-3}
\end{eqnarray}
with the estimates~\eqref{convergence-5-1}-\eqref{convergence-5-3} recalled. Going back~\eqref{convergence-6-1}, we arrive at
\begin{eqnarray}
   &&
     \frac{1}{4 \dt} ( \| L_N^\frac12 e^{n+1} \|_2^2 - \| L_N^\frac12 e^n \|_2^2
   + \| L_N^\frac12 ( 2 e^{n+1} - e^n ) \|_2^2
   - \| L_N^\frac12 ( 2 e^n - e^{n-1} ) \|_2^2  )
   \nonumber
\\
  &&
    + \frac{5}{16}   \| \nabla_N L_N e^{n+1} \|_2^2   \nonumber
\\
  &\le&
      - \langle \frac{ \tilde{r}^{n+1}}{\sqrt{ E_{1,N} (\hat{\phi}^{n+1} ) }}
   N_N (\hat{\phi}^{n+1}) ,
      \frac{ \frac32 e^{n+1} - 2 e^n + \frac12 e^{n-1} }{\dt}  \rangle   \nonumber
\\
  &&
    +  12 \tilde{C}_7^2  ( \tilde{r}^{n+1} )^2
     + 7 \tilde{C}_8 \| \nabla_N \Delta_N \hat{e}^{n+1}  \|_2^2
    +   6 \| \tau_\phi^{n+1}  \|_{-1,N}^2 .   \label{convergence-6-4}
\end{eqnarray}

Taking a discrete inner product of (\ref{consistency-2}b) with $2 \tilde{r}^{n+1}$ gives
\begin{eqnarray}
  \frac{1}{\dt} ( \frac32 \tilde{r}^{n+1} - 2 \tilde{r}^n + \frac12 \tilde{r}^{n-1} ) \cdot
  2 \tilde{r}^{n+1} = \frac{\tilde{r}^{n+1} }{ \sqrt{ E_{1,N} (\hat{\phi}^{n+1})}} \langle N_N ( \hat{\phi}^{n+1}) , \frac{\frac32 e^{n+1} - 2 e^n + \frac12 e^{n-1}}{\dt} \rangle
   \nonumber
\\
  + \frac{\tilde{r}^{n+1} }{ \sqrt{ E_{1,N} (\hat{\phi}^{n+1})}} \langle \tilde{N}^{n+1} , \frac{\frac32 \Phi^{n+1} - 2 \Phi^n + \frac12 \Phi^{n-1}}{\dt} \rangle  \nonumber
\\
  - B^{n+1} \tilde{E}_1^{n+1} \tilde{r}^{n+1}
  \langle N_N ( \hat{\Phi}^{n+1})  , \frac{\frac32 \Phi^{n+1} - 2 \Phi^n
  + \frac12 \Phi^{n-1}}{\dt} \rangle
+ 2 \tau_r^{n+1} \cdot \tilde{r}^{n+1} .   \label{convergence-7-1}
\end{eqnarray}
The estimate for the temporal stencil term is similar to that of~\eqref{eng stab-4-2}:
\begin{eqnarray}
   &&
   2 \tilde{r}^{n+1} ( \frac32 \tilde{r}^{n+1} - 2 \tilde{r}^n + \frac12 \tilde{r}^{n-1} )
    \nonumber
\\
  &=&
     \frac12 ( | \tilde{r}^{n+1} |^2 - | \tilde{r}^n |^2 + | 2 \tilde{r}^{n+1} - \tilde{r}^n |^2
   - | 2 \tilde{r}^n - \tilde{r}^{n-1} |^2
   + | \tilde{r}^{n+1} - 2 \tilde{r}^n + \tilde{r}^{n-1}  |^2  )  .   \label{convergence-7-2}
\end{eqnarray}
The inner product associated with the truncation error could be controlled via Cauchy inequality:
\begin{eqnarray}
   2 \tau_r^{n+1}  \cdot \tilde{r}^{n+1}
   \le  | \tau_r^{n+1} |^2 + | \tilde{r}^{n+1} |^2  .
     \label{convergence-7-3}
\end{eqnarray}
 The first nonlinear inner product on the right hand side is kept. The second and third nonlinear inner product terms could be analyzed as follows
 \begin{eqnarray}
   \hspace{-0.35in} &&
   \frac{\tilde{r}^{n+1} }{ \sqrt{ E_{1,N} (\hat{\phi}^{n+1})}} \langle \tilde{N}^{n+1} , \frac{\frac32 \Phi^{n+1} - 2 \Phi^n + \frac12 \Phi^{n-1}}{\dt} \rangle  \nonumber
\\
  \hspace{-0.35in} &\le&
  | \tilde{r}^{n+1} | \cdot | \Omega |^{-\frac12}
  \cdot \| \nabla_N \tilde{N}^{n+1} \|_2
  \cdot \| \frac{\frac32 \Phi^{n+1} - 2 \Phi^n + \frac12 \Phi^{n-1}}{\dt} \|_{-1,N}
  \nonumber
\\
  \hspace{-0.35in} &\le&
  | \tilde{r}^{n+1} | \cdot | \Omega |^{-\frac12}
  \cdot \tilde{C}_4 \| \nabla_N \Delta_N \hat{e}^{n+1} \|_2  \cdot C C^*
  \nonumber
\\
  \hspace{-0.35in} &\le&
 \tilde{C}_9  | \tilde{r}^{n+1} | \cdot  \| \nabla_N \Delta_N \hat{e}^{n+1} \|_2
 \le \frac{\tilde{C}_9}{2} (   | \tilde{r}^{n+1} |^2
 +  \| \nabla_N \Delta_N \hat{e}^{n+1} \|_2^2 ) ,  \, \, \,
\tilde{C}_9 =C  \tilde{C}_4 C^* | \Omega |^{-\frac12} ,
   \label{convergence-7-4}
\\
  \hspace{-0.35in} &&
  - B^{n+1} \tilde{E}_1^{n+1} \tilde{r}^{n+1}
  \langle N_N ( \hat{\Phi}^{n+1})  , \frac{\frac32 \Phi^{n+1} - 2 \Phi^n
  + \frac12 \Phi^{n-1}}{\dt} \rangle    \nonumber
\\
  \hspace{-0.35in} &\le&
   | B^{n+1} | \cdot | \tilde{E}_1^{n+1} | \cdot | \tilde{r}^{n+1} |
   \cdot \| \nabla_N N_N ( \hat{\Phi}^{n+1}) \|_2
  \cdot \| \frac{\frac32 \Phi^{n+1} - 2 \Phi^n + \frac12 \Phi^{n-1}}{\dt} \|_{-1,N}
  \nonumber
\\
  \hspace{-0.35in} &\le&
 \frac12 | \Omega |^{-\frac32}  \cdot \tilde{C}_2 \| \nabla_N \hat{e}^{n+1} \|_2
 \cdot  | \tilde{r}^{n+1} | \cdot \tilde{C}_3  \cdot C C^*
  \nonumber
\\
  \hspace{-0.35in} &\le&
 \tilde{C}_{10}  | \tilde{r}^{n+1} | \cdot  \| \nabla_N \Delta_N \hat{e}^{n+1} \|_2
 \le \frac{\tilde{C}_{10}}{2} (   | \tilde{r}^{n+1} |^2
 +  \| \nabla_N \Delta_N \hat{e}^{n+1} \|_2^2 ) ,  \, \, \,
\tilde{C}_{10} =C  \tilde{C}_2 \tilde{C}_3 C_3 C^* | \Omega |^{-\frac32} ,
   \label{convergence-7-5}
\end{eqnarray}
with repeated application of the preliminary estimates~\eqref{pre est-0-1}-\eqref{pre est-0-6} in Lemma~\ref{lem: pre est}. Subsequently, a substitution of~\eqref{convergence-7-2}-\eqref{convergence-7-5} into~\eqref{convergence-7-1} yields
\begin{eqnarray}
  &&
   \frac{1}{2 \dt} ( | \tilde{r}^{n+1} |^2 - | \tilde{r}^n |^2 + | 2 \tilde{r}^{n+1} - \tilde{r}^n |^2
   - | 2 \tilde{r}^n - \tilde{r}^{n-1} |^2
   + | \tilde{r}^{n+1} - 2 \tilde{r}^n + \tilde{r}^{n-1}  |^2  )    \nonumber
\\
  &\le&
   \frac{\tilde{r}^{n+1} }{ \sqrt{ E_{1,N} (\hat{\phi}^{n+1})}} \langle N_N ( \hat{\phi}^{n+1}) , \frac{\frac32 e^{n+1} - 2 e^n + \frac12 e^{n-1}}{\dt} \rangle  \nonumber
\\
  &&
   + \frac{\tilde{C}_9 + \tilde{C}_{10}}{2} (   | \tilde{r}^{n+1} |^2
 +  \| \nabla_N \Delta_N \hat{e}^{n+1} \|_2^2 )
 + | \tilde{r}^{n+1} |^2  + | \tau_r^{n+1} |^2 .
    \label{convergence-7-6}
\end{eqnarray}

Finally, a combination of~\eqref{convergence-6-4} and \eqref{convergence-7-6} results in
\begin{eqnarray}
   &&
     \frac{1}{4 \dt} ( \| L_N^\frac12 e^{n+1} \|_2^2 - \| L_N^\frac12 e^n \|_2^2
   + \| L_N^\frac12 ( 2 e^{n+1} - e^n ) \|_2^2
   - \| L_N^\frac12 ( 2 e^n - e^{n-1} ) \|_2^2  )   \nonumber
\\
  &&
    + \frac{5}{16}   \| \nabla_N L_N e^{n+1} \|_2^2   \nonumber
\\
  &&
  +  \frac{1}{2 \dt} ( | \tilde{r}^{n+1} |^2 - | \tilde{r}^n |^2 + | 2 \tilde{r}^{n+1} - \tilde{r}^n |^2
   - | 2 \tilde{r}^n - \tilde{r}^{n-1} |^2
   + | \tilde{r}^{n+1} - 2 \tilde{r}^n + \tilde{r}^{n-1}  |^2  )    \nonumber
\\
  &\le&
     \tilde{C}_{11}   | \tilde{r}^{n+1} |^2
     + \tilde{C}_{12} \| \nabla_N \Delta_N \hat{e}^{n+1}  \|_2^2
    +   6 \| \tau_\phi^{n+1}  \|_{-1,N}^2
    + | \tau_r^{n+1} |^2  .   \label{convergence-8-1}
\end{eqnarray}
with $\tilde{C}_{11} = 12 ( \tilde{C}_7^2 + \tilde{C}_9^2 + \tilde{C}_{10}^2 ) + 1$,
$\tilde{C}_{12} = 7 \tilde{C}_8 + \frac{\tilde{C}_9 + \tilde{C}_{10}}{2}$. In particular, we notice that the first nonlinear error inner product terms have been cancelled; this subtle fact has played a crucial role in the analysis. In addition, the following inequalities are observed:
\begin{eqnarray}
   \| \nabla_N \Delta_N \hat{e}^{n+1}  \|_2^2
   &=& \| \nabla_N \Delta_N ( 2 e^n - e^{n-1} )  \|_2^2
   \le 6 \| \nabla_N \Delta_N e^n  \|_2^2
   + 3 \| \nabla_N \Delta_N e^{n-1}   \|_2^2  ,  \label{convergence-8-2}
\\
  \| \nabla_N \Delta_N e^k  \|_2^2
   &\le& \| \Delta_N e^k \|_2^\frac43
   \cdot \| \nabla_N \Delta_N^2 e^k \|_2^\frac23    \, \, \,
   \mbox{(by~\eqref{pre est-0-5} ) }     \nonumber
\\
  &\le&
  \frac{4 \sqrt{3}}{3} \tilde{C}_{12}^\frac12 \| \Delta_N e^k \|_2^2
   + \frac{1}{36 \tilde{C}_{12}} \| \nabla_N \Delta_N^2 e^k \|_2^2 ,  \quad
   k = n,  n-1 ,    \label{convergence-8-3}
\end{eqnarray}
in which Young's inequality has been applied in the last step of~\eqref{convergence-8-3}. This in turn leads to
\begin{eqnarray}
   \tilde{C}_{12} \| \nabla_N \Delta_N \hat{e}^{n+1}  \|_2^2
   &\le& \tilde{C}_{12} ( 6 \| \nabla_N \Delta_N e^n  \|_2^2
   + 3 \| \nabla_N \Delta_N e^{n-1}   \|_2^2 )   \nonumber
\\
  &\le&
     4 \sqrt{3} \tilde{C}_{12}^\frac32 ( 2 \| \Delta_N e^n \|_2^2
   + \| \Delta_N e^{n-1} \|_2^2 )   \nonumber
\\
  &&
   +  \frac16 \| \nabla_N \Delta_N^2 e^n \|_2^2
   +  \frac{1}{12} \| \nabla_N \Delta_N^2 e^{n-1} \|_2^2 .  \label{convergence-8-4}
\end{eqnarray}
Going back~\eqref{convergence-8-1}, we arrive at (by denoting $\tilde{C}_{13} = 8 \sqrt{3} \tilde{C}_{12}^\frac32$)
\begin{eqnarray}
   &&
     \frac{1}{\dt} ( {\cal H}^{n+1} - {\cal H}^n )
    + \frac{5}{16}   \| \nabla_N L_N e^{n+1} \|_2^2   \nonumber
\\
  &\le&
     \tilde{C}_{11}   | \tilde{r}^{n+1} |^2
     + \tilde{C}_{13} (  \| \Delta_N e^n \|_2^2  + \| \Delta_N e^{n-1} \|_2^2 )
     +  \frac16 \| \nabla_N \Delta_N^2 e^n \|_2^2
   +  \frac{1}{12} \| \nabla_N \Delta_N^2 e^{n-1} \|_2^2   \nonumber
\\
  &&
    +   6 \| \tau_\phi^{n+1}  \|_{-1,N}^2
    + \| \tau_r^{n+1} \|_2^2  ,  \label{convergence-8-5}
\\
  &&  \mbox{with} \, \, \,
{\cal H}^{n+1} :=  \frac14 ( \| L_N^\frac12 e^{n+1} \|_2^2
   + \| L_N^\frac12 ( 2 e^{n+1} - e^n ) \|_2^2  )
  +  \frac12 ( | \tilde{r}^{n+1} |^2  + | 2 \tilde{r}^{n+1} - \tilde{r}^n |^2 ) .
    \label{convergence-8-6}
	\end{eqnarray}
Moreover, the following inequalities are recalled
\begin{eqnarray}
   \| \Delta_N e^k \|_2^2 \le \| L_N^\frac12 e^k \|_2^2
   \le 4 {\cal H}^k  ,   \quad
   | r^k |^2 \le 2 {\cal H}^k , \quad
   \| \nabla_N \Delta_N^2 e^k \|_2^2
   \le \| \nabla_N L_N e^k \|_2^2  ,  \, \, \, \mbox{(by~\eqref{pre est-0-5}) } ,
    \label{convergence-8-7}
\end{eqnarray}
for $k=n+1, n, n-1$. Then we obtain the following estimate
\begin{eqnarray}
   &&
     \frac{1}{\dt} ( {\cal H}^{n+1} - {\cal H}^n )
    + \frac{5}{16}   \| \nabla_N L_N e^{n+1} \|_2^2   \nonumber
\\
  &\le&
     ( 2 \tilde{C}_{11}  + 4 \tilde{C}_{13} ) (  {\cal H}^{n+1} + {\cal H}^n + {\cal H}^{n-1}  )
     +  \frac16 \| \nabla_N L_N e^n \|_2^2
   +  \frac{1}{12} \| \nabla_N L_N e^{n-1} \|_2^2   \nonumber
\\
  &&
    +   6 \| \tau_\phi^{n+1}  \|_{-1,N}^2
    + | \tau_r^{n+1} |^2  .  \label{convergence-8-8}
	\end{eqnarray}
Therefore, with an application of discrete Gronwall inequality, and making use of the fact that $\| \tau_\phi^{n+1} \|_{-1,N} ,  \| \tau_r^{n+1} \| \le C (\dt^2 + h^m)$, we arrive at
\begin{eqnarray}
   {\cal H}^{n+1}  + \frac{1}{16} \dt \sum_{j=1}^{n+1} \| \nabla_N L_N e^j \|_2^2
   \le \hat{C} ( \dt^4 + h^{2m}) ,   \label{convergence-8-9}
\end{eqnarray}
with $\hat{C}$ independent on $\dt$ and $h$. In turn, the desired convergence estimate is available
\begin{eqnarray}
   \| \Delta_N e^{n+1} \|_2  + \Bigl( \dt \sum_{j=1}^{k+1} \| \nabla_N \Delta_N^2 e^j \|_2^2  \Bigr)^\frac12 \le C \hat{C}^\frac12 ( \dt^2 + h^m) ,   \label{convergence-8-10}
\end{eqnarray}
in which the estimates~\eqref{convergence-8-7} has been recalled. This completes the proof of Theorem~\ref{thm:convergence}.
	\end{proof}

\begin{rem}
In an earlier error analysis work~\cite{LiXL19} for the SAV scheme applied to the Cahn-Hilliard flow, a linear refinement requirement for the time step size, $\dt \le C h$, has to be imposed for the convergence estimate, since an inverse inequality has to be applied in the error estimate in the energy norm. In contrast, we have derived a higher order $H^3$ bound for the numerical solution, which in turn leads to an unconditional convergence estimate (no scaling law constraint between $\dt$ and $h$) for the proposed SAV scheme.
\end{rem}

\begin{rem}
With the help of the optimal rate convergence estimate in the $\ell^\infty (0,T; H_N^2)$ norm, we are able to derive a sharper bound for the original energy functional. In more details, the error estimate~\eqref{SPFC-convergence-0} leads to the following inequalities
\begin{equation}
\begin{aligned}
  &
  \| \nabla_N \phi^m \|_4^4 - \| \nabla_N \Phi^m \|_4^4
  \le 4 ( \max ( \| \nabla_N \phi^m \|_4 ,  \| \nabla_N \Phi^m \|_4 ) )^3
  \| \nabla_N e^m \|_4
\\
  &  \qquad \qquad \qquad \qquad
  \le C \tilde{C}_1^3  \| \Delta_N e^m \|_2 \le  C \tilde{C}_1^3 \hat{C} ( \dt^2 + h^m) ,
\\
  &
  \| \Delta_N \phi^m \|_2^2 - \| \Delta_N \Phi^m \|_2^2
  \le 2  \max ( \| \Delta_N \phi^m \|_2 ,  \| \Delta_N \Phi^m \|_2 )
  \| \Delta_N e^m \|_2
  \le  C \tilde{C}_1 \hat{C} ( \dt^2 + h^m)  ,
\\
  &
  \| \phi^m \|_2^2 - \| \Phi^m \|_2^2 , \, \, \,
  \| \nabla_N \phi^m \|_2^2 - \| \nabla_N \Phi^m \|_2^2
    \le  C \tilde{C}_1 \hat{C} ( \dt^2 + h^m)  ,    \quad
    \mbox{(similar analysis)} ,
\end{aligned}
\label{SPFC-energy-refined-1}
\end{equation}
in which the discrete Sobolev inequality~\eqref{embedding-0} and the uniform-in-time $H_N^2$ bound~\eqref{SPFC-H2 stab-0} have been extensively applied. Then we get
\begin{equation}
\begin{aligned}
  &
   | E_N (\phi^m) - E_N (\Phi^m) |  \le  C ( \tilde{C}_1 + \tilde{C}_1^3)
   \hat{C} ( \dt^2 + h^m)  ,
\\
  &
  E_N (\Phi^m) - E_N (\Phi (t^m) )  = O (h^m) ,  \quad
  E_N (\Phi (t^m) )  - E (\Phi (t^m) ) = O (h^m) ,  \quad
\\
  &
  E (\Phi (t^m) ) \le  E (\Phi (t^0) ) := C_0 ,   \quad
  \mbox{so that} \, \, \, E_N (\phi^m) \le C_0 + C ( \tilde{C}_1^3 + 1 )
   \hat{C} ( \dt^2 + h^m)  \le C_0 + 1  ,
\end{aligned}
\label{SPFC-energy-refined-2}
\end{equation}
provided that $\dt$ and $h$ are sufficiently small. Of course, it is a much sharper estimate than the uniform-in-time bound~\eqref{SPFC-energy-0}, in which $\tilde{C}_1^*$ depends on $\tilde{C}_0$ in a quadratic way. On the other hand, it is notice that the refined estimate~\eqref{SPFC-energy-refined-2} is local-in-time, since the convergence constant $\hat{C}$ depends on the final time, while the rough bound~\eqref{SPFC-energy-0}  turns out to be a global quantity.
\end{rem}

\begin{rem}
In a recent work~\cite{cheng2019d}, a modified BDF scheme is applied to the SPFC equation~\eqref{equation-SPFC} in the primitive formulation, the energy stability and optimal rate convergence estimates have been provided as well. Due to the primitive formulation involved, the highly complicated 4-Laplacian term has be to treated implicitly to ensure an unconditional energy stability. This leads to a nonlinear system to be solved at each time step, and the corresponding computational cost for the nonlinear system is approximately three times the linear SAV scheme proposed in this work, with the same spatial and temporal resolution. As a result, the computational efficiency has been improved in this SAV approach.

In addition, only the $\ell^\infty (0, T; \ell^2) \cap \ell^2 (0, T; H_N^3)$ error estimate has been performed in the existing work~\cite{cheng2019d}, in comparison with the $\ell^\infty (0, T; H_N^2) \cap \ell^2 (0, T; H_N^5)$ error estimate provided in this article. In turn, the uniform-in-time $H_N^3$ bound of the numerical solution, as established in~\eqref{H3 est-0} (Theorem~\ref{thm: H3 est}), is not needed in~\cite{cheng2019d}. Therefore, this article has provided further technical tools for the theoretical analysis of higher order stability estimate and convergence analysis, in comparison with~\cite{cheng2019d}.
\end{rem}

	\section{Numerical results}
	\label{sec:numerical results}

	\subsection{Convergence test for the numerical scheme}

In this subsection we perform some numerical experiments to verify the accuracy order of the proposed SAV scheme. 
To test the convergence rate, we choose the following exact solution for \eqref{equation-SPFC} on the square domain $\Omega=(0, 1)^2$:
	\begin{equation}
	\label{eqn:init1}
\phi_e(x,y,t) = \frac{1}{2\pi}\sin (2 \pi x) \cos (2 \pi y) \cos(t).
	\end{equation}
We set $a =0.975$, 
and the final time is taken as $T=1$.


	\begin{figure}
	\begin{center}
\includegraphics[width=3.0in]{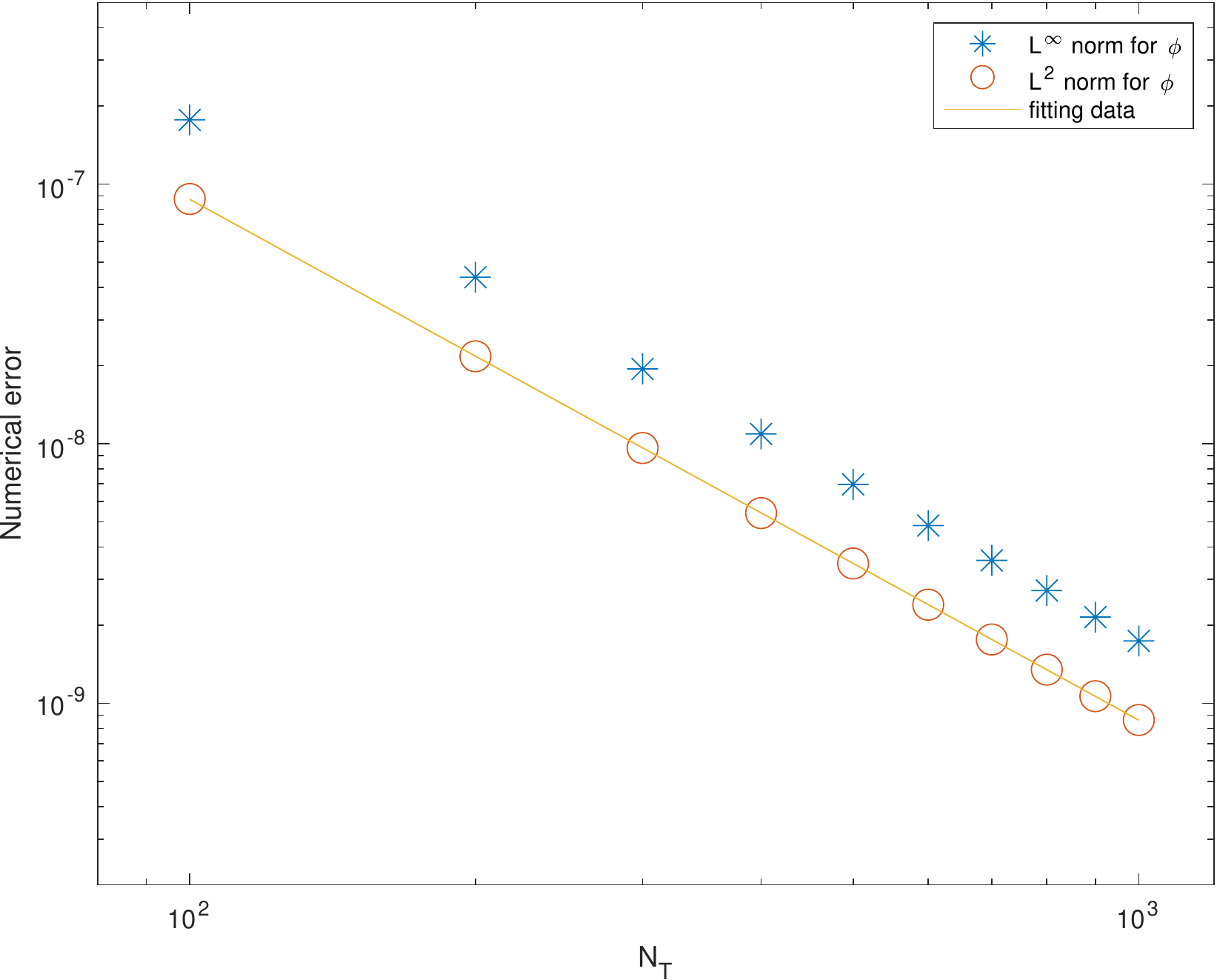}
	\end{center}
\caption{The discrete $\ell^2$ and $\ell^\infty$ numerical errors vs. temporal resolution  $N_T$ for $N_T = 100:100:1000$, with a spatial resolution $N=128$. The data lie roughly on curves $CN_T^{-2}$, for appropriate choices of $C$, confirming the full second-order accuracy of the scheme.}
	\label{fig1}
	\end{figure}

To make $\Phi$ satisfy the original PDE \eqref{equation-SPFC}, we have to add an artificial, time-dependent forcing term. Then the proposed third order BDF-type  scheme~\eqref{scheme-SAV-1} can be implemented to solve for the original PDE. To explore the temporal accuracy, we fix the spatial resolution as $N=128$ so that the numerical error is dominated by the temporal ones. We compute solutions with a sequence of time step sizes, $\dt = \frac{T}{N_T}$, with $N_T=100$ to $N_T=1000$ in increments of 100, and the same final time $T=1$. Fig.~\ref{fig1} shows the discrete $\ell^2$ norms of the errors between the numerical and exact solutions, computed by the proposed numerical scheme~\eqref{scheme-SAV-1}. 
The fitted line displayed in Figure~\ref{fig1} shows an approximate slope of -2, which in turn verifies a nice second order temporal convergence order, in both the discrete $\ell^2$ and $\ell^\infty$ norms.

\subsection{Numerical simulation of square symmetry patterns} 

The $4$-Laplacian term in \eqref{equation-SPFC} gives preference to rotationally invariant patterns with square symmetry. In this subsection, we perform two-dimensional numerical simulations showing the emergence of these patterns. The rest of the parameters are given by $a = 0.5$ and $\Omega =(0,L)^2$, with $L=100$. The initial data for the simulations are given by
\begin{equation}
	\label{eqn:init2}
 \phi^0_{i,j}=0.05\cdot(2r_{i,j}-1),
	\end{equation}
where the $r_{i,j}$ are uniformly distributed random numbers in $[0, 1]$. 
For the temporal step size $\dt$, we use increasing values of $\dt$ in the time evolution: $\dt = 0.01$ on the time interval $[0,1000]$ and $\dt = 0.02$ on the time interval $[1000, 21000]$. Whenever a new time step size is applied, we initiate the two-step numerical scheme by  taking $\phi^{-1} = \phi^0$, with the initial data $\phi^0$ given by the final time output of the last time period. The time snapshots of the evolution by using the given parameters are presented in Figures~\ref{fig:long-time-spfc} (one nucleation site). These tests confirm the emergence of the rotationally invariant square-symmetry patterns in the density field.

	\begin{figure}[ht]
	\begin{center}
	
	\begin{subfigure}{0.48\textwidth}
\includegraphics[height=0.48\textwidth,width=0.48\textwidth]{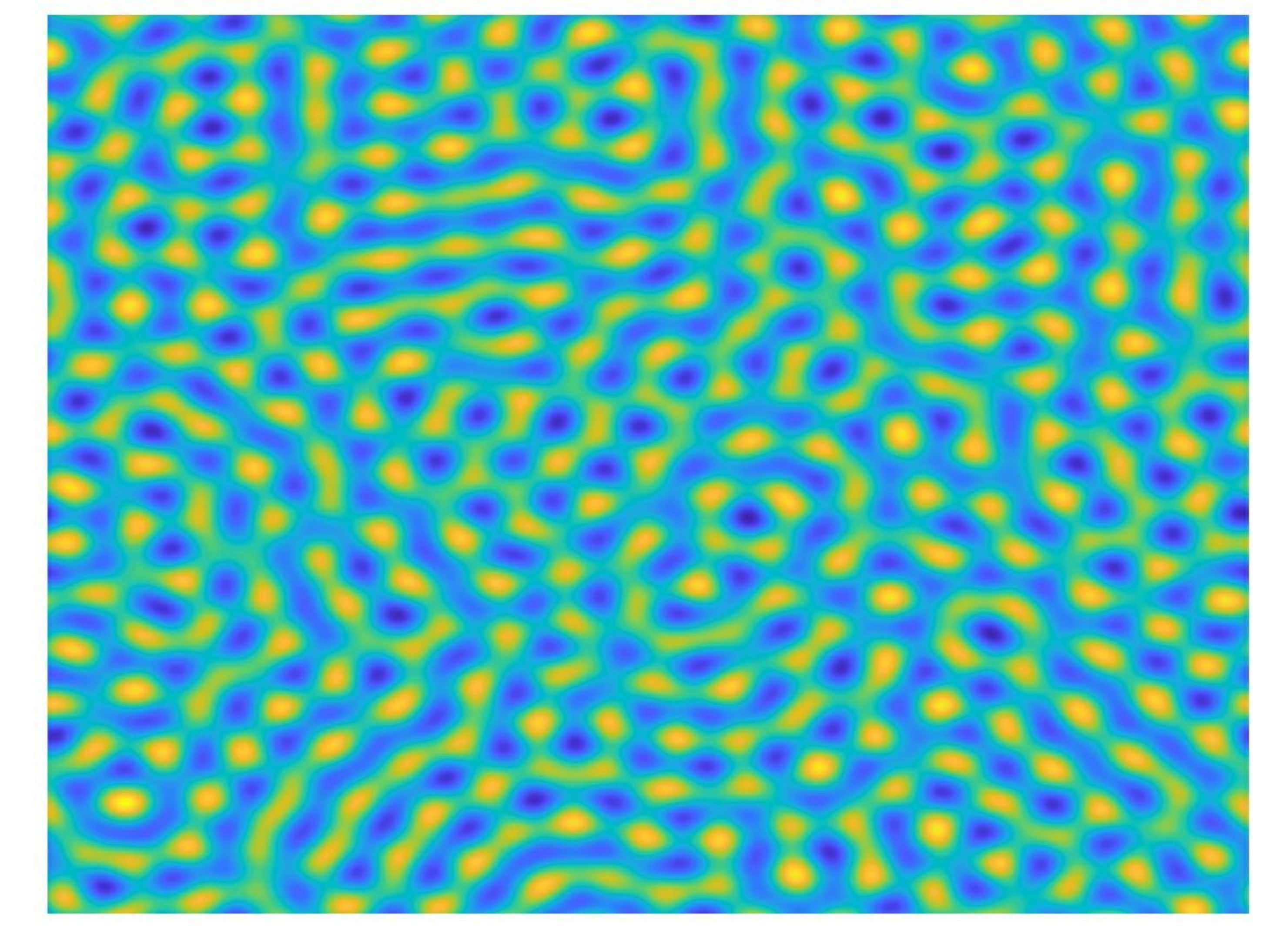}
			\includegraphics[height=0.48\textwidth,width=0.48\textwidth]{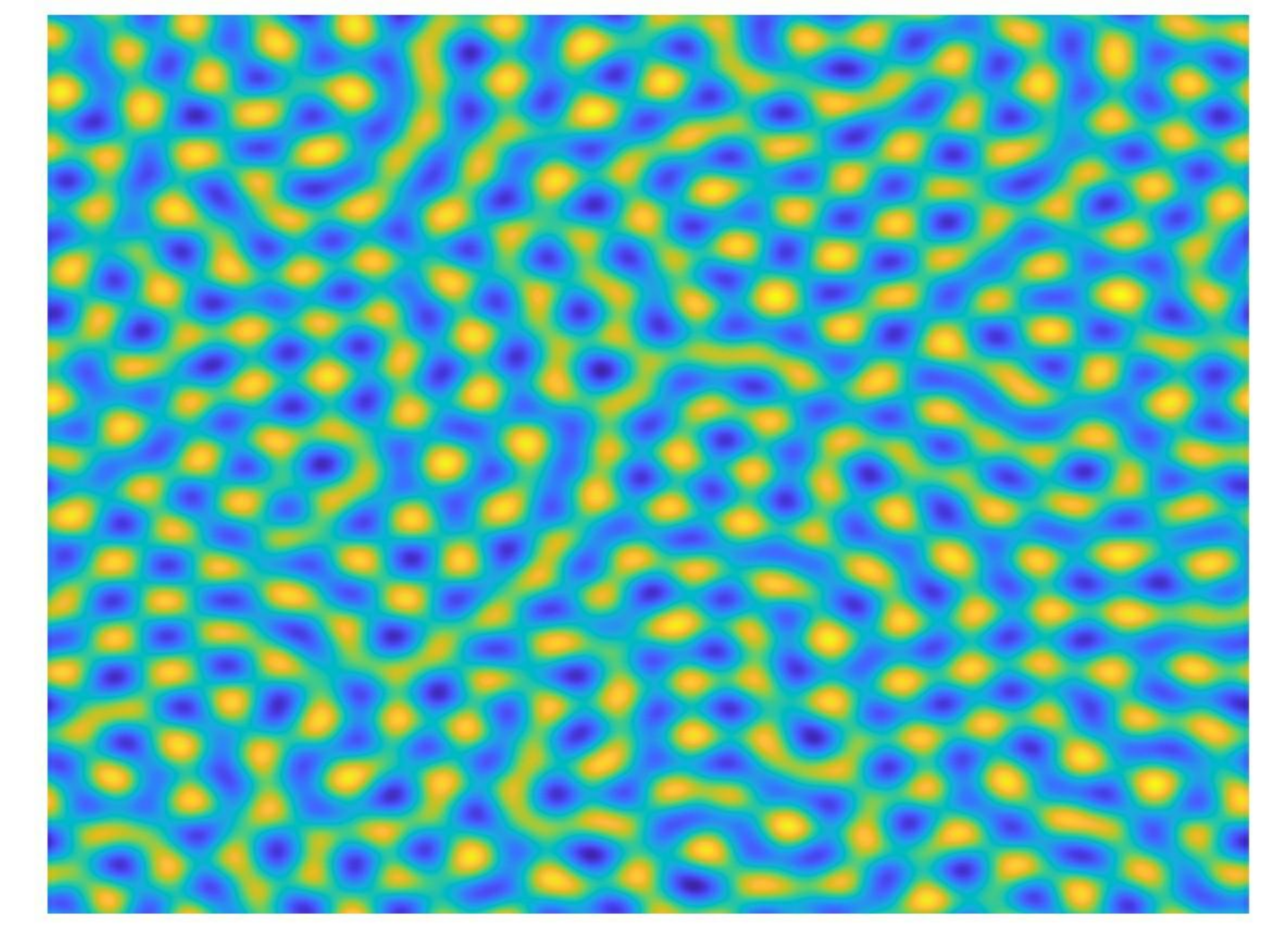}
		
\caption*{$t=10, 20$}

	\end{subfigure}
	\begin{subfigure}{0.48\textwidth}
	
\includegraphics[height=0.48\textwidth,width=0.48\textwidth]{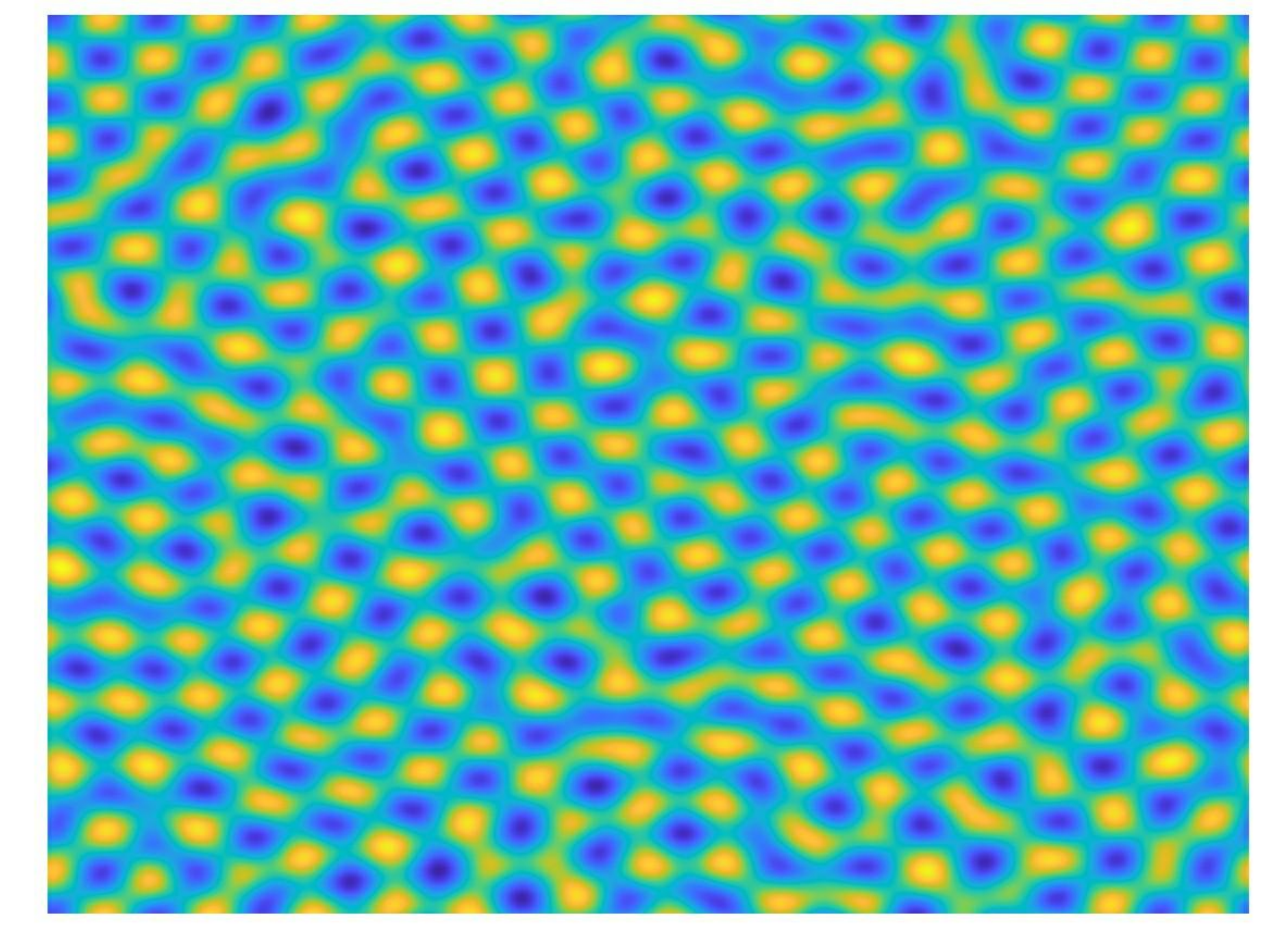}
			\includegraphics[height=0.48\textwidth,width=0.48\textwidth]{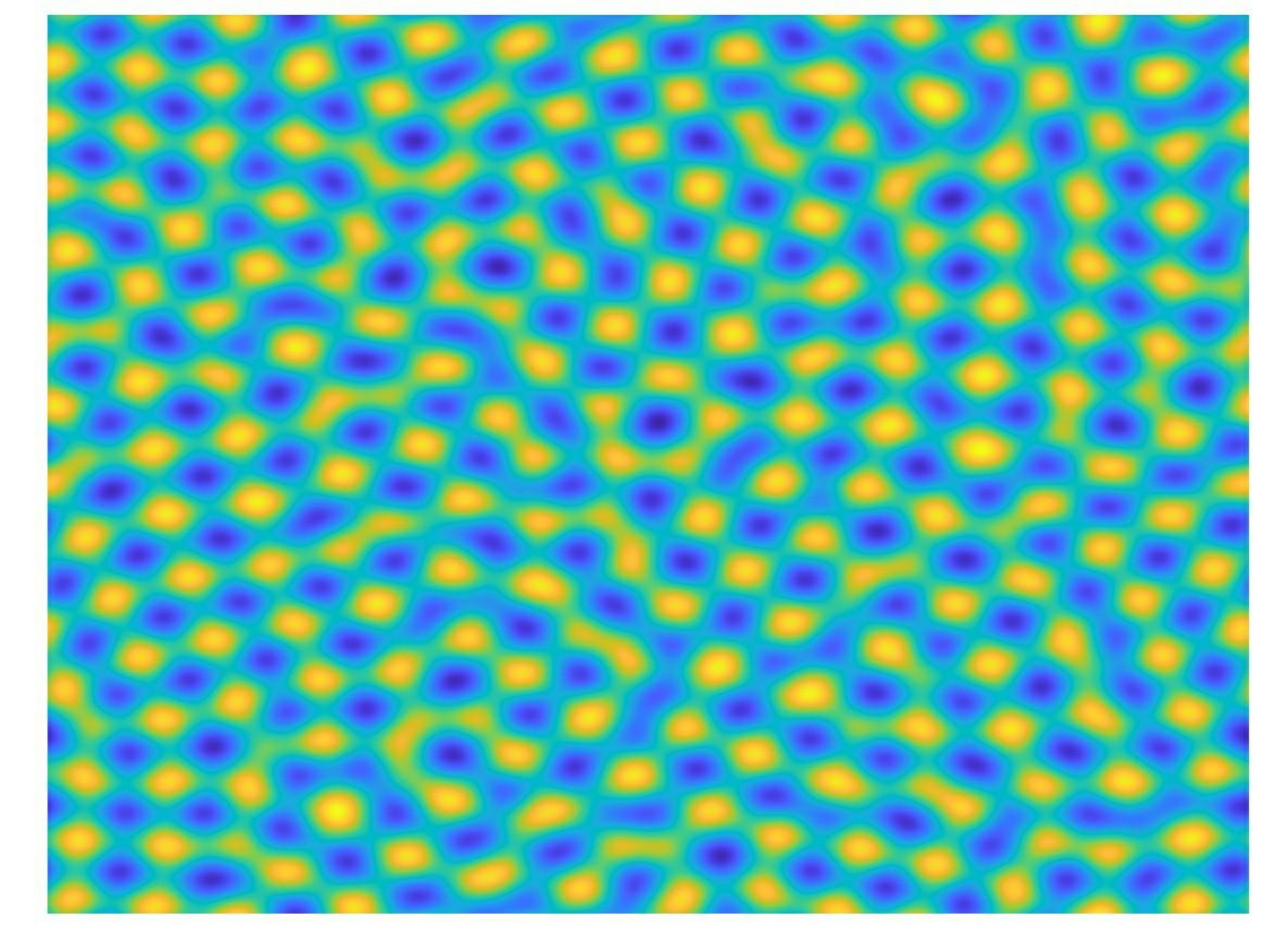}
			
\caption*{$t=40, 80$}

	\end{subfigure}
	\begin{subfigure}{0.48\textwidth}
			\includegraphics[height=0.48\textwidth,width=0.48\textwidth]{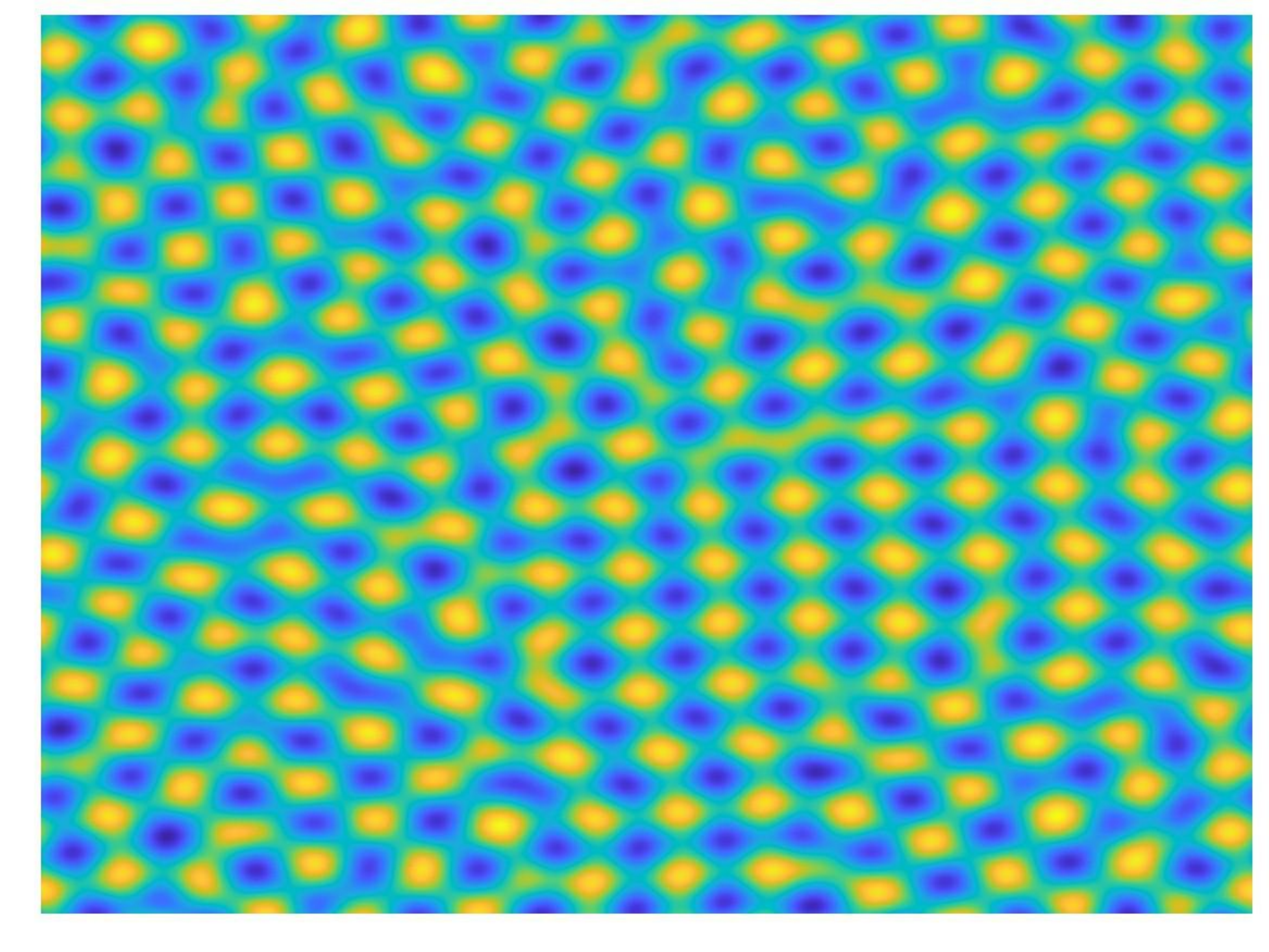}
			\includegraphics[height=0.48\textwidth,width=0.48\textwidth]{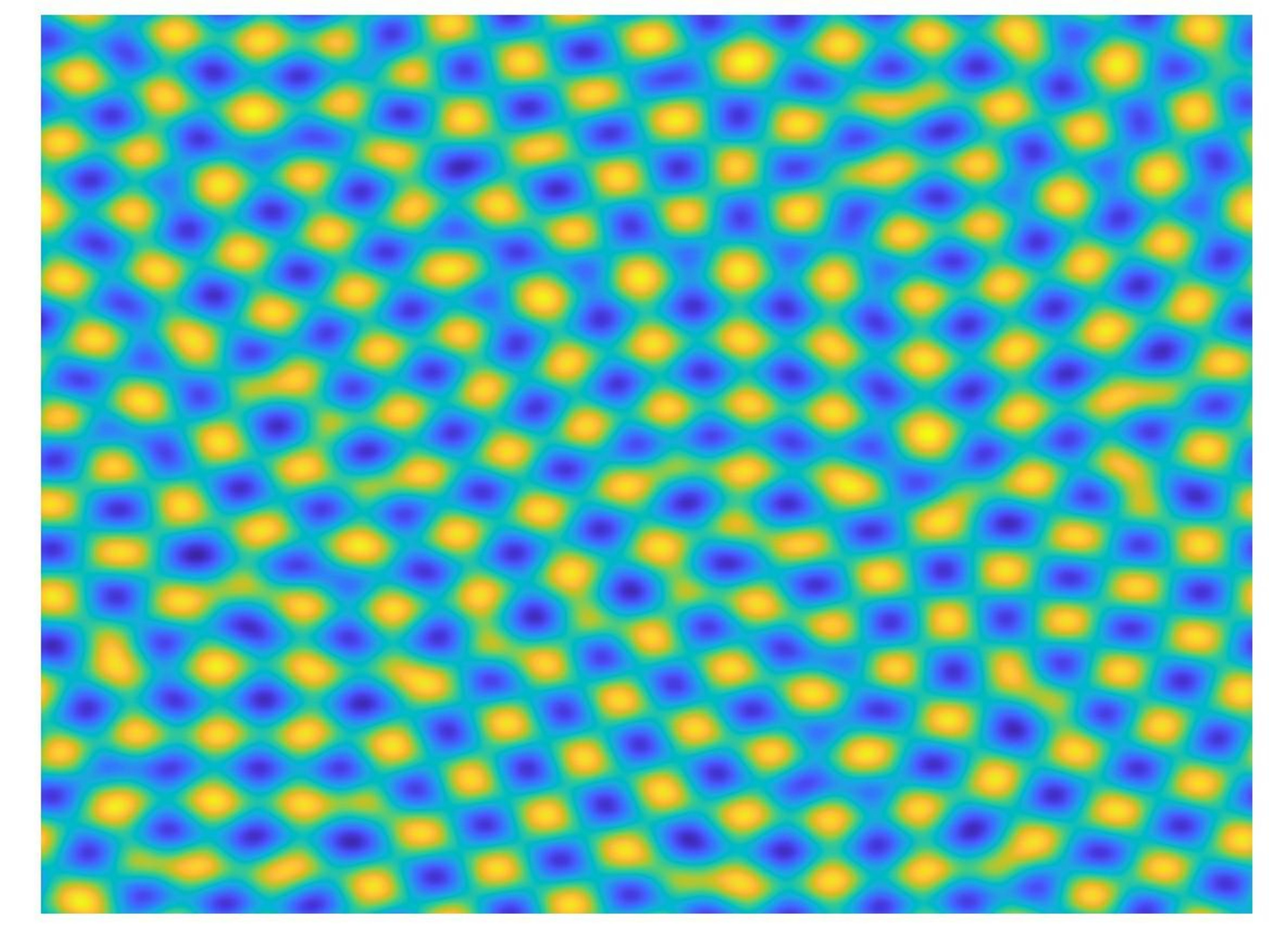}
			
\caption*{$t=100, 200$}

	\end{subfigure}
	\begin{subfigure}{0.48\textwidth}
			\includegraphics[height=0.48\textwidth,width=0.48\textwidth]{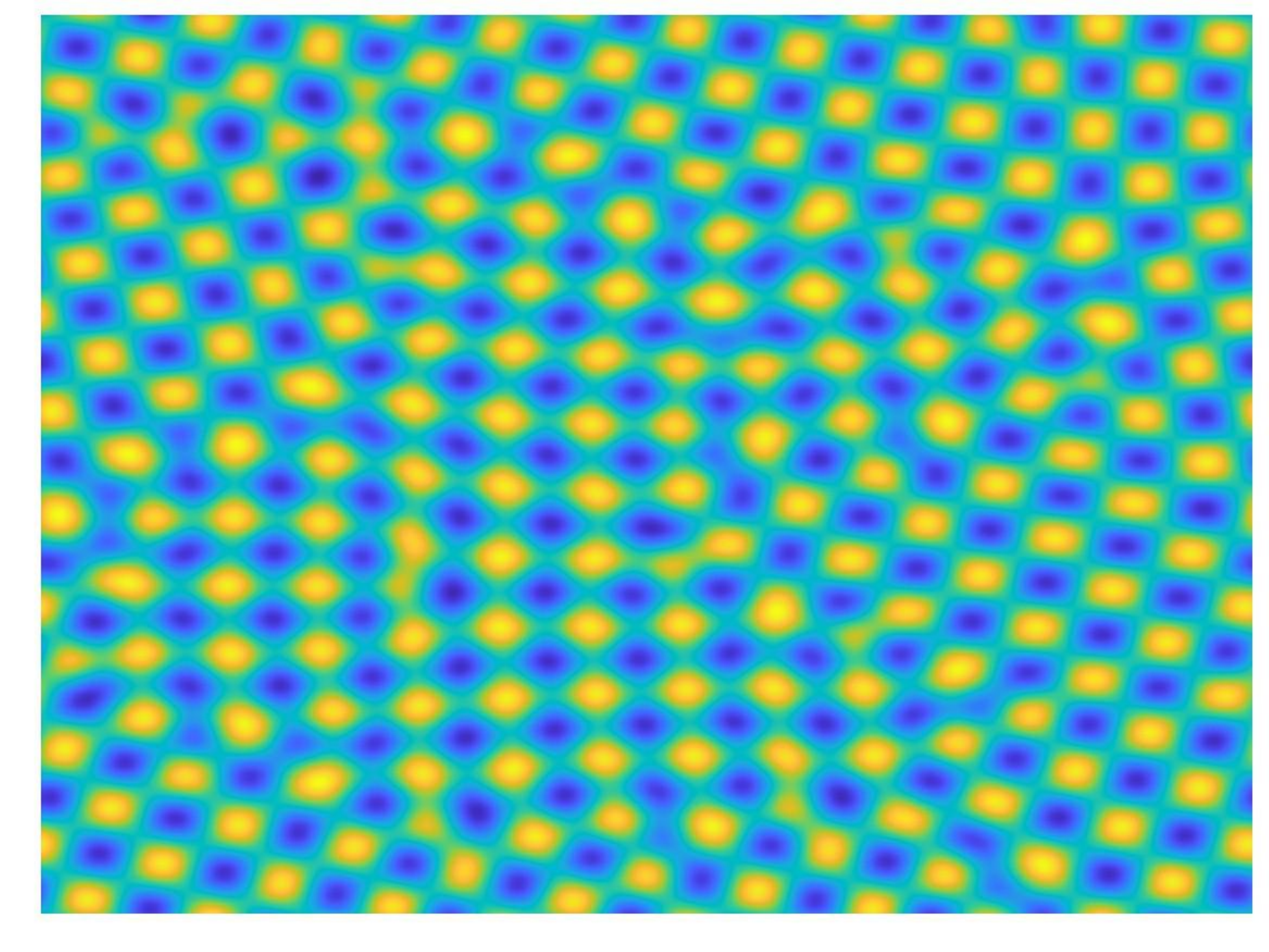}
			\includegraphics[height=0.48\textwidth,width=0.48\textwidth]{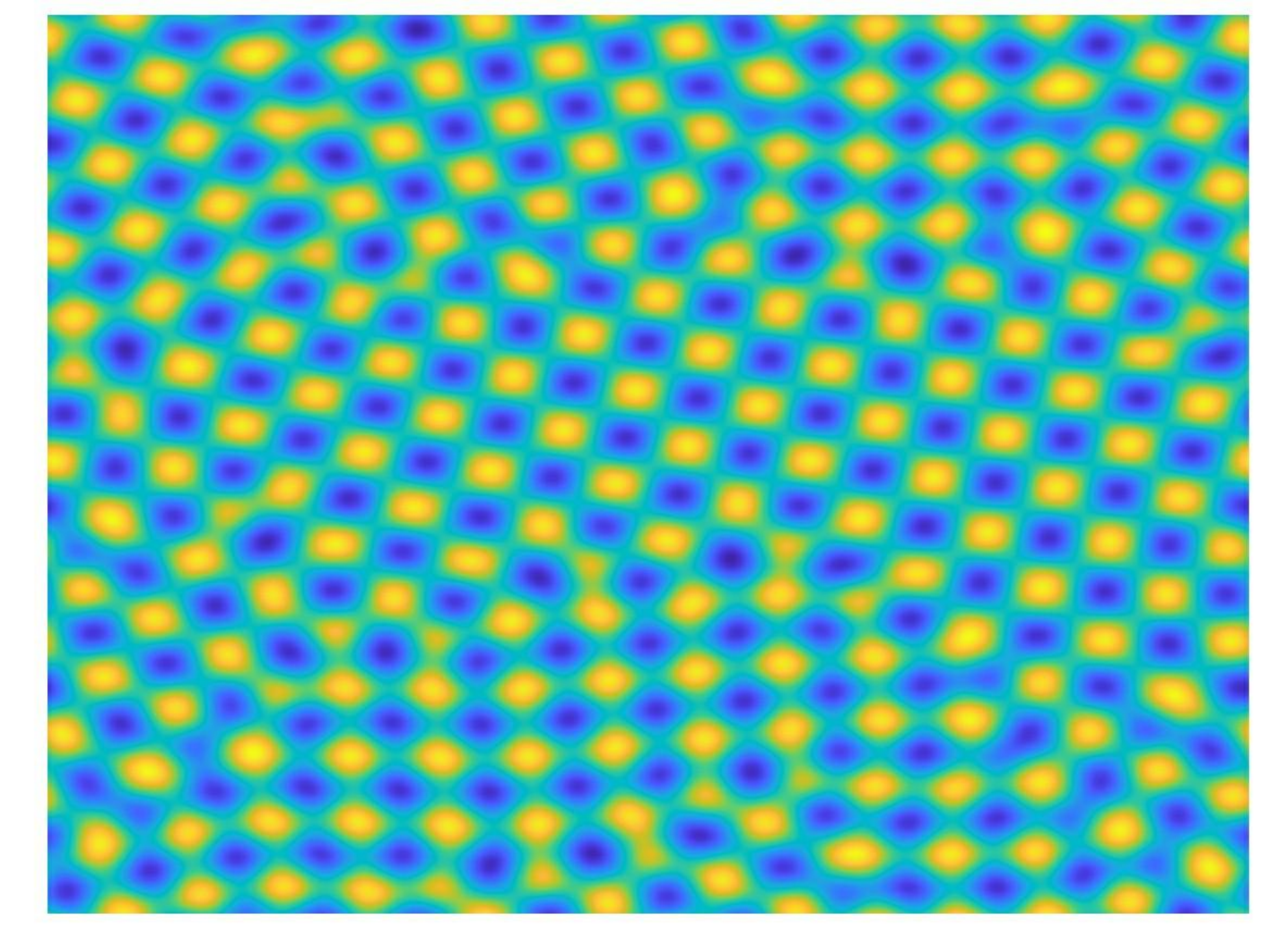}
			
\caption*{$t=500,1000$}
		
	\end{subfigure}
	\begin{subfigure}{0.48\textwidth}
			\includegraphics[height=0.48\textwidth,width=0.48\textwidth]{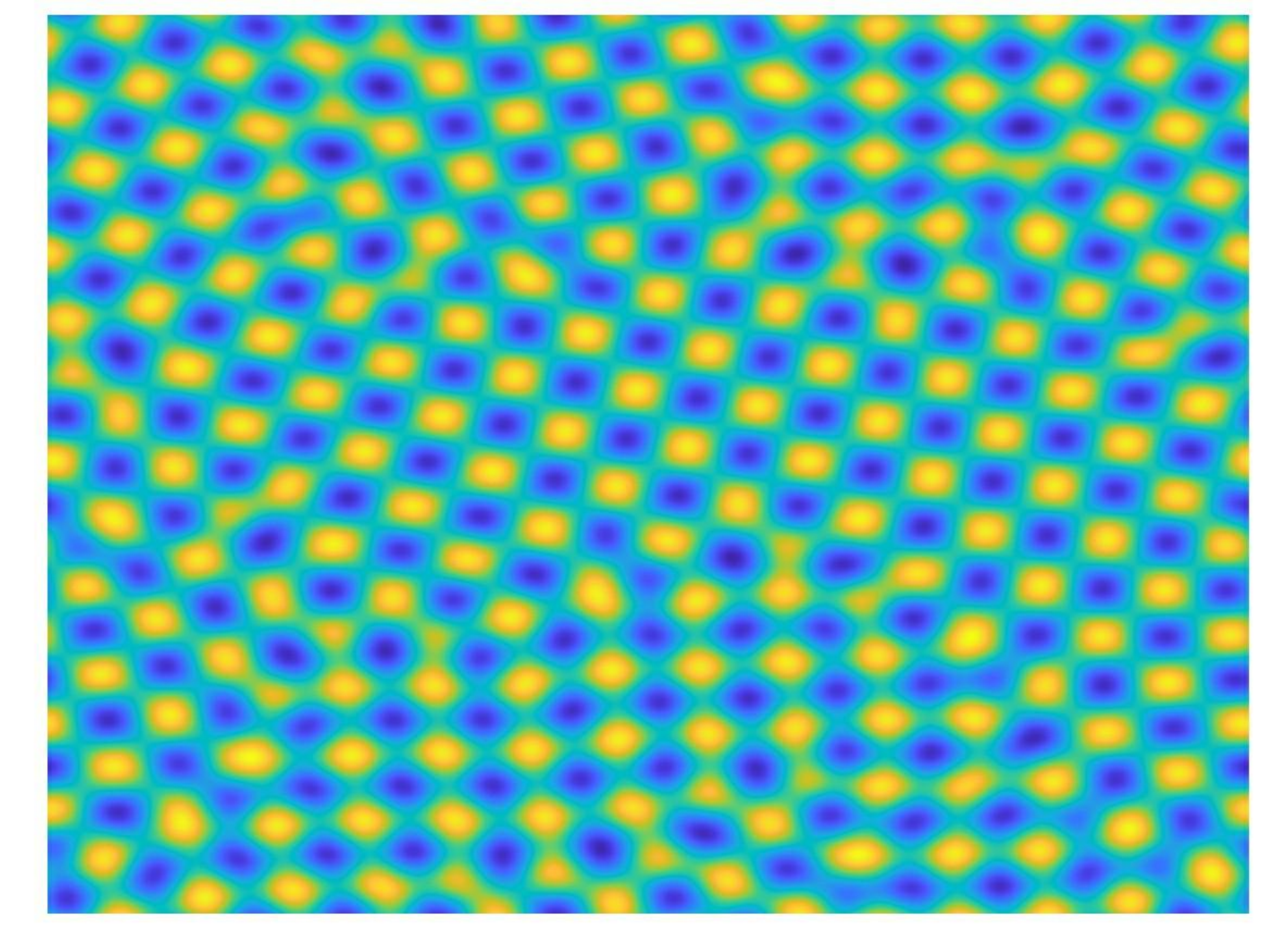}
			\includegraphics[height=0.48\textwidth,width=0.48\textwidth]{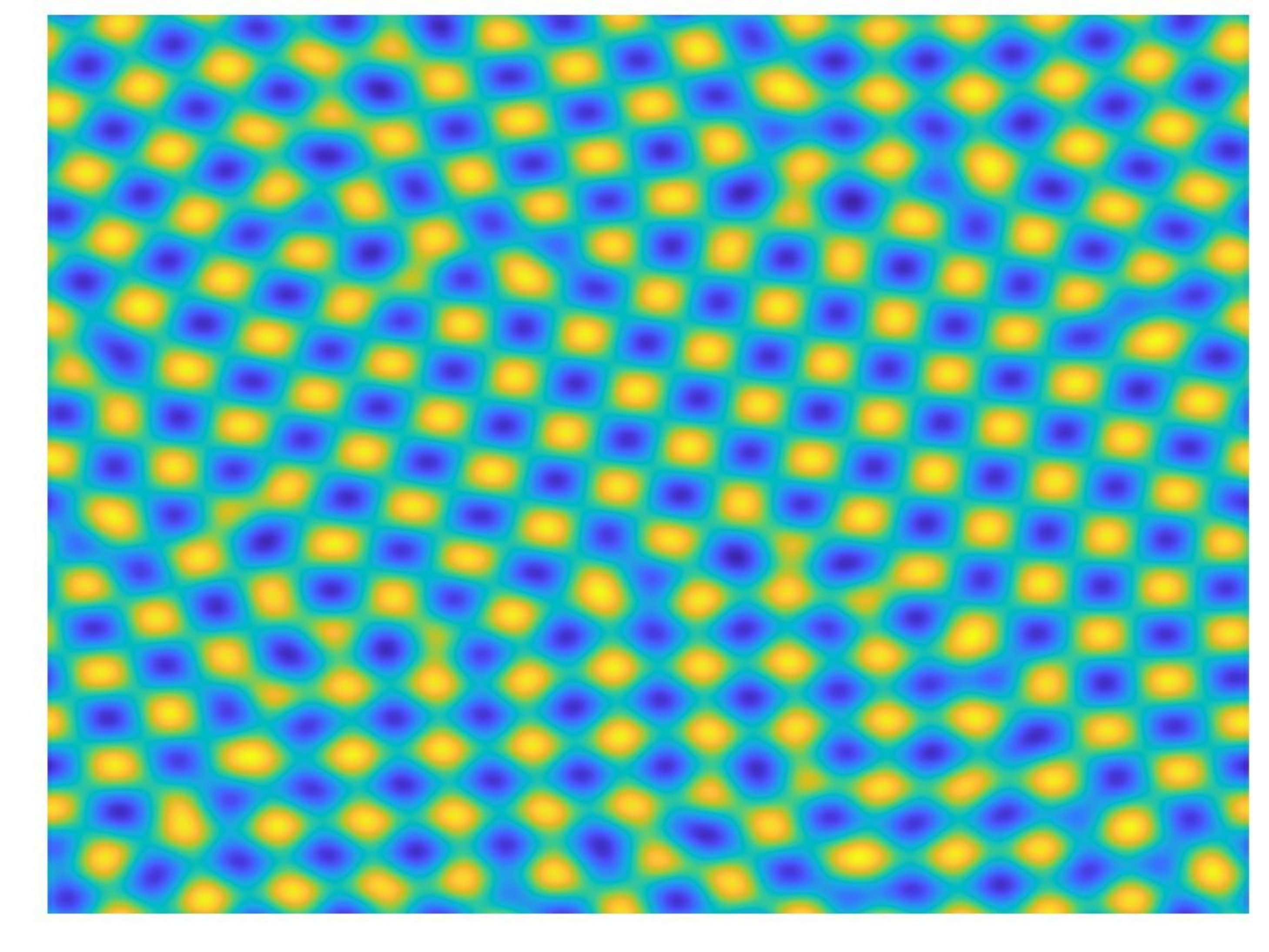}
			
\caption*{$t=3000, 9000$}
		
	\end{subfigure}
	\begin{subfigure}{0.48\textwidth}
			\includegraphics[height=0.48\textwidth,width=0.48\textwidth]{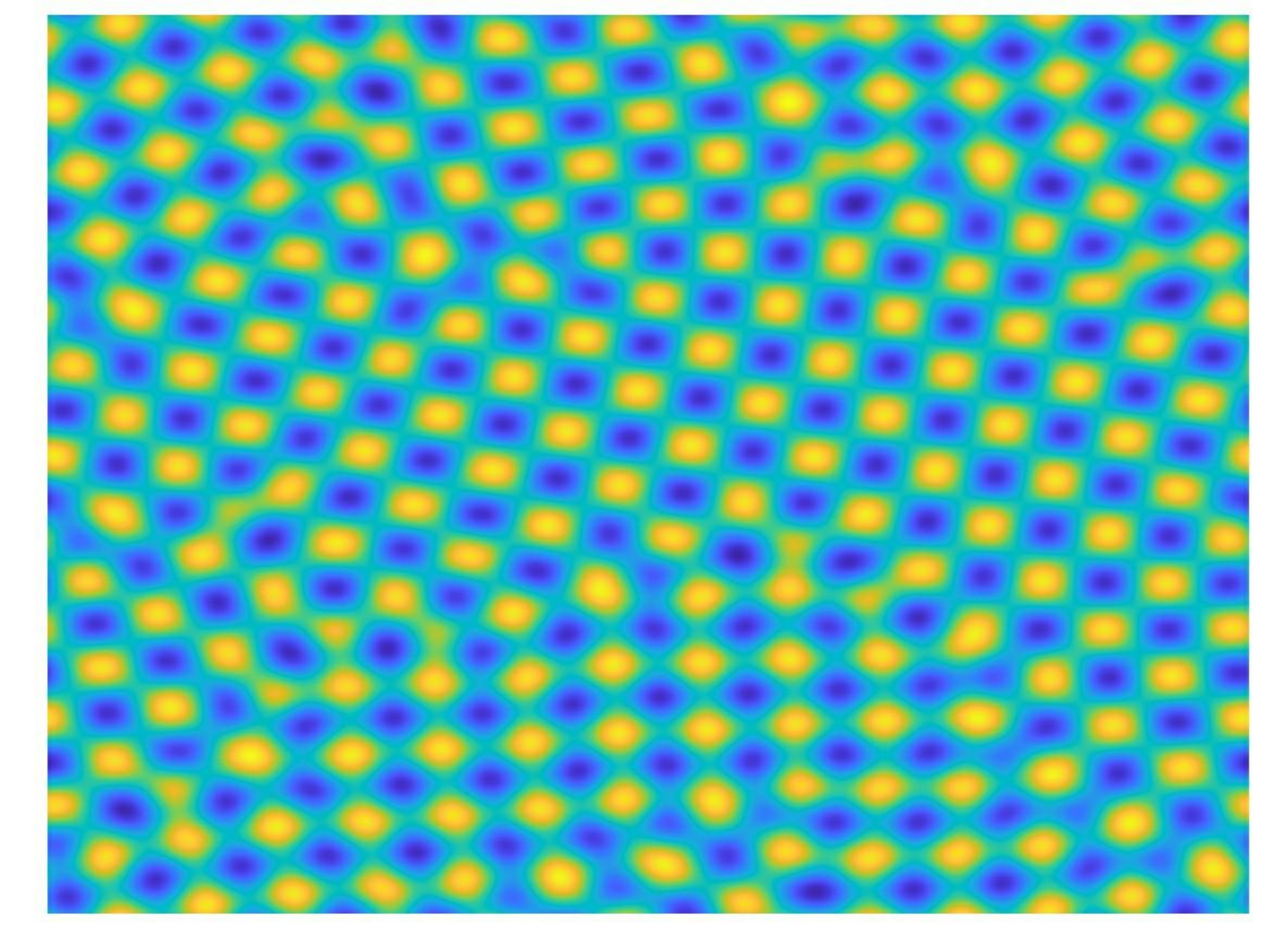}
			\includegraphics[height=0.48\textwidth,width=0.48\textwidth]{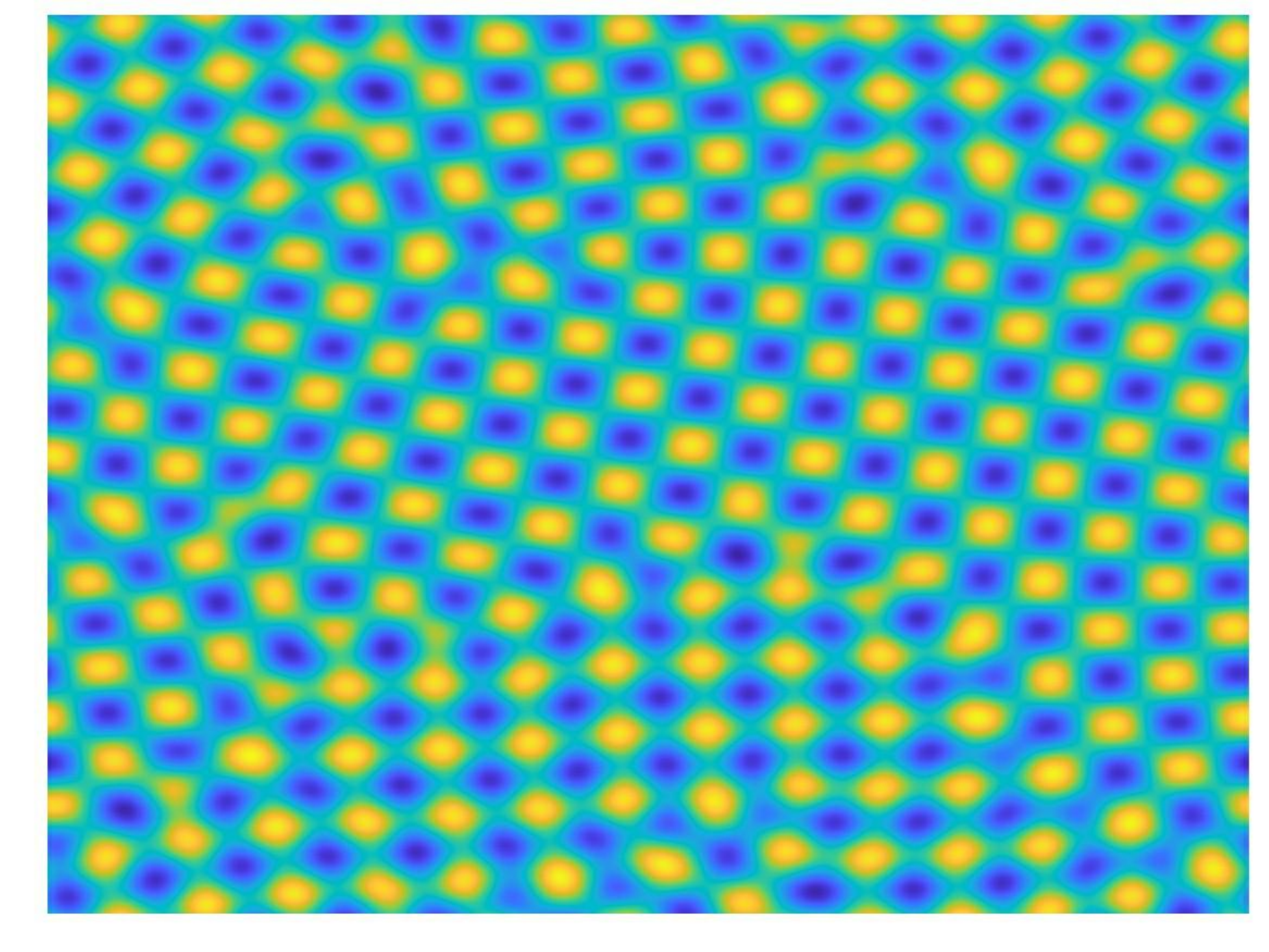}
			
\caption*{$t=15000, 21000$}
		
	\end{subfigure}
\caption{Time snapshots of the evolution for squared phase field crystal model, with random initial perturbation. The time sequence for the snapshots is set as $t=10, 20, 40, 80, 100, 200, 500, 1000, 3000, 9000, 15000~ \text{and}~ 21000$. The parameters are $a =0.5, \Omega=[0, 100]^2$.}
	\label{fig:long-time-spfc}
	\end{center}
	\end{figure}

To illustrate the energy stability property of the proposed numerical scheme, we display the energy evolution of the one nucleation site example, up to $t=1000$, in 
The solid and dotted plots stand for the time evolution of the original energy functional and the SAV-introduced energy functional, given by formula~\eqref{energy-discrete-spectral} and $\breve{E} (\phi, r) = \frac{a}{2} \| \phi \|_2^2 + \frac12 \|  \Delta_N \phi \|_2^2 + | r |^2$, respectively. The plots overlap so that differences are indistinguishable, and the energy dissipation property is clearly observed in the numerical simulation. This shows that the SAV approach is indeed an accurate numerical approximation to the original physical model.

	\begin{figure}
	\centering
	\includegraphics[width=4.0in]{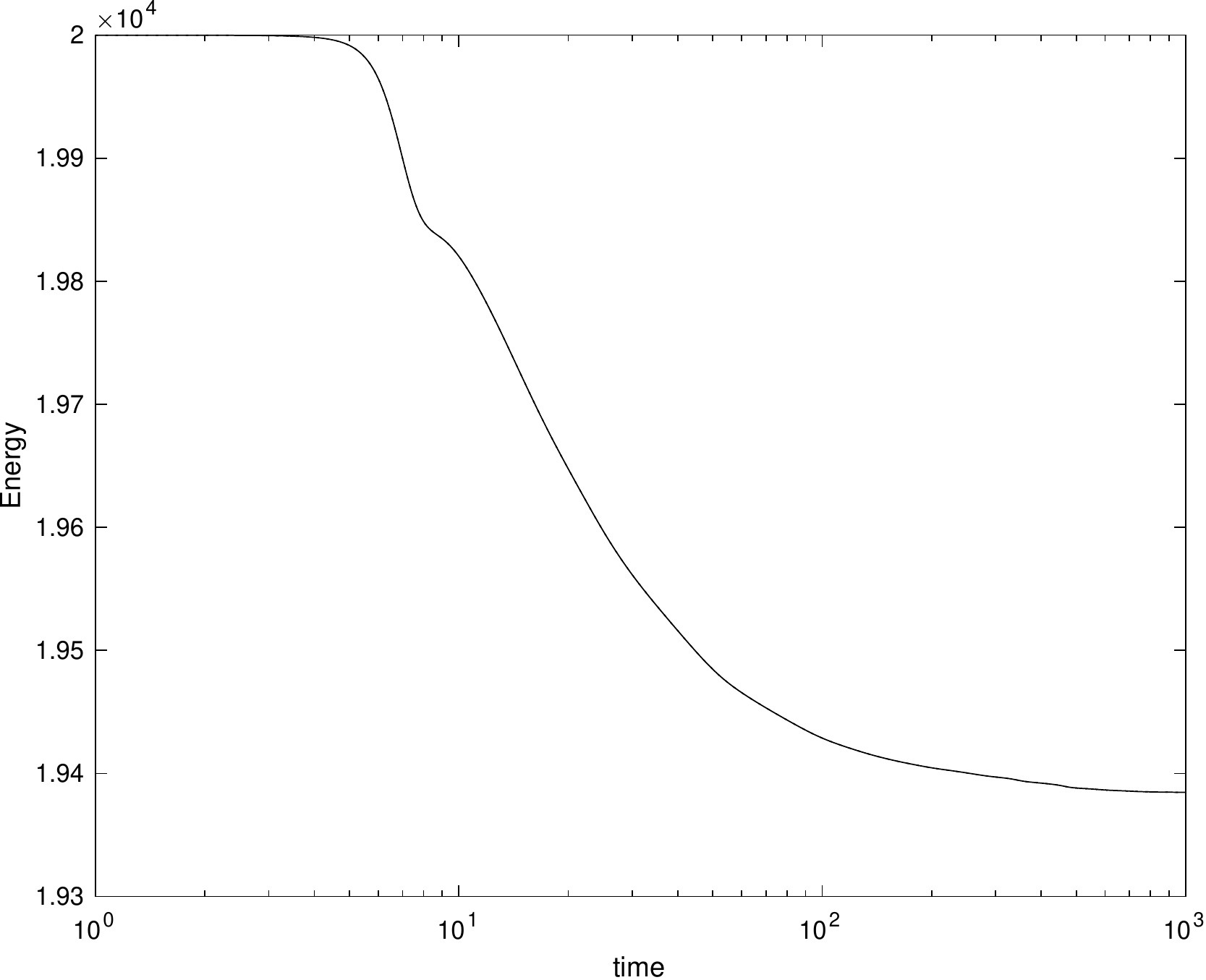}
\caption{Semi-log plot of the temporal evolution the energy up to $t=1000$. The solid and dotted plots stand for the time evolution of the original energy functional and the SAV-introduced energy functional, respectively. The plots overlap so that differences are indistinguishable.}
	\label{fig:energy evolution}
       \end{figure}

Our numerical experiments have also demonstrated that, the SAV numerical scheme works well for the smooth gradient flows, such as the numerical example presented above, with a mild amplitude of random initial perturbation. On the other hand, if a more singular perturbation is included at the initial data, such a nucleation at the center $(50, 50)$, with magnitude of 10, a direct application of the SAV numerical scheme is not able to create a reasonable numerical solution. 
Meanwhile, extensive numerical experiments have demonstrated that, a stabilized SAV scheme, with an inclusion of artificial regularization in the form of $- A \dt \Delta_N (\phi^{n+1} - \phi^n)$ (such as the one in the existing work~\cite{cheng2019d}), could overcome such a rough initial data difficulty and produce much nicer numerical results. In general, we conclude that, for smooth gradient flows in which there is no sharp gradient, the SAV scheme has greatly improved the computational efficiency. For a challenging numerical example in which an initial singularity is included, the stabilized SAV approach will overcome the subtle numerical difficulties and be able to enhance the scientific computing performances.

\section{Concluding remarks} \label{sec:conclusion}

In this article, we have proposed and analyzed an scalar auxiliary variable (SAV)-based numerical scheme  for the square phase field crystal (SPFC) equation, a gradient flow to model the crystal growth. An appropriate decomposition for the physical energy functional is formulated, so that the nonlinear energy part has a well-established global lower bound, and the rest terms lead to constant-coefficient diffusion terms with positive eigenvalues. This overcomes a key difficulty in the application of SAV idea to the SPFC model. In turn, the resulting numerical scheme could be very efficiently implemented by constant-coefficient Poisson-like type solvers (via FFT), and energy stability is established by introducing an auxiliary variable. As a result of this modified energy stability, a uniform in time $H^2$ bound is available for the numerical solution. In addition, we are able to derive a uniform in time $H^3$ bound for the numerical solution, with the help of discrete Sobolev embedding techniques. Such an $H^3$ bound for the numerical solution plays an essential role in the optimal rate convergence analysis in the energy norm, i.e., the error estimate in the $\ell^\infty (0,T; H^2) \cap \ell^2 (0, T; H^5)$ space. A few numerical experiments are presented to demonstrate the efficiency and accuracy of the proposed scheme, including the numerical accuracy test and numerical simulations of square symmetry patterns.

	\section*{Acknowledgements}
This work is supported in part by NSFC 11971047 (Q. Huang) and NSF DMS-2012669 (C.~Wang). 

	\appendix

\section{Proof of Proposition~\ref{prop:elliptic regularity}}
	\label{proof:Prop 2.4}

Due to the periodic boundary condition for $f$ and its cell-centered representation, it has a corresponding discrete Fourier transformation, as the form given by~\eqref{spectral-coll-1}:
	\begin{eqnarray}
f_{i,j,k} = \sum_{\ell,m,n=-K}^{K} \hat{f}_{\ell,m,n}^N \exp \left( 2 \pi {\rm i} ( \ell x_i + m y_j + n z_k ) \right) .
   \label{def:Fourier-1}
	\end{eqnarray}
Then we make its extension to a continuous function:
	\begin{equation}
	\label{def:extension-1}
f_N (x,y,z) = \sum^{K}_{\ell,m,n=-K} \hat{f}^N_{\ell,m,n} \exp \left( 2 \pi {\rm i} ( \ell x + m y + n z ) \right)  .
	\end{equation}
We denote a discrete grid function, $g := {\cal D}_x f$, at a point-wise level. Since $f$ corresponds to $f_N \in {\cal B}^K$ (the space of trigonometric polynomials of degree at most $K$), an application of Parseval identity implies that
\begin{equation}
\begin{aligned}
  &
  \| \nabla_N \Delta_N f \|_2^2  = \| \nabla \Delta f_N \|^2
  =   \sum^{K}_{\ell,m,n=-K}   \lambda_{\ell, m, n}^6 | \hat{f}^N_{\ell,m,n}  |^2  ,
\\
  &
  \| \Delta_N^3 f \|_2^2  = \| \Delta^3 f_N \|^2
  =   \sum^{K}_{\ell,m,n=-K}   \lambda_{\ell, m, n}^{12} | \hat{f}^N_{\ell,m,n}  |^2  ,
\end{aligned}
\label{prop elliptic regularity-1}
\end{equation}
with $ \lambda_{\ell, m, n}$ introduced in~\eqref{spectral-coll-4-b}. Meanwhile, the elliptic regularity for the continuous function $f_N$ indicates that
\begin{eqnarray}
    \| \nabla \Delta f_N \|  \le \hat{C}_0 \| \Delta^3 f_N \| ,  \quad
   \mbox{for some $\hat{C}_0$ only dependent on $\Omega$} .
   \label{prop elliptic regularity-2} 	
\end{eqnarray}
Finally, the discrete elliptic regularity inequality~\eqref{elliptic regularity-0} is a direct combination of~\eqref{prop elliptic regularity-1} and \eqref{prop elliptic regularity-2}. This completes the proof of Proposition~\ref{prop:elliptic regularity}.

	\bibliographystyle{plain}
	\bibliography{revision}

\end{document}